\RequirePackage{rotating}
\documentclass{amsart}

\usepackage[utf8]{inputenc}

\usepackage[centertags]{amsmath}
\usepackage{MnSymbol}
\usepackage{amsthm}
\usepackage[svgnames]{xcolor}

\usepackage{amscd}
\usepackage[abbreviations]{foreign}
\usepackage[all,cmtip,2cell]{xy}
\UseAllTwocells
\usepackage{xcolor}
\definecolor{darkblue}{rgb}{0,0,0.3}
\usepackage[ocgcolorlinks,colorlinks=true, citecolor=darkblue, filecolor=darkblue, linkcolor=darkblue, urlcolor=darkblue]{hyperref}
\usepackage{hyphenat}
\usepackage{enumerate, enumitem}
\usepackage{nameref}

\makeatletter
\let\orgdescriptionlabel\descriptionlabel
\renewcommand*{\descriptionlabel}[1]{%
  \let\orglabel\label
  \let\label\@gobble
  \phantomsection
  \edef\@currentlabel{#1}%
  \let\label\orglabel
  \orgdescriptionlabel{#1}%
}
\makeatother

\usepackage[all,2cell]{xy}
\UseAllTwocells
\input xy
\xyoption{2cell}
\xyoption{all}
\usepackage{verbatim}
\usepackage{eucal}
\usepackage{stackrel}

\usepackage{tikz}
\usetikzlibrary{cd}
\usetikzlibrary{calc}
\usetikzlibrary{math}
\usetikzlibrary{arrows}
\usetikzlibrary{fit}
\usetikzlibrary{positioning}

\usetikzlibrary{decorations.pathmorphing, arrows.meta}

\tikzset{Rightarrow/.style={double equal sign distance,>={Implies},->},
RightarrowDashed/.style={double equal sign distance, dashed, >={Implies},->},
triple/.style={-,preaction={draw,Rightarrow,}},
tripleDashed/.style={-,dashed,preaction={draw,RightarrowDashed}},
quadruple/.style={preaction={draw,Rightarrow,shorten >=0pt},shorten >=1pt,-,double,double
distance=0.2pt}
quadrupleDashed/.style={preaction={draw,RightarrowDahed,shorten >=0pt},shorten >=1pt,-,dashed,double,double distance=0.2pt}}

\usepackage{color}
\usepackage{eucal}

\usepackage{abstract}
\usepackage{changepage}

\newtheorem{thm}{Theorem}[section]
\newtheorem*{thmintro}{Theorem}
\newtheorem{cor}[thm]{Corollary}

\newtheorem{lem}[thm]{Lemma}
\newtheorem{prop}[thm]{Proposition}

\theoremstyle{definition}
\newtheorem{define}[thm]{Definition}
\newtheorem{notate}[thm]{Notation}
\newtheorem{assume}[thm]{Assumption}

\theoremstyle{remark}
\newtheorem{rem}[thm]{Remark}

\newtheorem{examples}[thm]{Examples}
\newtheorem{const}[thm]{Construction}

\newtheorem{warning}[thm]{Warning}

\newenvironment{nscenter}
{\parskip=0pt\par\nopagebreak\centering}
{\par\noindent\ignorespacesafterend}

\newcommand{\A}{\mathcal{A}}
\renewcommand{\plus}[1]{\mathop{\amalg}\limits_{#1}}
\newcommand{\An}{\mathrm{An}}

\newcommand{\thr}{\mathrm{th}}
\newcommand{\thb}{\mathbf{th}}

\newcommand{\Kk}{\mathcal{K}}
\newcommand{\Km}{\underline{\Kk}}
\newcommand\nbd\nobreakdash
\newcommand{\Aa}{\mathcal{A}}
\newcommand{\s}{\mathcal{S}\mspace{-2.mu}\mathrm{et}_{\Delta}}
\newcommand{\Ss}{\mathcal{S}\mspace{-2.mu}\mathrm{et}_{\Delta}^{\,\mathrm{sc}}}
\renewcommand{\th}{\mathcal{S}\mspace{-2.mu}\mathrm{et}_{\Theta_2}}

\newcommand{\ndef}{\emph}

\newcommand\pdfoo{\texorpdfstring{$\infty$}{oo}}
\newcommand\pdftwo{\texorpdfstring{$2$}{2}}

\newcommand{\Cat}{{\mathcal{C}\mspace{-2.mu}\mathit{at}}}
\newcommand{\nCat}[1]{{#1}\hbox{\protect\nbd-}\kern1pt\Cat}	
\newcommand{\RelCat}{\mathcal{R}\mspace{-1.mu}\Cat}
\newcommand{\msCat}{\mathcal{C}at^+_{\Del}}
\newcommand{\Fun}{\mathrm{Fun}}

\newcommand{\op}{\mathrm{op}}
\newcommand{\Dop}{\Delta^\op}
\newcommand{\Seg}{\mathcal{S}\mspace{-2.mu}\mathrm{eg}}
\newcommand{\inj}{\mathrm{inj}}
\newcommand{\proj}{\mathrm{proj}}
\newcommand{\kan}{\mathrm{kan}}

\newcommand{\cmp}{\mathrm{cmp}}

\newcommand{\C}{\mathcal{C}}
\newcommand{\D}{\mathcal{D}}

\newcommand{\E}{\mathcal{E}}
\newcommand{\W}{\mathcal{W}}
\newcommand{\M}{\mathcal{M}}
\newcommand{\N}{\mathcal{N}}

\newcommand{\B}{\mathcal{B}}
\newcommand{\I}{\mathcal{I}}

\newcommand{\cO}{\mathcal{O}}

\newcommand{\bS}{\mathbf{S}}
\newcommand{\rN}{\mathrm{N}}

\newcommand{\fC}{\mathfrak{C}}
\newcommand{\ho}{\mathrm{ho}}
\newcommand{\bho}{\mathbf{ho}}

\DeclareMathOperator{\im}{Im}
\DeclareMathOperator{\him}{hIm}
\DeclareMathOperator{\Iso}{Iso}
\DeclareMathOperator{\Hom}{Hom}

\DeclareMathOperator{\Map}{Map}

\DeclareMathOperator{\Ho}{Ho}

\DeclareMathOperator{\Ne}{N}

\DeclareMathOperator{\ad}{ad}

\DeclareMathOperator{\eq}{eq}

\DeclareMathOperator{\Set}{Set}

\DeclareMathOperator{\sca}{sc}

\DeclareMathOperator{\thi}{th}

\DeclareMathOperator{\St}{St}

\DeclareMathOperator{\Eq}{Eq}

\DeclareMathOperator{\Psh}{PSh}

\def\alp{{\alpha}}
\def\bet{{\beta}}

\def\eps{{\varepsilon}}

\def\sig{{\sigma}}
\def\vphi{{\varphi}}

\def\Del{{\Delta}}

\def\Lam{{\Lambda}}

\def\vphi{{\varphi}}

\def\hrar{\hookrightarrow}

\def\ovl{\overline}

\def\wtl{\widetilde}

\renewcommand{\twocell}[5]{\ar@<#4ex>@{}[#1] \ar@<#4ex>@{=>}?(#3)+/d #5cm/;?(#3)+/u #5cm/^{#2}}

\newcommand{\mgr}{\otimes}

\makeatletter
\renewcommand{\tocsection}[3]{%
  \indentlabel{\@ifnotempty{#2}{\bfseries\ignorespaces#1 #2\quad}}\bfseries#3} 

\renewcommand{\tocsubsection}[3]{%
  \indentlabel{\@ifnotempty{#2}{\hspace{1.6em}\ignorespaces#1 #2\quad}}#3}
\makeatother

\numberwithin{equation}{section}

\setlist[itemize]{leftmargin=*}
\setlist[enumerate]{leftmargin=*}

\title[On the equivalence of all models for (\pdfoo,\pdftwo)-categories]%
{On the equivalence of all models\\ for (\pdfoo,\pdftwo)-categories}

\author{Andrea Gagna}
\address{Universita Karlova\\ Matematicko-fyzik\'{a}ln\'{i} fakulta \\ Matematická sekce \\ Sokolovská 83 \\ 186 75 Praha 8 \\ Czech Republic}
\email{andrea.gagna@gmail.com}
\urladdr{https://sites.google.com/view/andreagagna/home}
\author{Yonatan Harpaz}
\address{Institut Galilée\\ Université Paris 13\\ 99 avenue Jean-Baptiste Clément\\ 93430 Villeta-neuse\\ France}
\email{harpaz@math.univ-paris13.fr}
\urladdr{https://www.math.univ-paris13.fr/~harpaz}
\author{Edoardo Lanari}
\address{Institute of Mathematics CAS \\ \v{Z}itn\'a 25 \\115 67   Praha 1\\ Czech Republic}
\email{edoardo.lanari.el@gmail.com}
\urladdr{edolana.github.io/}
\begin{document}
\input{pent}
\input{square2}

\maketitle
\begin{abstract}
	The goal of this paper is to provide the last equivalence needed in order to identify all known models for \((\infty,2)\)-categories. We do this by showing that Verity's model of saturated \(2\)-trivial complicial sets is equivalent to Lurie's model of \(\infty\)\nobreakdash-bicategories, 
	which, in turn, has been shown to be equivalent to all other known models for \((\infty,2)\)-categories. A key technical input is given by identifying the notion of \(\infty\)-bicategories with that of \emph{weak} \(\infty\)-bicategories, a step which allows us to understand Lurie's model structure in terms of Cisinski--Olschok's theory.
	Several of our arguments use tools coming from a new theory of \emph{outer (co)car\-te\-sian fibrations}, further developed in a companion paper.
	In the last part of the paper we construct a homotopically fully faithful scaled simplicial nerve functor for \(2\)-cat\-e\-gories, give two equivalent descriptions of it, and show that the homotopy \(2\)-category of an \(\infty\)\nbd-bi\-cat\-e\-gory retains enough information to detect thin \(2\)-simplices.
\end{abstract}
\tableofcontents

\section*{Introduction}
Nowadays, \(\infty\)-categories are widely recognized as an extremely important tool to develop homotopy coherent mathematics. These were first introduced by Boardman and Vogt~\cite{BoardmanVogt},
and later developed by Joyal and Lurie, in terms of simplicial sets admitting fillers for inner horns~\cite{JoyalQCatsAndKan,HTT}. Just as in order to develop category theory for ordinary categories one cannot avoid the framework of 2-categories (in particular, one has to consider the 2-category \(\mathbf{Cat}\)), \((\infty,1)\)-categories are best understood when their theory is framed inside an \((\infty,2)\)-category. 

There are currently many models for \((\infty,2)\)-categories, and almost all of them have been proven to be equivalent
to one another in the works of Lurie, Bergner--Rezk, Ara, Barwick--Schommer-Pries and others (\cite{LurieGoodwillie},~\cite{BergnerRezkInftynI},~\cite{BergnerRezkInftynII},~\cite{AraQCatvsRezk},~\cite{BarwickSchommerPriesUnicity}).
Yet, an important one is still left out: the complicial model. This has been developed by Verity, initially in its strict version motivated by proving the Roberts--Street's conjecture~\cite{VerityComplicial}, and later on weakened in \cite{VerityWeakComplicialI}, where a model structure for \emph{weak complicial sets} is introduced.
These are \emph{stratified sets}, \ie simplicial sets bearing some extra structure in the form of a marking on \(n\)-simplices for \(n>0\), which satisfy an extension property with respect to a set of anodyne morphisms (the analogues of inner horns for \(\infty\)-categories).
When we specialize to those which are \emph{saturated}, \ie those in which the marked \(n\)-simplices satisfy the \(2\)-out-of-\(6\) property,
and \emph{\(2\)-trivial}, \ie every \(n\)-simplex is marked when \(n>2\), 
we get a model for \((\infty,2)\)\nbd-cat\-e\-gories known as 2-complicial sets.
In \cite{OzornovaRovelliNComplicial}, a model structure whose fibrant objects are the saturated \(n\)-trivial complicial sets is constructed,
based on a general principle established in~\cite{VerityWeakComplicialI}. In this paper we will only consider the case where \(n=2\).

So far, the complicial model has not been proven to be equivalent to any other known one. This has resulted in an undesirable gap between the theory as developed by Verity and later taken up in the works of
Riehl~\cite{RiehlOuverture} and Ozornova--Rovelli~\cite{OzornovaRovelliNComplicial,OzornovaRovelli2Complicial},
and the theory developed in other settings. In this paper we close this gap by providing a Quillen equivalence between 2-complicial sets
and Lurie's bicategorical model structure on scaled simplicial sets~\cite{LurieGoodwillie},
whose fibrant objects are called \emph{\(\infty\)-bicategories} and are not very explicit.
We do so in two steps: first, we show that \(\infty\)-bicategories coincide with \emph{weak} \(\infty\)-bicategories, where the latter are defined as those scaled simplicial sets with the extension property with respect to a certain set of anodyne morphisms. Second, using the identification above we \emph{redefine} Lurie's model structure on
scaled simplicial sets by invoking the combinatorial machinery of Cisinski--Olschok. 
The equivalence between the two definitions allows us to prove the result we previously mentioned:
\begin{thmintro}[Theorem~\ref{Quillen eq}]
	There is a canonical Quillen equivalence 
	\[ \xymatrixcolsep{1pc}
		\vcenter{\hbox{\xymatrix{
				**[l]\Ss \xtwocell[r]{}_{U}^{\iota
				}{'\perp}& **[r] \St_{2}}}},
	\]
where \(\Ss\) is the category of scaled simplicial sets endowed with
	Lurie's bicategorical model structure and \(\St_{2}\) is the category of stratified simplicial sets endowed with the
	2-complicial model structure.
\end{thmintro}
We may then conclude that the model of 2-complicial sets is equivalent to all other models, as depicted in Figure~\ref{fig:models} below
\footnote{The term \((\mathcal{S}\mathrm{et}^+_{\Box})_2\) in Figure~\ref{fig:models} refers to the 2-trivial comical model structure on marked cubical sets~\cite{cubical}, which appeared around the same time as the first version of the present work, together with a Quillen adjunction relating it to $\St_2$. This Quillen adjunction was later verified to be a Quillen equivalence in~\cite{DKM-equivalence}.}
(with arrows indicating right Quillen equivalences;
our contribution then corresponds to the two dotted arrows). 

\begin{figure*}
\centering
\begin{adjustwidth*}{}{-5em} 
	\begin{tikzpicture}[scale=3.5]
	\node (cubical) at (-2,1.5) {$(\mathcal{S}\mathrm{et}^+_{\Box})_2$};
	\node (css) at (-3,-1) {$\Fun(\Delta^\op, \s)_{\cmp}$};
	\node (mcss) at (-3,0) {$\Fun(\Delta^\op, \s^+)_{\cmp}$};
	\node (over) at (-3,1) {$(\s^+)_{/N(\Dop)}$};
	\node (seginj) at (-2,-1) {$\Seg^\inj_{\s}$};
	\node (mseginj) at (-2,0) {$\Seg^\inj_{\s^+}$};
	\node (segproj) at (-1,-1) {$\Seg^\proj_{\s}$};
	\node (msegproj) at (-1,0) {$\Seg^\proj_{\s^+}$};
	\node (csscss) at (-3,-2.5) {$\Fun\bigl(\Dop, \Fun(\Dop, (\s)_{\kan})_{\cmp}\bigr)_\cmp$};
	\node (jcat) at (0,-1) {$\nCat{\s}$};
	\node (mcat) at (0,0) {$\nCat{\s^+}$};
	\node (bicat) at (0,1) {$(\Ss)_{\mathrm{-bicat}}$};
	\node (wbicat) at (0,1.5) {$(\Ss)_{\mathrm{w-bicat}}$};
	\node (segthproj) at (-1,-2.5) {$\Seg^{\proj}_{\Fun(\Dop,(\s)_{\kan})_{\cmp}}$};
	\node (segthinj) at (-2,-2) {$\Seg^{\inj}_{\Fun(\Dop,(\s)_{\kan})_{\cmp}}$};
	\node (csscat) at (0,-2) {$\nCat{\Fun(\Dop,(\s)_{\kan})_{\cmp}}$};
	\node (strat) at (-1,1.5) {$\St_2$};
	\node (cssfold) at (-3,-3) {$\Fun((\Delta^\op)^{\times 2},(\s)_{\kan})_{\cmp}$};
	\node (cssth) at (-1.5,-3) {$\Fun(\Theta_2^\op, (\s)_{\kan})_{\cmp}$};
	\node (th) at (0,-3) {$\th$};
	\draw[->,>=latex] (strat) to node [above]{\cite{DKM-equivalence}} (cubical);
	\draw[->, >=latex] (segthinj) to node [above] {\cite{BergnerRezkInftynI}} (segthproj);
	\draw[<->, >=latex] (css) to  node [left] {\cite{JoyalTierneyQcatVSSegal}} (csscss);
	\draw[->, >=latex] (css) to node [above] {\cite{LurieGoodwillie}} (seginj);
	\draw[<->, >=latex] (csscss) to node [left] {$=$} (cssfold);
	\draw[->, >=latex] (mcss) to node [left] {\cite{HTT}} (css);
	\draw[->, >=latex] (mcss) to node [left] {\cite{LurieGoodwillie}} (over);
	\draw[->, >=latex] (mcss) to node [above] {\cite{LurieGoodwillie}} (mseginj);
	\draw[->, >=latex] (seginj) to node [above] {\cite{LurieGoodwillie}} (segproj);
	\draw[->, >=latex] (mseginj) to node [above] {\cite{LurieGoodwillie}} (msegproj);
	\draw[->, >=latex] (mseginj) to node [right] {\cite{LurieGoodwillie}} (seginj);
	\draw[->, >=latex] (msegproj) to node [right] {\cite{LurieGoodwillie}} (segproj);
	\draw[->, >=latex] (jcat) to node [above] {\cite{LurieGoodwillie}} (segproj);
	\draw[->, >=latex] (mcat) to node [above] {\cite{LurieGoodwillie}} (msegproj);
	\draw[->, >=latex] (mcat) to node [right] {\cite{HTT}} (jcat);
	\draw[->, >=latex] (mcat) to node [right] {\cite{LurieGoodwillie}} (bicat);
	\draw[->, >=latex] (over) to node [above] {\cite{LurieGoodwillie}} (bicat);
	\draw[<->, >=latex] (jcat) to node [right] {\cite{JoyalTierneyQcatVSSegal}} (csscat);
	\draw[<->, >=latex] (segproj) to node [right] {\cite{JoyalTierneyQcatVSSegal}} (segthproj);
	\draw[<->, >=latex] (seginj) to node [right] {\cite{JoyalTierneyQcatVSSegal}} (segthinj);

	\draw[->, >=latex] (csscat) to node [above] {\cite{BergnerRezkInftynI}} (segthproj);
	\draw[->, >=latex] (cssfold) to 
	node [above] {\cite{BergnerRezkInftynII}} (cssth);
	\draw[->, >=latex] (csscss) to node [above] {\cite{BergnerRezkInftynII}} (segthinj);
	\draw[<->, >=latex] (th) to node [above] {\cite{AraQCatvsRezk}} (cssth);
	\draw[->, >=latex] (csscat) to node [right] {\cite{GindiRigidification}} (th);
	
	\draw[<->, dotted, >=latex] (bicat) to node [right] {\S\ref{sec:weak-is-strong}} (wbicat);
	\draw[->, dotted, >=latex] (strat) to node [above] {\S\ref{sec:equivalence}} (wbicat);
	\end{tikzpicture}
	\end{adjustwidth*}
	\caption{Models for \((\infty, 2)\)-categories.}
	\label{fig:models}
\end{figure*}
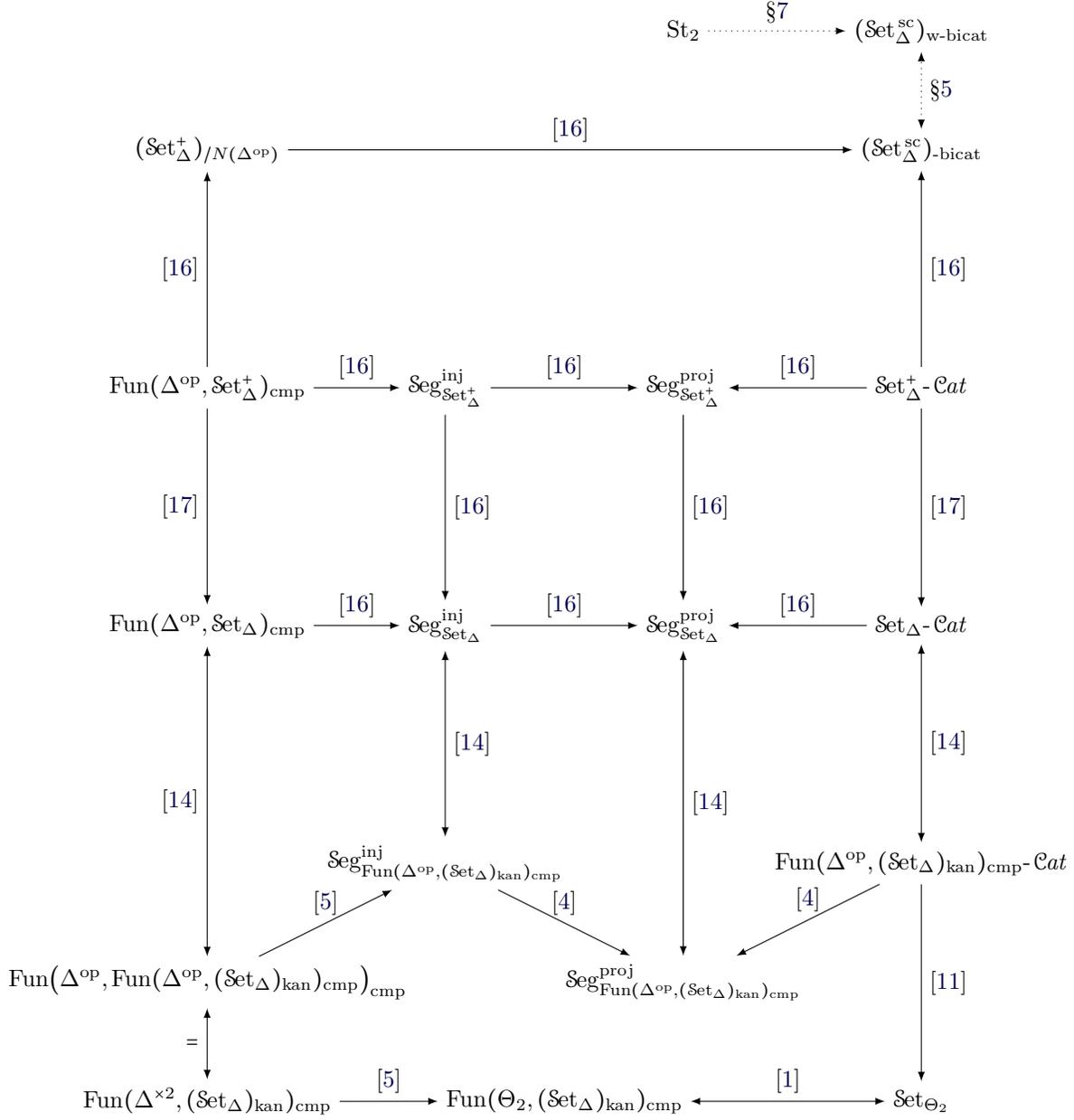

In the final part of the paper we construct a homotopy fully faithful nerve functor for \(2\)-categories, 
which embeds them (in the \(\infty\)-categorical sense) in the category of scaled simplicial sets as a particular class of \(\infty\)-bicategories.
Such a construction was already known for the 2-complicial model~\cite{OzornovaRovelli2Complicial} and for the cellular model of \(\Theta_2\)-sets~\cite{CampbellCellular}. We then show that the homotopy \(2\)-category of an \(\infty\)-bicategory has a conservativity property: a \(2\)\nbd-sim\-plex in \(X\) is thin if and only if it represents an invertible \(2\)-cell in~\(\ho_2(X)\).

Several of our arguments make use of a new theory of \emph{outer (co)cartesian fibrations}. 
This notion, which we introduce in this paper, is developed in greater detail in the companion paper \cite{GagnaHarpazLanariLaxLimits}, 
where we also develop a theory of 
\emph{lax limits} in the setting of scaled simplicial sets.

This paper is organized as follows.
In \S\ref{sec:prelim} we review all the necessary definitions and preliminary results concerning the model structures involved in what follows, including the one on scaled simplicial sets and that of \(n\)-trivial saturated complicial sets.

In \S\ref{sec:fibrations} we introduce the notion of outer (co)cartesian fibrations and develop some of its basic properties. We then generalize the join and slice constructions to the setting of scaled simplicial sets, yielding in particular a model for the hom-\(\infty\)\nbd-cat\-egories of an \(\infty\)-bicategory. 

In \S\ref{sec:thin} and \S\ref{sec:moving} we do some preparatory work for the subsequent proofs by developing a few technical tools and proving some key results concerning thin triangles in weak \(\infty\)-bicategories.
  
In \S\ref{sec:weak-is-strong} we prove one of the main results of this paper, namely that weak \(\infty\)\nbd-bi\-cat\-egories and \(\infty\)-bicategories coincide. In particular, this implies that the fibrant objects in Lurie's model structure can be detected by means of a generating set of anodyne morphisms. This part uses the technology of outer cartesian fibrations.

In \S\ref{sec:cis} we define the model structure for weak \(\infty\)-bicategories using the machinery developed in \cite{CisinskiPrefaisceaux} and \cite{OlschokLeft} and recorded in the appendix. Using the results of \S\ref{sec:weak-is-strong} we then prove that the fibrant objects are exactly Lurie's \(\infty\)\nbd-bi\-categories. 

In \S\ref{sec:equivalence} we construct a Quillen pair between the model structure for \(\infty\)\nbd-bi\-cat\-egories and that of \(2\)-trivial saturated complicial sets, and we show in Theorem~\ref{Quillen eq} that it is a Quillen equivalence. This is achieved by producing an explicit model for a fibrant replacement of the image of a weak \(\infty\)-bicategory under the left adjoint.

Finally, in \S\ref{sec:2-nerve} we construct a scaled simplicial nerve for \(2\)-categories in two different ways, and prove they coincide in Proposition~\ref{nerves coincide}. Then, in Proposition~\ref{nerve is fully faithful}, we show that this nerve is homotopy fully faithful.
The left adjoint of this scaled \(2\)-nerve gives a construction of the \ndef{homotopy \(2\)-category} associated to an \(\infty\)\nbd-bi\-category. We then prove a conservativity property by showing that the homotopy \(2\)-category detects thin triangles.

\section*{Acknowledgements}
The first author is supported by GA\v{C}R EXPRO 19-28628X.
The third author gratefully acknowledges the support of Praemium Academiae of M.~Markl and RVO:67985840. The authors thank H.~Gindi for the diagram of figure~\ref{fig:models}
and the anonymous reviewer for the precise and detailed recommendations.

\section*{Notation}
We will denote by \(\s\) the category of simplicial sets.
We will employ the standard notation \(\Del^n \in \s\) for the \(n\)-simplex,
and for \(\emptyset \neq S \subseteq [n]\) we write \(\Del^S \subseteq \Del^n\) the \((|S|-1)\)-dimensional face of \(\Del^n\) whose set of vertices is \(S\). For \(0 \leq i \leq n\) we will denote by \(\Lam^n_i \subseteq \Del^n\) the \(i\)'th horn in \(\Del^n\), that is, the subsimplicial set spanned by all \((n-1)\)-dimensional faces containing the \(i\)-th vertex. 
By an \(\infty\)\nbd-ca\-tegory we will always mean a \emph{quasi-category}, \ie a simplicial set \(X\) which admits extensions for all inclusions 
\(\Lambda^n_i\rightarrow\Delta^n\), for all \(n > 1\) and all \(0 < i < n\) (known as \emph{inner} horns). Given an \(\infty\)-category \(X\), we will denote its homotopy category by \(\ho(X)\). This is the ordinary category having as objects the 0-simplices of \(X\), and as morphisms \(x \rightarrow y\) the set of equivalence classes of 1-simplices \(f\colon x \rightarrow y\) of \(X\) under the equivalence relation generated by identifying \(f\) and \(f'\) if there is a 2-simplex \(H\) in \(X\) with \( H|_{\Del^{\{1,2\}}}=f, \ H|_{\Del^{\{0,2\}}}=f'\) and \(H|_{\Del^{\{0,1\}}}\) degenerate on~\(x\). We recall that the functor \(\ho\colon \nCat{\infty}\rightarrow \nCat{1}\) is left adjoint to the ordinary nerve functor.

\section{Preliminaries}\label{sec:prelim}

In this section we will review the main definitions and basic results concerning marked simplicial sets, marked-simplicial categories, scaled simplicial sets and stratified sets. In particular, we will recall the relevant model category structures and some of the Quillen adjunctions relating them. 

\subsection{Marked simplicial sets and marked-simplicial categories}

\begin{define}
	A \emph{marked simplicial set} is a pair \((X,E_X)\) where \(X\) is  simplicial set
	and \(E_X\) is a subset of the set of 1-simplices of \(X\), called \emph{marked} simplices, such that it contains the degenerate ones. A map of marked simplicial sets \(f\colon (X,E_X)\rightarrow (Y,E_Y)\) is a map of simplicial sets \(f\colon X \rightarrow Y\) satisfying \(f(E_X)\subseteq E_Y\).
\end{define}

The category of marked simplicial sets will be denoted by \(\s^+\). 

\begin{notate}\label{no:marked}
For simplicity, we will often speak only of the non-degenerate marked edges when considering a marked simplicial set. For example, if \(X\) is a simplicial set and \(E_X\) is any set of edges in \(X\) then we will denote by \((X, E_X)\) the marked simplicial set whose underlying simplicial set is \(X\) and whose marked edges are \(E_X\) together with the degenerate edges.  In addition, when there is no risk of ambiguity we will omit the set of marked 1-simplices and just denote \((X, E_X)\) by \(X\).\end{notate}

\begin{rem}
 \label{rem:marked_sset_presentable}
	The category \(\s^+\) of marked simplicial sets
	admits an alternative description, as the category of models of a limit sketch.
	Indeed, denote by \(\Delta^+\) the category whose objects and arrows are
	those of \(\Delta\), to which we add an object \([1]_t\)
	and arrows \(\varphi \colon [1] \to [1]_t\) and \(\zeta \colon [1]_t \to [0]\)
	together with the relation \(\zeta\varphi = \sigma^0\).
	Then \(\s^+\) it is the reflective localization of the presheaf category on
	\(\Delta^+\) at the morphism \([1]_t \coprod_{[1]} [1]_t \to [1]_t\),
	where we identify \(\Delta^+\) with its image under the Yoneda embedding,
	so that	the marked simplicial sets are precisely the presheaves \(X\) on \(\Delta^+\)
	for which the map \(X(\varphi) \colon X([1]_t) \to X([1])\) is injective.
	In particular, \(\s^+\) is a cartesian closed category.
\end{rem}

\begin{define}
 \label{def:marked_core}
    Given a marked simplicial set \((X, E_X)\) we define its \ndef{marked core}
    as the sub-simplicial set of \(X\) spanned by those \(n\)-simplices
    whose \(1\)-dimensional faces are marked, \ie they belong to \(E_X\).
    We will denote by \(\kappa (X, E_X)\) the marked core of \((X, E_X)\).
\end{define}

\begin{thm}[\cite{HTT}]\label{t:marked-categorical}
	There exists a model category structure on the category \(\s^+\) of marked simplicial sets in which cofibrations are exactly the monomorphisms and the fibrant objects are marked simplicial sets \((X, E_X)\) in which \(X\) is an \(\infty\)-category and \(E_X\) is the set of equivalences of \(X\), \ie \(1\)-simplices \(f\colon \Delta^1 \rightarrow X\) which are invertible in \(\ho(X)\). 
\end{thm}

This is a special case of Proposition 3.1.3.7 in~\cite{HTT}, when \(S=\Delta^0\),
combined with the explicit description of the fibrant objects given
in Proposition 3.1.4.1 of~\opcit. We will refer to the model structure of Theorem~\ref{t:marked-categorical} as the \ndef{marked categorical model structure}, and its weak equivalences as \ndef{marked categorical equivalences}.

\begin{rem}
	Part (A0) of Proposition 3.1.5.1 of~\cite{HTT} shows that
	marked simplicial sets are a model for \((\infty,1)\)-categories. 
\end{rem}

Recall that a \ndef{relative category} is a pair \((\C, \W)\), where \(\C\) is a category and \(\W\) is a subcategory of \(\C\), called the subcategory of \ndef{weak equivalences in~\(\C\)}, containing all the objects of~\(\C\). We denote by \(\RelCat\) the category of small relative categories having as morphisms the functors which preserve weak equivalences.

\begin{define}\label{d:marked-nerve}    
    We define the \ndef{marked nerve}
    \[\rN^+ \colon \RelCat \to \s^+\,,\] 
    to be the functor which sends a relative category
    \((\C, \W)\) to the marked simplicial set \((\rN(\C), \mathrm{Arr}(\W))\), where
    \(\rN(\C)\) is the standard nerve of the small category \(\C\) and
    the marking \(\mathrm{Arr}(\W)\) consists of those edges of \(\rN(\C)\) which are contained in \(\rN(\W)\).
    The marked nerve functor admits a left adjoint 
    \[\overline{\ho} \colon \s^+ \to \RelCat\]
    which can be explicitly described as follows: a marked simplicial set \((X, E_X)\) is mapped to the
    relative category \((\ho(X), \him(X, E_X))\), where \(\ho \colon \s \to \Cat\)
    is the standard left adjoint to the nerve functor and
    \(\him \colon \s^+ \to \Cat\) is defined by mapping \((X, E_X)\) to
    the smallest subcategory containing the image of the functor
    \(\ho\,\kappa(X, E) \to \ho(X)\). 
\end{define}

\begin{rem}\label{r:preserves-products}
	The image of a functor is a functorial assignment and defines a functor
    $\Cat^{[1]} \to \Cat$ preserving finite products.
    Notice that also the functors $\ho \colon \s \to \Cat$
    and $\kappa \colon \s^+ \to \s$ both
    preserve finite products. Thus, the functor $\him \colon \s^+ \to \Cat$
    preserves finite products. 
    For any two marked simplicial sets
    \((X, E_X)\) and \((Y, E_Y)\) we have
    \begin{align*}
     \overline{\ho}\bigl((X, E_X) \times (Y, E_Y)\bigr) & \cong
     \overline{\ho}(X\times Y, E_X \times E_Y) \\ & \cong
     \bigl(\ho(X\times Y), \him (X\times Y, E_X \times E_Y)\bigr) \\ & \cong
     \bigl(\ho(X) \times \ho(Y), \him(X, E_X) \times \him(Y, E_Y)\bigr) \\ & \cong
     \bigl(\ho(X), \him(X, E_X)\bigr) \times \bigl(\ho(Y), \him(Y, E_Y)\bigr) \\ &\cong
     \overline{\ho}(X, E_X) \times \overline{\ho}(Y, E_Y)\,;
    \end{align*}
    hence, the functor \(\overline{\ho} \colon \s^+ \to \RelCat\) preserves finite products.
\end{rem}

    The canonical functor \(\iota \colon \Cat \to \RelCat\), mapping
    a small category \(\C\) to the relative category \((\C, \Iso(\C))\)
    having as weak equivalences the set of isomorphisms of \(\C\),
    has a left adjoint \(L \colon \RelCat \to \Cat\) which is the localization
    functor mapping a small relative category \((\C, \W)\) to the small category
    \(\C[\W^{-1}]\), that is to say, the category in which we have formally inverted
    the arrows of \(\W\). It is well-known that the functor \(L\) preserves finite products. Composing the adjunctions \(L \dashv \iota\) and \(\ovl{\ho} \dashv \rN^+\) we now obtain an adjunction
\begin{equation}\label{e:Lho}
\xymatrixcolsep{1pc}
\vcenter{\hbox{\xymatrix{
			**[l]\s^+ \xtwocell[r]{}_{\rN^+\iota}^{L\ovl{\ho}
			}{'\perp}& **[r] \Cat}}}
\end{equation}
in which the left adjoint \(L\ovl{\ho}\) preserves finite products.

Recall that the category \(\Cat\) carries the \emph{canonical model structure} in which the weak equivalences are the categorical equivalences, the fibrations are the isofibrations and the cofibrations are the functors which are injective on objects.

\begin{lem}\label{l:quillen}
The adjunction~\eqref{e:Lho} is a Quillen adjunction. In particular, the functor \(L\ovl{\ho}\) preserves weak equivalences (since all objects in \(\s^+\) are cofibrant).
\end{lem}
\begin{proof}
Since the objects of \(L\ovl{\ho}(X,E_X)\) are the vertices of \(X\) the functor \(L\ovl{\ho}\) clearly preserves cofibrations. Now let \(f\colon (X,E_X) \to (Y,E_Y)\) be a trivial cofibration of marked simplicial sets, so that \(L\ovl{\ho}(f) \colon L\ovl{\ho}(X,E_X) \to L\ovl{\ho}(Y,E_Y)\) is a cofibration. In order to prove that \(L\ovl{\ho}(f)\) is also an equivalence of categories it will suffice to show that for every category \(\C\) the induced map
\[\Fun(L\ovl{\ho}(Y,E_Y),\C) \to \Fun(L\ovl{\ho}(X,E_X),\C)\]
is trivial fibration of categories. Replacing \(f\) by its pushout-products with \(\partial\Del^1 \to \Del^1\) and \(\partial \Del^2 \to \Del^2\) and using Remark~\ref{r:preserves-products} we may reduce to showing that the induced map \(\Fun(L\ovl{\ho}(Y,E_Y),\C) \to \Fun(L\ovl{\ho}(X,E_X),\C)\) is surjective on objects, that is, every functor \(L\ovl{\ho}(X,E_X) \to \C\) extends to \(L\ovl{\ho}(Y,E_Y)\). Finally, since \(f\) is a trivial cofibration it will suffice to check that \(\rN^+(\iota\C) = (\rN(\C),\Iso(\C))\) is fibrant. Indeed, \(\rN(\C)\) is an \(\infty\)-category whose equivalences are exactly \(\Iso(\C)\).
\end{proof}

Since \(\s^+\) is a model for \(\infty\)-categories, using enrichment in marked simplicial sets one can form a model for the theory of \((\infty,2)\)-categories.

\begin{define}
	We let \(\msCat\) denote the category of categories enriched over marked simplicial sets. We will refer to these as \emph{marked-simplicial categories}.
\end{define}

By virtue of Proposition A.3.2.4 and Theorem A.3.2.24 of \cite{HTT}, the category \(\msCat\) is endowed with a model category structure in which the fibrant objects are
categories enriched over marked simplicial sets which are fibrant for the
marked categorical model structure and the weak equivalences
are the \ndef{Dwyer--Kan equivalences}. To define these, let us 
denote for a marked-simplicial category \(\E\) by \(\bho(\E)\) the category whose objects are the objects of \(\E\) and such that 
\[\Hom_{\bho(\E)}(x,y) := \Hom_{\Ho(\s^+)}(\Del^0,\Map_{\C}(x,y))\] 
is the set of morphisms in the homotopy category of the model category \(\s^+\) 
from \(\Del^0\) to \(\Map_{\C}(x,y)\). We point out that this does not require any special adaptation for $\E$ non-fibrant 
since all products in \(\s^+\) are homotopy products and hence the functor \(\Hom_{\Ho(\s^+)}(\Del^0,-)\colon \s^+ \to \Ho(\s^+)\) is product preserving.
The Dwyer-Kan equivalences are then defined to be the maps \(f\colon \C \to \D\) which are
\begin{itemize}
\item
\emph{fully-faithful:} in the sense that the maps
\(f_*\colon\Map_{\C}(x,y)\rightarrow \Map_{\D}(f(x),f(y))\)
are marked categorical equivalences for all pairs of vertices \((x, y)\) of \(\C\);
\item
\emph{essentially surjective:} in the sense that the functor of ordinary 
categories given by \(f_*\colon \bho(\C)\rightarrow\bho(\D)\) is essentially surjective.
\end{itemize}
We also note that the trivial fibrations in \(\msCat\) are the maps \(f\colon\C \rightarrow \D\) which are surjective on objects and such that \(f_*\colon\Map_{\C}(x,y)\rightarrow \Map_{\D}(f(x),f(y))\) is a trivial fibration of marked simplicial sets for every pair of objects \((x,y)\) in \(\C\). 

Now since the left adjoint in the adjunction~\eqref{e:Lho} above preserves products it induces in particular an adjunction on the level of enriched categories, which we will denote by
\begin{equation}\label{e:ho-ne} 
\xymatrixcolsep{1pc}
\vcenter{\hbox{\xymatrix{
			**[l]\msCat \xtwocell[r]{}_{\rN_*}^{\ho_*
			}{'\perp}& **[r] \nCat{2} }}}\, .
\end{equation}

In~\cite{LackModel2Cat,LackModelBicat} Lack builds a model structure on the category \(\nCat{2}\)
of \(2\)-categories, in which the weak equivalences are the \(2\)-categorical equivalences and the fibrations are the equiv-fibrations (that is, 2-functors which are isofibrations on each mapping category and admit lifts for invertible \(1\)-morphisms in the base). Furthermore, the trivial fibrations in the Lack model structure are the 2-functors which are surjective on objects and induce trivial fibrations on the level of mapping categories.
We note that \(2\)-categorical equivalences can be described in a way analogous to Dwyer--Kan equivalences. 
For this, let us 
denote for a 2-category $\E$ by \(\bho_1(\E)\) the category whose objects are the objects of \(\E\) and such that 
\[\Hom_{\bho_1(\E)}(x,y) := \Hom_{\Ho(\Cat)}(\ast,\Map_{\C}(x,y))\] 
is the set of morphisms in the homotopy category of the model category $\Cat$ (with respect to the canonical model structure) 
from \(\ast\) to \(\Map_{\C}(x,y)\), where we note that the projection to the homotopy category is product preserving since all objects in $\Cat$ are fibrant.
A \(2\)-functor \(f\colon \C \to \D\) between \(2\)-categories is then a \(2\)-categorical equivalence if and only if it is:
\begin{itemize}
\item
\emph{fully-faithful:} in the sense that the maps \(f_*\colon\Map_{\C}(x,y)\rightarrow \Map_{\D}(f(x),f(y))\) are categorical equivalences for all pairs of objects \((x, y)\) of \(\C\); 
\item
\emph{bi-essentially surjective:} in the sense that the functor of ordinary 
categories given by \(f_*\colon \bho_1(\C)\rightarrow\bho_1(\D)\) is essentially surjective. 
This is equivalent to saying that for every object \(d\) in \(\D\),
we can find an equivalence \(d' \to d\) of \(\D\) and an object \(c\) of \(\C\) such that
\(f(c) = d'\) (cf.~also paragraph~2.2 and Theorem~2.5 of~\cite{BergerMoerdijkHomotopyEnriched}).
\end{itemize}

\begin{prop}\label{p:quillen-2}
The adjunction~\eqref{e:ho-ne} is a Quillen adjunction. 
Furthermore, the functor \(\ho_*\) preserves weak equivalences.
\end{prop}

\begin{proof}
We first note that by Lemma~\ref{l:quillen} and the analogous description of trivial fibrations on both sides we have that \(\rN_*\) preserves trivial fibrations, and so \(\ho_*\) preserves cofibrations. We will now show that \(\ho_*\) preserves weak equivalences (and hence in particular trivial cofibrations). First since \(L\ovl{\ho}\) preserves weak equivalences (Lemma~\ref{l:quillen}) we have that \(\ho_*\) preserves fully-faithful functors. To finish the proof it will hence suffice to show that \(\ho_*\) preserves essentially surjective functors. For this, let \(f\colon \C \to \D\) be an essentially surjective functors, that is, a functor such that \(f_*\colon \bho(\C) \to \bho(\D)\) is essentially surjective. 
The left derived functor \(\Ho(\s^+) \to \Ho(\Cat)\) of the left Quillen functor \(L\ovl{\ho}\), together with the (unique) isomorphism \(L\ovl{\ho}(\Del_0) \cong \ast\) in $\Ho(\Cat)$, determine a natural transformation
\(\bho(-) \Rightarrow \bho_1(\ho_*(-))\) which is entry-wise bijective on objects, and hence 
a commutative square of ordinary categories
\[ \xymatrix{
\bho(\C) \ar[r]\ar[d] & \bho_1(\ho_*(\C)) \ar[d] \\
\bho(\D) \ar[r] & \bho_1(\ho_*(\D)) \\
}\]
in which the horizontal maps are bijective on objects. In particular, given an object \(x\) in \(\bho_1(\ho_*(\D))\),
we may lift it to an object \(x' \in \bho(\D)\). Since the left vertical map is assumed essentially surjective there exists a \(y' \in \bho(\C)\) and an isomorphism \(\alp\colon f(y') \overset{\cong}{\longrightarrow} x\) in \(\bho(\D)\). The images of \(y'\) and \(\alp\) on the right hand side then show that \(x\) is in the essential image of the right vertical map, as desired.
\end{proof}

\subsection{Scaled simplicial sets and \pdfoo-bicategories}

We now introduce scaled simplicial sets, which form another model for the theory of (\pdfoo,\pdftwo)-categories.
Indeed, they were introduced Lurie in~\cite{LurieGoodwillie} and shown in~\cite[Theorem~4.2.7]{LurieGoodwillie}
to support a model category structure
Quillen equivalent to that of marked-scaled categories introduced in the previous section.

\begin{define}[{\cite[\S 3.1]{LurieGoodwillie}}]
	A \emph{scaled simplicial set} is a pair \((X,T_X)\) where \(X\) is  simplicial set and \(T_X\) is a subset of the set of 2-simplices of \(X\), called the subset of \emph{thin} simplices, containing the degenerate ones. A map of scaled simplicial sets \(f\colon (X,T_X)\rightarrow (Y,T_Y)\) is a map of simplicial sets \(f\colon X \rightarrow Y\) satisfying \(f(T_X)\subset T_Y\).
\end{define}

The category of scaled simplicial sets will be denoted by \(\Ss\). 

\begin{notate}\label{d:flat-sharp}
	Given a simplicial set \(X\) we will denote by \(X_{\flat} = (X,\deg_2(X))\) the scaled simplicial consisting of \(X\)
	with only degenerate triangles \(\deg_2(X)\) as thin \(2\)-simplices, and by \(X_{\sharp} = (X,X_2)\) the scaled simplicial set consisting of \(X\) with all triangles thin.
\end{notate}

\begin{rem}\label{alternative def}
	The category \(\Ss\) admits an alternative description, as the category of models of a limit sketch, similarly to marked simplicial sets in Remark~\ref{rem:marked_sset_presentable}.
	In particular, it is a reflective localization of a presheaf category.
	In fact, we can define a category \(\Delta_{\sca}\) having as set of objects the set
	\(\{[k]\}_{k\geq 0}\cup \{[2]_t\}\), obtained from \(\Delta\) by adding
	an extra object and maps \(\varphi \colon [2]\rightarrow[2]_t\) and
	\(\zeta^i \colon  [2]_t\rightarrow [1]\) for \(i=0,1\) satisfying
	the relations \(\zeta^i\varphi = \sigma^i\), for \(i=0, 1\).
	The category \(\Ss\) is then the reflective localization
	of the category of presheaves on \(\Delta_{\sca}\) 
	at the morphism \([2]_t\coprod_{[2]} [2]_t \rightarrow [2]_t\),
	where we have identified an object of \(\Delta_{\sca}\) with its corresponding representable presheaf. 
	Equivalently, the local objects are those presheaves \(X\colon\Delta_{\sca}^{\mathrm{op}}\rightarrow\mathbf{Set}\) for which \(X([2]_t)\rightarrow X([2])\) is a monomorphism.
	In particular, the category \(\Ss\) is cartesian closed and it is easy to check
	that the reflector functor \(\Psh(\Delta_{\sca}) \to \Ss\) preserves monomorphisms and commutes with finite products.
\end{rem}

\begin{rem}
 \label{rem:monos_scaled}
	The class of monomorphisms of scaled simplicial sets corresponds to the class of maps for
	which the underlying map of simplicial sets is a monomorphism.
	This class is the cellular completion
	of the set of minimally marked boundaries and of the thin triangle inclusion:
	\[
		\bigl\{\partial\Delta^n_\flat \to \Delta^n_\flat : n \geq 0\bigr\} \cup
		\bigl\{\Delta^2_\flat \to \Delta^2_\sharp\bigr\}
	\]
\end{rem}

\begin{notate}\label{no:scaled}
For simplicity, we will often speak only of the non-degenerate thin 2-simplices when considering a scaled simplicial set. For example, if \(X\) is a simplicial set and \(T\) is any set of triangles in \(X\) then we will denote by \((X,T)\) the scaled simplicial set whose underlying simplicial set is \(X\) and whose thin triangles are \(T\) together with the degenerate triangles. If \(L \subseteq K\) is a subsimplicial set then we use \(T|_L := T \cap L_2\) to denote the set of triangles in \(L\) whose image in \(K\) is contained in \(T\). 
\end{notate}

\begin{define}[{\cite[Definition~3.1.3]{LurieGoodwillie}}]
	\label{anodyne defi}
	Let \(\bS\) be the set of maps of scaled simplicial sets consisting of:
	\begin{description}
		\item[An1\label{item:scaled_anodyne_i}]  scaled inner horns inclusions \(\bigl(\Lam^n_i,\{\Del^{\{i-1,i,i+1\}}\}|_{\Lam^n_i}\bigr)
		\subseteq \bigl(\Del^n,\{\Del^{\{i-1,i,i+1\}}\}\bigr)\) 
		for \(n \geq 2\) and \(0 < i < n\).		
		\item[An2\label{item:scaled_anodyne_ii}]  the map
		\( (\Delta^4,T)\rightarrow \bigl(\Delta^4,T\cup \{\Delta^{\{0,3,4\}}, \ \Delta^{\{0,1,4\}}\}\bigr)\), where
		\[T:=\bigl\{\Delta^{\{0,2,4\}}, \ \Delta^{\{ 1,2,3\}}, \ \Delta^{\{0,1,3\}}, \ \Delta^{\{1,3,4\}}, \ \Delta^{\{0,1,2\}}\bigr\};\]
		\item[An3\label{item:scaled_anodyne_iii}]  the set of maps
		\[ \Bigl(\Lambda^n_0\plus{\Delta^{\{0,1\}}}\Delta^0,\{\Delta^{\{0,1,n\}}\}|_{\Lam^n_0 \coprod_{\Delta^{\{0, 1\}}} \Delta^0}\Bigr)
		\rightarrow \Bigl(\Delta^n\plus{\Delta^{\{0,1\}}}\Delta^0,\{\Delta^{\{0,1,n\}}\}\Bigr)\]
		for \(n\geq 3\).
	\end{description}
\end{define}

We call \(\bS\) the set of \ndef{generating anodyne morphisms}, and its weak saturation is the class of \ndef{(scaled) anodyne maps}.

\begin{rem}
\label{rmk:j's are anod}
As observed in Remark 3.1.4 of \cite{LurieGoodwillie}, the inclusions of scaled simplicial sets \(j_i\colon (\Del^3,T_i)\to \Del^3_{\sharp}\), for \(i=1,2\), where \(T_i\) is the collection of 2-simplices of \(\Del^3\) containing the \(i\)'th vertex, 
are both anodyne. 
\end{rem}

\begin{define}\label{d:scaled-fib}
	We will say that a map of scaled simplicial sets \(X \to Y\) is a \emph{scaled fibration} if it has the right lifting property with respect to scaled anodyne maps.
\end{define}

The scaled fibrations whose codomain is a point are of special interest:
\begin{define}[{\cite[Definition~4.1.1]{LurieGoodwillie}}]\label{d:weak-bicat}
	A \emph{weak \(\infty\)-bicategory} is a scaled simplicial set \(\C\) which admits extensions along all scaled anodyne maps. 
\end{define}

\begin{rem}\label{r:kerodon}
In the recent online project~\cite[\href{https://kerodon.net/tag/01W9}{Tag 01W9}]{LurieKerodon} Lurie works with a variant of Definition~\ref{d:weak-bicat}, which he refers to simply as \((\infty,2)\)-categories. These are a-priori just simplicial sets, but have an internally determined collection of thin triangles defined via an extension property with respect to inner horns (see \cite[\href{https://kerodon.net/tag/00AD}{Tag 00AD}]{LurieKerodon}). In particular, any \((\infty,2)\)-category in the above sense determined a scaled simplicial set \(X\) which satisfies~\ref{item:scaled_anodyne_i} by construction and both~\ref{item:scaled_anodyne_iii} and its opposite by definition. 
Applying~\cite[\href{https://kerodon.net/tag/01XL}{Tag 01XL}]{LurieKerodon} together with \cite[\href{https://kerodon.net/tag/01Y0}{Tag 01Y0}]{LurieKerodon} one also concludes that \(X^{\op}\) satisfies~\ref{item:scaled_anodyne_ii}, so that \(X^{\op}\) is a weak \(\infty\)-bicategory. The same then holds for \(X\) by Corollary~\ref{weak are strong} of the present paper. 
On the other hand, any weak \(\infty\)-category \(X\) arises in this manner from an \((\infty,2)\)-category in the sense of ~\cite[\href{https://kerodon.net/tag/01W9}{Tag 01W9}]{LurieKerodon}. Indeed, since the condition of being a weak \(\infty\)-bicategory is \(\op\)-invariant by Corollary~\ref{weak are strong}, this amounts to verifying that any triangle in a weak \(\infty\)-bicategory which satisfies the property in \cite[\href{https://kerodon.net/tag/00AD}{Tag 00AD}]{LurieKerodon} is thin. This is true for \(\infty\)-bicategories which are scaled nerves of fibrant marked-simplicial categories by~\cite[Lemma 3.2.6]{GagnaHarpazLanari2Fib} (applied with \(\B=\ast\)), and hence for all \(\infty\)-bicategories by~\cite[Lemma 2.0.22]{GagnaHarpazLanari2Fib}.
\end{rem}

\begin{define}
	\label{def:core}
	The \emph{core} of a weak \(\infty\)-bicategory \(X\)
	is the sub-scaled simplicial set of \(X\) 
	spanned by those \(n\)-simplices whose 2-dimensional faces are thin.
	We shall denote by \(X^{\thi}\) the core of \(X\) to emphasize that
	we are considering the subobject ``spanned by thin triangles''.
\end{define}

\begin{warning}
	In \cite[\href{https://kerodon.net/tag/01XA}{Tag 01XA}]{LurieKerodon}, Lurie uses the term \emph{pith} in place of core, and denotes it by \(\operatorname{Pith}(\operatorname{\mathcal{C}})\).
\end{warning}

\begin{rem}
\label{rem:core-oo-cat}
We observe that the set of maps~\ref{item:scaled_anodyne_ii} ensures
that marked 2-simplices of weak \(\infty\)-bicategories satisfy a \emph{saturation} property.
The first set~\ref{item:scaled_anodyne_i} guarantees, among other things, that the core of an \(\infty\)-bicategory \(X\) is an \(\infty\)-category.
\end{rem}

\begin{define}
\label{d: equivalences}
	Let \(\C\) be a weak \(\infty\)-bicategory. We will say that an edge \(e\colon x \to y\) in \(\C\) is \emph{invertible}
	if it is invertible when considered in the core \(\infty\)-category \(\C^{\thi}\) of \(\C\), that is,
	if the corresponding arrow in the homotopy category \(\ho(\C^{\thi})\) is an isomorphism.
	In this case we will also say that \(e\colon x \to y\) is \emph{an equivalence} in \(\C\).
	
	If \(X\) is an arbitrary scaled simplicial set then we say that an edge in \(X\) is invertible if its image in \(\C\) is invertible for every scaled anodyne \(X \hrar \C\) with \(\C\) a weak \(\infty\)-bicategory.
\end{define}

\begin{rem}
	Notice that the notion of invertible edge for a scaled simplicial set is well-defined,
	since it does not depend on the choice of the weak \(\infty\)-bicategory \(\C\) and of the
	scaled anodyne map \(X \hrar \C\).
\end{rem}

\begin{rem}
\label{rem:eq}
	More explicitly, if \(\C\) is a weak \(\infty\)-bicategory then \(e\colon x \to y\) is invertible in \(\C\) if and only if there exist triangles of the form
	\[\begin{tikzcd}[column sep=small]
	x\ar[rr,equal]\ar[dr,"e"{swap}] & {}& x & & y\ar[rr,equal]\ar[dr,"f"{swap}] &{} & y \\
	&y \ar[u,phantom, "\simeq"]\ar[ur,"{f'}"{swap}]&&&& x \ar[u,phantom, "\simeq"] \ar[ur,"e"{swap}]
	\end{tikzcd}\]
	Indeed, this condition is equivalent to saying that \(e\) (or better, its homotopy class)
	is an isomorphism in \(\ho(\C^{\thi})\) (see~\cite[\href{https://kerodon.net/tag/004Z}{Tag~004Z}]{LurieKerodon}).
\end{rem}

\begin{define}[{\cite[Definition~3.1.10]{LurieGoodwillie}}]
	The \ndef{scaled coherent nerve} is a functor \(\mathrm{N}^{\sca}\colon\msCat\rightarrow \Ss \)
	defined by letting the underlying simplicial set of \(\mathrm{N}^{\sca}(\C)\) be the coherent nerve of
	the marked-simplicial category \(\C\) regarded as a simplicially enriched category (as in Definition 1.1.5.5 of \cite{HTT}),
	and setting its thin 2-simplices to be those maps \(f\colon\mathfrak{C}(\Delta^2)\rightarrow \C \)
	that send the unique non-degenerate 1-simplex of \(\mathfrak{C}(\Delta^2)(0,2)\) to a marked 1-simplex in \(\C(f(0),f(2))\).
\end{define}

The functor \(\mathrm{N}^{\sca}\) admits a left adjoint \(\mathfrak{C}^{\sca}\colon \Ss \rightarrow \msCat\), whose explicit description can be found in Definition 3.1.10 of \cite{LurieGoodwillie}.

\begin{thm}[{\cite[Theorem 4.2.7]{LurieGoodwillie}}]
	\label{coherent nerve eq}
	There exists a model structure on the category \(\Ss\) of scaled simplicial sets, characterized as follows:
	\begin{itemize}
		\item a map \(f \colon (X, T_X) \rightarrow (Y, T_Y)\) is a cofibration if and only if it is a monomorphism;
		\item a map \(f \colon (X, T_X) \rightarrow (Y, T_Y)\) is a weak equivalence if and only
		if the marked-simplicial functor \(\mathfrak{C}^{\sca}(f)\colon \mathfrak{C}^{\sca}(X, T_X)\rightarrow \mathfrak{C}^{\sca}(Y, T_Y)\) is a weak equivalence in \(\msCat\).
	\end{itemize}
	Moreover, the adjoint pair 
	\[ \xymatrixcolsep{1pc}
	\vcenter{\hbox{\xymatrix{
				**[l]\Ss \xtwocell[r]{}_{\mathrm{N}^{\sca}}^{\mathfrak{C}^{\sca}
				}{'\perp}& **[r] \msCat}}}\]
	is a Quillen equivalence with respect to the model structures defined above.
\end{thm}

We will refer to the model structure of Theorem~\ref{coherent nerve eq} as the \ndef{bicategorical model structure}.
Following~\cite{LurieGoodwillie} we will refer to the fibrant objects of this model category as \emph{\(\infty\)-bicategories}
and to its weak equivalences as \ndef{bicategorical equivalences};
moreover we will call \ndef{bicategorical fibrations} its fibrations.

\begin{prop}[{\cite[Proposition~3.1.13]{LurieGoodwillie}}]
	\label{oo-bicat are weak oo-bicat}
	If \(f\) belongs to \(\bS\) then \(\mathfrak{C}^{\sca}(f)\) is a trivial cofibration of \(\s^+\)-categories. Therefore, every \(\infty\)-bicategory is a weak \(\infty\)-bicategory. Similarly, every bicategorical fibration  
	is a scaled fibration.
\end{prop}

We will prove a converse to Proposition~\ref{oo-bicat are weak oo-bicat} in \S\ref{sec:weak-is-strong} below (see Theorem~\ref{t:fibrant}).

\begin{rem}\label{r:cartesian}
The bicategorical model structure on \(\Ss\) is \emph{cartesian}. Indeed, the cofibrations are the monomorphisms and are closed under pushout products, and it follows from~\cite[Lemma 4.2.6]{LurieGoodwillie} via a 2-out-of-3 argument the pushout product of a cofibration and a trivial cofibration is a trivial cofibration. 
\end{rem}

\begin{rem}
 \label{rem:biequivalences_closed_op}
	If \(f \colon (X, T_X) \to (Y, T_Y)\) is a bicategorical equivalence
	then also \(f^\op \colon (X, T_X)^\op \to (Y, T_Y)^\op\) is so.
	This is clear, since for every pair of vertices \((x, y)\) of \(X\)
	we have that
	\(\Map_{\mathfrak{C}^{\sca}(X, T_X)}(x, y) \cong \Map_{\mathfrak{C}^{\sca}((X, T_X)^\op)}(y, x)\).
\end{rem}

\begin{rem}\label{rem:outer_horn-2}
	The scaled outer horn
	\(\bigl(\Lambda^2_0 \amalg_{\Delta^{\{0, 1\}}} \Delta^0, \emptyset\bigr)
	\to \bigl(\Delta^2 \amalg_{\{0, 1\}} \Delta^0, \{\Delta^2_2\}\bigr)\)
	is a trivial cofibration of the bicategorical model structure.
	Indeed, it is a monomorphism and one easily checks that its image
	via \(\mathfrak{C}^{\sca}\) is a Dwyer--Kan equivalence, see~\cite[Lemma 3.5.13]{LurieGoodwillie}.
\end{rem}

\subsection{Stratified sets}

In this subsection we introduce the notion of stratified sets, and define the model category for complicial sets and other variations such as \(n\)\nbd-triv\-i\-al complicial sets.

\begin{define}[\cite{VerityWeakComplicialI}]
	A \emph{stratified set} is a pair \((X,M_X)\) where \(X\) is a simplicial sets and \(M_X\) is a collection of its \(n\)-simplices for \(n > 0\) that contains the degenerate ones.
	A map \(f\colon (X,M_X)\rightarrow (Y,M_Y)\) of stratified sets is a map of simplicial sets \(f \colon X \rightarrow Y\) such that \(f(M_X)\subset M_Y\). 
	We denote the category of stratified sets by \(\St\).
\end{define}

\begin{rem}
	As with the case of scaled simplicial sets, the category of stratified sets admits a description in terms of a reflective localization of a category of presheaves, see~\cite[Observation 108]{VerityWeakComplicialI}.
\end{rem}

\begin{notate}\label{no:stratified}
If \(X\) is a simplicial set and \(M\) is any set of simplices in \(X\) then we will denote by \((X,M)\) the stratified set whose underlying simplicial set is \(X\)
and whose marked simplices are \(M\) together with the degenerate simplicies of \(X\). If \(L \subseteq K\) is a subsimplicial set then we use \(M|_L\) to denote the set of simplices in \(L\) whose image in \(K\) is contained in~\(M\). 
\end{notate}

\begin{warning}\label{w:flat}
	Comparing Notation~\ref{no:stratified} and Notation~\ref{no:scaled} reveals a certain abuse of notation: given a simplicial set \(K\) and a collection of triangles \(T\), the symbol \((K,T)\) can either mean the scaled simplicial set whose thin triangles are \(T\) plus the degenerate ones or the stratified set on \(K\) whose marked simplices are \(T\) plus the degenerate ones. A similar ambiguity exists between Notation~\ref{no:stratified} and~\ref{no:marked}. We believe however that in the way these terms are used in this paper there will be no place where an actual confusion can arise. We also note that the category \(\Ss\) of scaled simplicial set can be identified with the full subcategory of \(\St\) consisting of those stratified sets for which all non-degenerate marked simplices are of dimension~\(2\). Under this identification the ambiguity above disappears. This full inclusion \(\iota\colon \Ss\rightarrow \St\) is also the left counterpart of the Quillen adjunction we consider in~\S\ref{sec:equivalence}.
\end{warning}

\begin{rem}
In the specific case of \(\Del^0\) we will simplify notation and denote \((\Del^0,\emptyset)\) by \(\Del^0\), since here there is really no possibility for confusion.
\end{rem}

\begin{define}\label{d:flat}
Given an integer \(k\) we will denote by \(\thr_k(K)\) the stratified set whose underlying simplicial set is \(K\) and whose marked simplices are the degenerate ones and all those in dimension \(> k\). In the case \(k=0\) we will simply write \(\thr(K)\) to \(\thr_0(K)\). If \((X,M)\) is a stratified set then we will denote by \(\thb_k(X,M)\) the stratified set whose marked simplices consist of \(M\) together with all simplices of dimension \(> k\).
\end{define}

\begin{define}[{\cite[Definition~31]{VerityWeakComplicialI}}]
	Given two stratified sets \((X,M_X)\) and \((Y,M_Y)\), we define their \emph{join} \(X\ast Y\) to be the stratified set 
	whose underlying simplicial set consists of Joyal's join of simplicial sets \(X \ast Y\),
	and such that an \(n\)-simplex \((x,y)\in (X \ast Y)_n\) with \(x\in X_k\) and \(y \in Y_{n-k-1}\) for \(-1 \leq k \leq n\) 
	is marked if and only if \(x\) is marked in \(X\) \emph{or} \(y\) is marked in \(Y\) (including the case in which both are marked).
	Here we employ the convention that \(X_{-1}\) and \(Y_{-1}\) contain a single element which is \emph{not} considered as marked for the purpose of this definition. 
\end{define}

We let \(\Delta^3_{\eq}\) be the stratified set whose underlying simplicial set is \(\Delta^3\), whose marked 1-simplices are \(\Del^{\{0,2\}}\) and \(\Del^{\{1,3\}}\) and whose \(n\)-simplices for \(n\) strictly greater than 1 are all marked. 

\begin{define}[\cite{RiehlOuverture}, \cite{VerityWeakComplicialI}]\label{d:complicial-anodyne}
	The class of \emph{complicial horns} is the weak saturation of the set of inclusions 
	\[(\Lam^n_i,M_i|_{\Lam^n_i}) \to (\Del^n,M_i)\] 
	for \(n \geq 1\) and \(0 \leq i \leq n\), where \(M_i\) consists of all the degenerate simplices and all the simplices which contain the vertices in the set
	\[\{i-1,i,i+1\}\cap \{0,\ldots, \ n\} = \{j \in [n] : |j-i|\leq 1\},\]
	whose size may be either \(2\) or \(3\) depending on \(i\) and \(n\).
	
	The class of \emph{thinness extensions} is the weak saturation of the set of inclusions of the form 
	\[ (\Del^n,M_i') \to (\Del^n,M_i'') \] 
	with \(n \geq 2\) where \(M_i'\) consists of the simplices of \(M_i\) as above together with the two \((n-1)\)-faces opposite to the vertices \(i-1\) and \(i+1\),
	while \(M_i''\) consists of the the simplices of \(M_i\) and together with all \((n-1)\)-faces. 
	
	The class of \emph{\(k\)-trivializing morphisms} is the weak saturation of the set of inclusions 
	\[(\Delta^n,\emptyset)\rightarrow (\Delta^n, \{\Del^n\})\] 
	for \(n > k\). We consider the class of \(\infty\)-trivializing morphisms to be the empty class. 
	
	Finally, the class of \emph{saturation morphisms} is the weak saturation of the set of inclusions of the form \(\Delta^3_{\eq}\ast (\Delta^n,\emptyset)\rightarrow \thr(\Delta^3)\ast (\Delta^n,\emptyset)\) for \(n \geq -1\), where \(\Del^{-1} = \emptyset\) by convention.
\end{define}

The following definition isolates the stratified sets of interest.
\begin{define}[\cite{RiehlOuverture}, \cite{VerityWeakComplicialI}]
	\label{compl set defi}
	A stratified set \((X,M)\) is a \emph{complicial set} if it has the right lifting property with respect to complicial horns and thinness extensions.
	
	A stratified set \((X,M)\) is \emph{\(k\)-trivial} if it has the right lifting property with respect to \(k\)-trivializing morphisms.
	
	A stratified set \((X,M)\) is \emph{saturated} if it has the right lifting property with respect to saturation morphisms.
\end{define} 

The following result establishes the existence of a model category structure for (saturated, \(n\)-trivial) complicial sets.

\begin{thm}[\cite{OzornovaRovelliNComplicial}, \cite{RiehlOuverture}, \cite{VerityWeakComplicialI}]
	For every \(0 \leq n \leq \infty\) there exists a model category structure on the category \(\St\) of stratified sets characterized by the following properties:
	\begin{itemize}
		\item a map \(f \colon (X, M_X) \rightarrow (Y, M_Y)\) in \(\St\) is a cofibration if and only if it is a monomorphism;
		\item a stratified set \((X,M_X)\) is fibrant if and only if it is an \(n\)-trivial saturated complicial set.
	\end{itemize}
	We denote this model category structure by \(\St_n\).
\end{thm}

\begin{rem}
The thinness extensions  \((\Del^n,M_i') \to (\Del^n,M_i'')\) for \(n \geq 4\) are already contained in the weakly saturated closure of the \(2\)-trivializing morphisms. When working in the model category \(\St_2\) one can hence restrict the attention to the thinness extensions of dimensions \(2\) and \(3\). Concretely, these consist of the following maps:
\begin{itemize}
\item
\((\Del^2,M_{(2,i)}) \to \thr(\Del^2)\), for \(i=0,1,2\), when \(M_{(2,i)}\) consists of all faces of \(\Del^2\) of dimension \(\geq 1\) which contain the vertex \(i\);
\item
\((\Del^3,M_{(3,i)}) \to \thr(\Del^3)\), for \(i=0,1,2,3\), when \(M_{(3,i)}\) consists of all faces of dimension \(\geq 2\) which contain the vertex \(i\)
as well as the edge \(\{i-1,i,i+1\} \cap \{0,1,2,3\}\) in the case \(i=0,3\).
\end{itemize}

For instance, \(M_{(2,0)}\) contains precisely \(\{\Delta^{\{0,1,2\}},\Delta^{\{0,1\}},\Delta^{\{0,2\}}\}\) plus the degenerate simplices.
\end{rem}

\begin{rem}\label{r:closure-nat}
The model category \(\St_2\) is cartesian, see~\cite[Theorem A]{OzornovaRovelliNComplicial}. This has in particular the following implication on the collection of marked simplices in a given fibrant stratified set \((X,M_X)\): if \(h\colon \thr(\Del^1) \times (\Del^n,\emptyset) \to (X,M_X)\) is a map, which we consider as encoding a natural equivalence from \(\sig_0 := h|_{\{0\} \times \Del^n}\) to \(\sig_1 := h|_{\{1\} \times \Del^n}\), then \(\sig_0\) is marked if and only if \(\sig_1\) is marked. Indeed, for \(e=0,1\) the pushout-product of the anodyne map \(\Delta^{\{e\}} \to \thr(\Delta^1)\)
with the monomorphism \(\Delta^n_\flat \to (\Delta^n, \{\Delta^{\{0, 1, \dots, n\}}\})\)
is a trivial cofibration.
\end{rem}

\section{Outer cartesian fibrations}\label{sec:fibrations}

In the \((\infty,2)\)-categorical setting one often encounters \((\infty,2)\)-categories which are \emph{fibered} into \((\infty,1)\)-categories over a given base. To describe such situations effectively one requires a robust theory which encompasses the four types of \emph{variance} such a fibration may encode. 
One such type, which we will call \ndef{inner cocartesian fibration}, was studied in~\cite{LurieGoodwillie}. It corresponds to the situation where the fiber \(\D_x\) depends \emph{covariantly} on both \(1\)-morphisms and \(2\)-morphisms. If \(\pi\colon\C \rightarrow \D\) is an inner cocartesian fibration then \(\pi^{\op}\colon\C^{\op} \rightarrow \D^{\op}\) is an \ndef{inner cartesian fibration} (whose fibers are the opposites of the fibers of \(\pi\)). These correspond to the situation where \(\D_x\) depends contravariantly on both \(1\)-morphisms and \(2\)-morphisms. The reason we use the term ``inner'' to describe these two types of fibrations is that they are inner fibrations in the sense of Joyal--Lurie on the level of the underlying simplicial sets. By contrast, the fibrations which describe a dependence which is covariant in \(1\)-morphisms and contravariant in \(2\)-morphisms (or the other way around) are generally not inner fibrations. Instead, we call them \ndef{outer (co)cartesian fibrations}. To our knowledge these types of fibrations have not yet appeared in the literature. 

In this section we introduce the notion of outer cartesian fibration and investigate its basic properties. We will begin in \S\ref{s:out} by presenting the main definitions. The prototypical example we have in mind is that of the outer cartesian fibration represented by a given object. We will construct this fibration using the slice construction in \S\ref{s:join} and show that its fibers model the corresponding mapping \(\infty\)-categories in~\S\ref{s:map}. Finally, in \S\ref{s:lift} we will show that outer cartesian fibrations satisfy a lifting property for natural transformations --- a key feature which we will exploit in later parts of the paper.

\subsection{Outer fibrations and cartesian edges}\label{s:out}
We begin by introducing the basic definitions.

\begin{define}\label{d:weak}
	We will say that a map of scaled simplicial sets \(X \rightarrow Y\) is a \ndef{weak fibration} if it has the right lifting property with respect to the following types of maps:
	\begin{description}
		\item[An1\label{item:weak-an}]
		All scaled inner horn inclusions:
		\[ (\Lam^n_i,\{\Del^{\{i-1,i,i+1\}}\}|_{\Lam^n_i}) \subseteq (\Del^n,\{\Del^{\{i-1,i,i+1\}}\}) \] 
		for \(n \geq 2\) and \(0 < i < n\).
		\item[W2\label{item:weak_outer-0}]
		The scaled horn inclusions, for all \(n \geq 2\):
		\[\Bigl(\Lam^n_0 \plus{\Del^{\{0,1\}}}\Del^0,
			\{\Del^{\{0,1,n\}}\}|_{\Lam^n_0 \plus{\Del^{\{0,1\}}}\Del^0}\Bigr)
		  \subseteq \Bigl(\Del^n\plus{\Del^{\{0,1\}}}\Del^0,\{\Del^{\{0,1,n\}}\}\Bigr).
		\]
		\item[W3\label{item:weak_outer-n}]
		The scaled horn inclusions, for all \(n \geq 2\):
		\[\Bigl(\Lam^n_n \plus{\Del^{\{n-1,n\}}}\Del^0,
			\{\Del^{\{0,n-1,n\}}\}|_{\Lam^n_n \plus{\Del^{\{n-1,n\}}}\Del^0}\Bigr) 
			\subseteq \Bigl(\Del^n\plus{\Del^{\{n-1,n\}}}\Del^0,\{\Del^{\{0,n-1,n\}}\}\Bigr).
		\]
	\end{description}
\end{define}

\begin{rem}\label{r:bicategorical-weak}
	The maps~\ref{item:scaled_anodyne_i}, \ref{item:weak_outer-0} and~\ref{item:weak_outer-n}
	in Definition~\ref{d:weak} are trivial cofibrations with respect to the bicategorical model structure: indeed, the maps of the first type are scaled anodyne,
	the maps in~\ref{item:weak_outer-0} are scaled anodyne of type~\ref{item:scaled_anodyne_iii}
	for \(n\geq 3\) and for \(n=2\) the map is a trivial cofibration by Remark~\ref{rem:outer_horn-2}; finally, the maps of the third set~\ref{item:weak_outer-n} are the opposite of a scaled anodyne map (see Remark~\ref{rem:biequivalences_closed_op}).
	It follows that every bicategorical fibration is a weak fibration. 
\end{rem}

\begin{define}[cf.~{\cite[Remark~2.4.1.4]{HTT}}]
\label{d:cartesian}
	Let \(p\colon X \rightarrow Y\) be a weak fibration. We will say that an edge \(e\colon \Del^1 \rightarrow X\) is \ndef{\(p\)-cartesian} if the dotted lift exists in any diagram of the form
	\[ \xymatrix{
		(\Lam^n_n,\{\Del^{\{0,n-1,n\}}\}|_{\Lam^n_n}) \ar^-{\sig}[r]\ar[d] & X \ar^p[d] \\
		(\Del^n,\{\Del^{\{0,n-1,n\}}\}) \ar@{.>}[ur]\ar[r] & Y \\
	}\]
	with \(n \geq 2\) and \(\sig|_{\Del^{\{n-1,n\}}} = e\). 
\end{define}

\begin{define}
\label{d:outer-fibration}
	Let \(p\colon X \rightarrow Y\) be a weak fibration.
	We will say that \(p\) is an \ndef{outer fibration} if:
\begin{description}
\item[O1\label{item:outer_fibration-i}]
it detects thin simplices, that is, a triangle in \(X\) is thin if and only if its image in \(Y\) is thin, that is \(p \colon X \to Y\) satisfies the right lifting
property with respect to the map \((\Delta^2, \emptyset) \to (\Delta^2, \Delta^2_2)\);
\item[O2\label{item:outer_fibration-ii}]
the map \(p \colon X \to Y\) satisfies the right lifting property with respect to the inclusion 
\[\Bigl(\Lam^n_0 \coprod_{\Del^{\{0,1\}}}\Del^0, \emptyset\Bigr) \subseteq
  \Bigl(\Del^n\coprod_{\Del^{\{0,1\}}}\Del^0, \emptyset\Bigr)\]
  for \(n \geq 2\);
\item[O3\label{item:outer_fibration-iii}]
the map \(p \colon X \to Y\) satisfies the right lifting property with respect to the inclusion 
 \[
	\Bigl(\Lam^n_n \coprod_{\Del^{\{n-1,n\}}}\Del^0, \emptyset\Bigr) \subseteq
  \Bigl(\Del^n\coprod_{\Del^{\{n-1,n\}}}\Del^0, \emptyset\Bigr)
\]
for \(n \geq 2\).
\end{description}
\end{define}

\begin{rem}
	The second and third lifting properties~\ref{item:outer_fibration-ii} and~\ref{item:outer_fibration-iii} above are conditions on the
	underlying simplicial set of the map \(p \colon X \to Y\).
\end{rem}

\begin{warning}
	After a preprint of this paper was released, Lurie introduced the term
	\emph{interior fibration} \cite[\href{https://kerodon.net/tag/01WF}{Tag 01WF}]{LurieKerodon}
	to encode what we just defined as outer fibrations. Our choice is motivated by the intent of highlighting that \emph{special outer horns} admit fillers against such maps.
\end{warning}

\begin{rem}\label{r:weak-scaled}
	If \(p\colon X \to Y\) is a weak fibration and detects thin triangles 
	then \(p\) is a scaled fibration (Definition~\ref{d:scaled-fib}). 
	In particular, every outer fibration is a scaled fibration.
\end{rem}

\begin{define}
	Let \(p\colon X \rightarrow Y\) be a map scaled simplicial sets.
	We will say that \(p\) is an \ndef{outer cartesian fibration} if the following conditions hold:
	\begin{enumerate}[leftmargin=*]
		\item
		The map \(p\) is an outer fibration.
		\item
		For every edge \(e\colon y \rightarrow y'\) in \(Y\) and every \(x' \in X\) such that \(p(x')=y'\) there exists a \(p\)-cartesian edge \(f\colon x\rightarrow x'\)
		of \(X\) such that \(p(f) = e\).
	\end{enumerate}
	Dually, we will say that \(p\colon X \rightarrow Y\) is an \ndef{outer cocartesian fibration} if \(p^{\op}\colon X^{\op} \rightarrow Y^{\op}\) is an outer cartesian fibration.
\end{define}

\begin{rem}\label{r:base-change}
	It is straightforward to verify that the classes of weak fibrations, outer fibrations and outer (co)cartesian fibrations are all closed under base change. In addition, cartesian edges in a pulled back fibration are detected in the original one, that is, if
	\[ \xymatrix{
	X \ar[r]\ar[d]^{p} & X' \ar[d]^{p'} \\
	Y \ar[r] & Y' 
	}\]
	is a pullback square whose vertical maps are weak fibrations, and \(e\colon x \to y\) is an arrow in \(X\) whose image in \(X'\) is \(p'\)-cartesian, then \(e\) is \(p\)-cartesian.  
\end{rem}

\begin{rem}
 \label{rem:outer_cartesian_maximal_marking}
	Suppose that \(p \colon X \to Y\) is an outer cartesian fibration
	and that~\(Y\), and so \(X\), is endowed with the maximal scaling,
	\ie all triangles of \(Y\) and of \(X\) are thin.
	Then the map of simplicial sets \(\bar{p}\) underlying \(p\) is a cartesian fibration
	of simplicial sets (see, for instance, \cite[Definition~2.4.2.1]{HTT}
	or~\cite[\href{https://kerodon.net/tag/01UA}{Tag 01UA}]{LurieKerodon}).
	Indeed, condition~\ref{item:scaled_anodyne_i} for weak fibrations implies that \(\bar{p}\)
	is an inner fibration of simplicial sets, and we have enough \(\bar{p}\)\nbd-carte\-sian edges since these agree in this case with the \(p\)-cartesian ones, as in Definition \ref{d:cartesian}.	
	Notice that if \(Y\), and so \(X\), is a weak \(\infty\)-bicategory,
	or equivalently if their underlying simplicial sets are \(\infty\)-categories,
	then conditions~\ref{item:weak_outer-0} and~\ref{item:weak_outer-n} for weak fibrations,
	which in this case are equivalent
	to conditions~\ref{item:outer_fibration-ii} and~\ref{item:outer_fibration-iii} for outer fibrations, are already satisfied
	(see, for instance, \cite[\href{https://kerodon.net/tag/01H0}{Tag 01H0}]{LurieKerodon}).
	Hence, in this particular case the map \(p \colon X \to Y\) is an outer
	cartesian fibration if and only if the underlying map \(\bar{p}\)
	of simplicial sets is a cartesian fibration.
\end{rem}

\begin{rem}\label{r:iso}
	Let \(p\colon X \to Y\) be an outer cartesian fibration and assume in addition that \(Y\) is a weak \(\infty\)-bicategory. Then \(X\) is a weak \(\infty\)-bicategory as well by Remark~\ref{r:weak-scaled}. The condition that \(p\) detects thin triangles then implies that
	\(X^{\thi} \cong X \times_{Y} Y^{\thi}\), and so the induced map \(p^{\thi}\colon X^{\thi} \to Y^{\thi}\) 
	is a cartesian fibration of \(\infty\)-categories by Remark~\ref{rem:outer_cartesian_maximal_marking}. 
	In particular, \(p^{\thi}\) is a categorical fibration and so an isofibration,
	that is, admits lifts for equivalences (see, for instance,
	\cite[\href{https://kerodon.net/tag/01EN}{Tag 01EN}]{LurieKerodon}).
	We may hence conclude that every outer cartesian fibration is an isofibration.
\end{rem}

\begin{rem}\label{r:fibers}
	It follows from Remarks~\ref{r:weak-scaled}, \ref{r:base-change}
	and~\ref{rem:outer_cartesian_maximal_marking} that if \(X \to Y\) is an outer fibration then for every \(y \in Y\) the fiber \(X_y\) is a weak \(\infty\)-bicategory
	in which every triangle is thin.
	Forgetting the scaling, we may simply consider these fibers as \(\infty\)-categories
	(cf.~Remark~\ref{rem:core-oo-cat}).
\end{rem}

\begin{rem}\label{r:cart-equiv}
	Let \(p\colon X \to Y\) be a weak fibration between weak \(\infty\)-bicategories. 
	If \(e\colon x \to y\) is a \(p\)-cartesian edge of \(X\) then it is also cartesian with respect to the inner fibration of \(\infty\)-categories \(X^{\thi} \to Y^{\thi}\). This implies, in particular, that any \(p\)-cartesian edge which lies above an equivalence in \(Y\) is necessarily an equivalence in \(X\) (see~\cite[Proposition~2.4.1.5]{HTT}).
\end{rem}

Our approach is that outer cartesian fibrations encode the data of functors from \(Y\) to \(\Cat_\infty\) which are contravariant on the level of \(1\)-morphisms but covariant on the level of \(2\)-morphisms. Similarly, outer cocartesian fibration encode functors which are covariant on the level of \(1\)-morphisms but contravariant on the level of \(2\)-morphisms. For example, we will see below how one can encode in this manner representable functors using a suitable slice construction (see \S\ref{s:map}).
A more comprehensive treatment and justification 
of this can be found in~\cite{GagnaHarpazLanariLaxLimits}.

\subsection{The join and slice constructions}\label{s:join}

Let \(\C\) be a weak \(\infty\)-bicategory and let \(y \in \C\) be a vertex. 
Define a scaled simplicial set \(\ovl{\C}_{/y}\) as follows. 
The \(n\)-simplices of \(\ovl{\C}_{/y}\) are given by \((n+1)\)-simplices 
\(\Del^{n+1}_{\flat} \to \C\) of \(\C\) which send \(\Del^{\{n+1\}}\) to \(y\), 
and a triangle in \(\ovl{\C}_{/y}\), corresponding to a \(3\)-simplex \(\sig\colon\Del^3_\flat \to \C\), is declared to be thin \(\ovl{\C}_{/y}\) if \(\sig|_{\Del^{\{0,1,2\}}}\) is thin in \(\C\). 
The scaled simplicial set \(\ovl{\C}_{/y}\) admits a natural map \(\ovl{\C}_{/y} \to \C\) 
sending \(\sig\colon\Del^{n+1}_{\flat} \to \C\) to \(\sig|_{\Del^{\{0,\dots,n\}}}\). 
If all the triangles in \(\C\) are thin (that is, if \(\C\) is actually 
an \(\infty\)-category) then the same holds for \(\ovl{\C}_{/y}\) and 
the projection \(\ovl{\C}_{/y} \to \C\) is a right fibration which 
classifies the presheaf on \(\C\) represented by \(y\). 
In this section we will detail the definitions and results in the outline just provided and
will do so in a more general case: 
we will show that for a general weak \(\infty\)-bicategory \(\C\) 
the projection \(\ovl{\C}_{/y} \to \C\) is an \emph{outer cartesian fibration}. 
In \S\ref{s:map} we will show that the fibers of these fibrations are models 
for the mapping \(\infty\)-categories in \(\C\) with target~\(y\).

In what follows it will be useful to work in a setting where both marking and scaling are allowed.
The scaling still has the function of keeping track of triangles that we think as
filled by an invertible \(2\)-cell, while the marking on edges allows us to
keep record of the cartesian edges when dealing with fibrations.

\begin{define}
	A \emph{marked-scaled} simplicial set is a triple \((X,E,T)\) where \(X\) is a simplicial set, \(E \subseteq X_1\) is a collection of edges containing all the degenerate edges and \(T \subseteq X_2\) is collection of \(2\)-simplices containing all the degenerate \(2\)-simplices. In particular, if \((X,E,T)\) is a marked-scaled simplicial set then \((X, E_X)\) is a marked simplicial set and \((X,T)\) is a scaled simplicial set. A map of marked-scaled simplicial sets \(f\colon (X,E_X,T_X)\rightarrow (Y,E_Y,T_Y)\) is a map of simplicial sets \(f\colon X \rightarrow Y\) satisfying \(f(E_X) \subseteq E_Y\) and \(f(T_X)\subseteq T_Y\).
	The collection of scaled simplicial sets and their morphisms forms a category which will be denoted by \(\Set^{+,\sca}_\Del\). 
\end{define}

\begin{rem}
	As for marked simplicial sets and scaled simplicial sets, the category of marked-scaled
	simplicial sets is the category of models of a limit sketch.
	Indeed, denote by \(\Delta_{+,\sca}\) the category \(\Delta\), to which we add
	the objects \([1]_t\) and \([2]_t\) and the morphisms \(\varphi_k \colon [k] \to [k]_t\),
	with \(k = 1, 2\), and \(\zeta^i_k \colon [k]_t \to [k-1]\),
	with \(k=1, 2\) and \(0\leq i \leq k-1\). We impose the relations
	\(\zeta^i_k \varphi_k = \sigma^i_k \colon [k] \to [k-1] \),
	for \(k = 1, 2\) and \(0 \leq i \leq k-1\),
	and \(\sigma^0\zeta^0_2 = \sigma^0 \zeta^1_2 \colon [2]_t \to [0]\).
	The category of marked-scaled simplicial sets is (equivalent to) the reflective
	localization of the category of presheaves on \(\Delta_{+, \sca}\) at the
	morphisms \([k]_t \coprod_{[k]} [k]_t \to [k]_t\), for \(k = 1, 2\).
	Here we have identified the category \(\Delta_{+, \sca}\) to its image
	via the Yoneda embedding. In particular, the category of marked-scaled
	simplicial sets is (finitely) locally presentable and it is cartesian closed.
\end{rem}

\begin{notate}
	Given a simplicial set \(K\), a set \(E\) of edges in \(K\) and a set \(T\) of triangles in \(K\) we will denote by \((K,E,T)\) the marked-scaled simplicial set whose underlying simplicial set is \(K\), whose marked edges are the degenerate ones and those contained in \(E\) and whose thin triangles are the degenerate ones and those contained in \(T\). 
	As in Notation~\ref{no:scaled}, if \(i \colon L \subseteq K\) is a subsimplicial set then we use \(E|_L\) and \(T|_L\) to denote the set of edges and triangles, respectively, in \(L\) whose image via \(i\) is contained in \(E\) and \(T\), respectively. 
\end{notate}

Given a scaled simplicial set \(S\) we will denote by \((\Set^{+,\sca}_\Del)_{/S}\) the category of marked-scaled simplicial sets \((X,E_X,T_X)\) equipped with a map of scaled simplicial sets \((X,T_X) \to S\). Notice that, if \(S^\sharp\) denotes the marked-scaled simplicial set whose underlying scaled simplicial set is \(S\)
and such that all the edges are marked, then the category \((\Set^{+,\sca}_\Del)_{/S}\) is isomorphic to the slice category \((\Set^{+,\sca}_\Del)_{/S^\sharp}\).

\begin{define}\label{d:anodyne-marked}
	Let \(S\) be a scaled simplicial set. We will denote by \(\A_S\) the smallest weakly saturated class of maps in the category \((\Set^{+,\sca}_\Del)_{/S}\) containing the following maps:
	\begin{description}
		\item[C1\label{item:marked_outer-1}]
		The inclusion \((\Lam^n_i,\varnothing,\{\Del^{\{i-1,i,i+1\}}\}|_{\Lam^n_i}) \subseteq (\Del^n,\varnothing,\{\Del^{\{i-1,i,i+1\}}\})\) for \(0 < i < n\) and every map \((\Del^n,\{\Del^{\{i-1,i,i+1\}}\}) \rightarrow S\).
		\item[C2\label{item:marked_outer-2}]
		The inclusion \((\Lam^n_n,\{\Del^{\{n-1,n\}}\}|_{\Lam^n_n},\varnothing) \subseteq (\Del^n,\{\Del^{\{n-1,n\}}\},\varnothing)\)
		for every \(n \geq 1\) and every map \(\Del^n_{\flat} \rightarrow S\). 
		\item[C3\label{item:marked_outer-3}]
		The inclusion \((\Lam^n_0 \plus{\Del^{\{0,1\}}}\Del^0,\varnothing,\varnothing) \subseteq (\Del^n\plus{\Del^{\{0,1\}}}\Del^0,\varnothing,\varnothing)\) for every \(n \geq 2\) and every map \(\Del^n_{\flat} \plus{\Del^{\{0,1\}}}\Del^0 \rightarrow S\).
		\item[C4\label{item:marked_outer-4}]
		The inclusion \(\Del^2 \subseteq (\Del^2,\emptyset,\{\Del^2\})\) for every \(\Del^2_{\sharp} \to S\).
	\end{description}
\end{define}

\begin{rem}
	The set of maps~\ref{item:marked_outer-1}, \ref{item:marked_outer-2},
	\ref{item:marked_outer-3} and~\ref{item:marked_outer-4} are part of the
	defining set of the notion of outer cartesian anodyne maps,
	which is defined and studied in~\cite[\S 2.4]{GagnaHarpazLanariLaxLimits}.
	Notice that the maps of type~\ref{item:marked_outer-2} impose that the
	image of the edge \(\Delta^{\{n-1, 1\}}\) acts as a cartesian edge
	(cf.~Definition~\ref{d:cartesian}). The maps of type~\ref{item:marked_outer-3},
	on the other hand, impose that the image of the triangle \(\Delta^{\{0, 1, n\}}\),
	(that we may think as a ``2-cell'') acts somehow as a cocartesian 2-cell.
\end{rem}

\begin{prop}\label{p:char}
Let \(p\colon X \to S\)  
be a map of marked-scaled simplicial sets. If \(f\) satisfies the right lifting property with respect to the class \(\A_S\) of Definition~\ref{d:anodyne-marked} then
\begin{itemize}
\item
The map of scaled simplicial sets \(\ovl{p}\colon \ovl{X} \to \ovl{S}\) underlying \(p\) is an outer fibration.
\item
Every marked edge in \(X\) is \(\ovl{p}\)-cartesian.
\item
For every marked arrow \(f\colon x \to y\) in \(S\), and every lift \(y' \in X\) of \(y\), there exists a marked edge \(e'\colon x' \to y'\) in \(X\) such that \(p(e') = e\).
\end{itemize}
In particular, the base change of \(\ovl{p}\) along the marked core (see~\ref{def:marked_core}) of \(S\) is an outer cartesian fibration.
\end{prop}

\begin{proof}
If we interpret the maps~\ref{item:weak-an},~\ref{item:outer_fibration-i},~\ref{item:outer_fibration-ii} and~\ref{item:outer_fibration-iii} as maps of marked-scaled simplicial sets with minimal marking on edges then the first three correspond to the generators~\ref{item:marked_outer-1},~\ref{item:marked_outer-4} and~\ref{item:marked_outer-3} of \(\A_S\), while the last one is a pushout of the generator~\ref{item:marked_outer-2}. We conclude that the map of scaled simplicial sets underlying \(p\) is an outer fibration.
	The lifting property against maps of type~\ref{item:marked_outer-2} for \(n \geq 2\) then further implies that every marked edge in \(X\) is \(\ovl{p}\)-cartesian, and for \(n=1\) implies that every marked edge in \(S\) admits marked lifts for every lift of its codomain.
	
For the last claim, note that in the marked core of $S$ all edges are (by definition) marked, and hence all edges admit $\ovl{p}$-cartesian lifts by the above.
\end{proof}

\begin{define}\label{d:join}
	Let \((Z,E_Z,T_Z)\) be a marked-scaled simplicial set and \((K,T_K)\) a scaled simplicial set. We define the \ndef{join} \((Z \ast K,T_{Z\ast K})\) to be the \emph{scaled simplicial set} whose underlying simplicial set is the ordinary join of simplicial sets \(Z \ast K\) and whose thin triangles \(T_{Z\ast K}\) are given by the subset
	\[ T_Z \amalg [E_Z \times K_0] \amalg \emptyset \amalg T_K\] of the set \[  Z_2 \amalg [Z_1 \times K_0] \amalg [Z_0 \times K_1] \amalg K_2 = (Z \ast K)_2 .\]
\end{define}

\begin{define}\label{d:marked-slice}
	For a fixed scaled simplicial set \((K,T_K)\) we may consider 
	the assignment given by \((Z,E_Z,T_Z) \mapsto (Z \ast K,T_{Z \ast K})\) 
	as a functor \(\Set^{+,\sca}_{\Del} \to \Set^{\sca}_{(K,T_K)/}\). 
	As such, it admits a right adjoint \(\Set^{\sca}_{(K,T_K)/} \to \Set^{+,\sca}_{\Del}\)
	and it is therefore a colimit preserving functor. 
	Indeed, given a map of scaled simplicial sets \(f\colon(K,T_K) \to X\), 
	considered as an object of \(\Set^{\sca}_{(K,T_K)/}\),
	define a simplicial set \(X_{/f}\) whose \(n\)-simplices, for \(n \geq 0\), are
	given by the set
	\[ 
		\Hom_{\Set^{\sca}_{(K,T_K)/}}\bigl({}{^\flat}\Delta^n \ast (K, T_K), X\bigr),
	\]
	where we have set \(^\flat\Delta^n = (\Delta^n, \emptyset, \emptyset)\).
	We define its marked edges to be the maps \(^\flat\Delta^1 \ast (K, T_K) \to X\)
	factoring through
	\((\Delta^1, \Delta^1_1, \emptyset) \ast (K, T_K) \to X\)
	and its thin triangles to be the maps \(^\flat\Delta^2 \ast (K, T_K) \to X\)
	factoring through
	\((\Delta^2, \emptyset, \Delta^2_2) \ast (K, T_K) \to X\).
	It is not hard to check that for every marked-scaled simplicial set \((Z, E_Z, T_Z)\)
	we get a (natural) bijection
	\[
		\Hom_{\s^{+, \sca}}\bigl((Z, E_Z, T_Z), X_{/f}\bigr) =
		\Hom_{\Set^{\sca}_{(K,T_K)/}}\bigl((Z \ast K,T_{Z \ast K}),X\bigr).
	\]
	We will denote by \(\ovl{X}_{/f}\) the scaled simplicial set underlying \(X_{/f}\).
\end{define}

\begin{warning}
	Lurie introduces in~\cite[Notation~4.1.5]{LurieGoodwillie} a notion of slice
	\(\ovl{\mathcal{C}}^{X/}\) over a vertex \(X\) of a scaled simplicial set \(\mathcal{C}\). We then point out that when considering the vertex \(X\) of \(\mathcal{C}\) as a map \(\Delta^0_\flat \to \mathcal{C}\), Lurie's scaled simplicial set
	\(\ovl{\mathcal{C}}^{X/}\) differs from \(\ovl{\mathcal{C}}_{X/}\) as defined above.
	Nonetheless, Lurie's scaled slice can be recovered using
	a ``thick'' version of the slice construction, that we study in detail
	in~\cite[\S4.2]{GagnaHarpazLanariLaxLimits} and show being
	bicategorical equivalent to that obtained from Definition~\ref{d:marked-slice} 
	(see~\cite[Corollary~4.2.11]{GagnaHarpazLanariLaxLimits}).
	To avoid confusion, let us also note that while we use the notation \(\C_{X/}\) without the over-line to denote a marked-scaled simplicial set refining \(\ovl{\C}_{X/}\), Lurie uses the notation \(\mathcal{C}^{X/}\) to denote a certain simplicial set, which contains, but not is equal to, the underlying simplicial set of \(\ovl{\C}^{X/}\). 
\end{warning}

\begin{rem}
 \label{rem:join_underlying_id}
 	Consider a simplicial set \(X\) endowed with two scalings \(T_X \subseteq T_X'\)
 	and a monomorphism of marked-scaled simplicial sets \(f \colon (K, E_K, T_K) \hookrightarrow (L, E_L, T_L)\).
 	In order to simplify the argument, we identify \((K, E_K, T_K)\) as a marked-scaled
 	simplicial subset of \((L, E_L, T_L)\).
 	The scaling of \((L, E_L, T_L) \ast (X, T_X')\) is by definition given by the subset
 	\(T_L \amalg [E_L \times X_0] \amalg \emptyset \amalg T_X'\) of the set 
 	\[ 
 		(L \ast X)_2 = L_2 \amalg [L_1 \times X_0] \amalg [L_0 \times X_1] \amalg X_2.
 	\]
 	Setting \(g \colon (X, T_X) \to (X, T_X')\), we may consider the pushout-join
 	of \(f\) and \(g\), given by
 	\[
 		(L, E_L, T_L) \ast (X, T_X) \bigcup (K, E_K, T_K) \ast (X, T_X').
 	\]
 	Its underlying simplicial set is clearly \(L \ast X\), while its scaling is the union
 	of the subsets 
 	\[
 		T_L \amalg [E_L \times X_0] \amalg \emptyset \amalg T_X
 		\quad\text{and}\quad
 		T_K \amalg [E_K \times X_0] \amalg \emptyset \amalg T_X'
 	\]
 	of \((L \ast X)_2\). This union is precisely the set \(T_{L\ast X}\).

 	A completely similar argument shows that if \(g \colon (X, T_X) \hookrightarrow (Y, T_Y)\)
 	is a monomorphism of scaled simplicial sets and \(f\) is a monomorphism
 	of marked-scaled simplicial sets of the type \((L, \emptyset, T_L) \hookrightarrow (L, \emptyset, T_L')\),
 	then the pushout-join of \(f\) and \(g\) is isomorphic to \((L, \emptyset, T_L') \ast (Y, T_Y)\).
\end{rem}

\begin{lem}\label{l:pushout-join}
	Let \(f\colon (X,E_X,T_X) \to (Y,E_Y,T_Y)\) be a map of marked-scaled simplicial sets and \(g\colon (A,T_A) \to (B,T_B)\) a map of scaled simplicial sets.
	If \(f\) belongs to \(\A_S\) and \(g\) is a monomorphism then the map of scaled simplicial sets
	\begin{equation}\label{e:pushout-join} 
	(X \ast B,T_{X\ast B}) \coprod_{(X \ast A,T_{X\ast A})} (Y \ast A,T_{Y\ast A} )\to (Y \ast B,T_{Y \ast B}) 
	\end{equation}
	is in the weakly saturated closure of maps of type~\ref{item:scaled_anodyne_i}
	and~\ref{item:scaled_anodyne_iii} in Definition~\ref{anodyne defi}. In particular, it is scaled anodyne.
\end{lem}
\begin{proof}
	It suffices to check the claim on generators, and so we may assume that \(g\) is either the inclusion \((\partial \Del^n)_\flat \hrar \Del^n_\flat\)
	or the inclusion \(\Del^2_\flat \hrar \Del^2_{\sharp}\) (see Remark~\ref{rem:monos_scaled}), 
	and \(f\) is one of the generating maps appearing in Definition~\ref{d:anodyne-marked}.
	When \(g\) is the map \(\Del^2 \hrar \Del^2_{\sharp}\) then~\eqref{e:pushout-join} is an isomorphism by Remark~\ref{rem:join_underlying_id}. 
	We may hence assume that \(g\) is the inclusion \(\partial \Del^n \hrar \Del^n\) for some \(n \geq 0\). Let us now consider the various possibilities for \(f\) case by case.
	The combinatorics of the underlying simplicial sets is well-known, see for instance~\cite[Lemma~3.3]{JoyalQCatsAndKan}.
	\begin{description}
		\item[C1]
		When \(f\) is the inclusion \((\Lam^m_i,\varnothing,\{\Del^{\{i-1,i,i+1\}}\}|_{\Lam^m_i}) \subseteq (\Del^m,\varnothing,\{\Del^{\{i-1,i,i+1\}}\})\)
		for \(0 < i < m\) the map~\eqref{e:pushout-join} is isomorphic to the map
		\[ \bigl(\Lam^{[m]\ast[n]}_{i},\{\Del^{\{i-1,i,i+1\}}\}\bigr) \to \bigl(\Del^{[m]\ast[n]},\{\Del^{\{i-1,i,i+1\}}\}\bigr) \]
		which is a map of type~\ref{item:scaled_anodyne_i}.
		\item[C2]
		When \(f\) is the inclusion \((\Lam^m_m,\{\Del^{\{m-1,m\}}\}|_{\Lam^m_m},\varnothing) \subseteq (\Del^m,\{\Del^{\{m-1,m\}}\},\varnothing)\) for \(m \geq 1\)
		the map~\eqref{e:pushout-join} is isomorphic to the map
		\[
			\bigl(\Lam^{[m]\ast[n]}_{m}, T|_{\Lam^{[m]\ast[n]}}\bigr) \to \bigl(\Del^{[m]\ast[n]}, T\bigr),
		\]
		where \(T\) is the set of triangles \(\{m-1, m, k\}\) of \(\Del^{[m]\ast[n]}\), for \(m < k \leq m+n+1\).
		This map is isomorphic to the pushout of
		\[ \bigl(\Lam^{[m]\ast[n]}_{m},\{\Del^{\{m-1,m,m+1\}}\}|_{\Lam^{[m]\ast[n]}}\bigr) \to \bigl(\Del^{[m]\ast[n]},\{\Del^{\{m-1,m,m+1\}}\}\bigr), \]
		which is a map of type~\ref{item:scaled_anodyne_i}, along
		\[\bigl(\Lam^{[m]\ast[n]}_{m},\{\Del^{\{m-1,m,m+1\}}\}|_{\Lam^{[m]\ast[n]}}\bigr) \to \bigl(\Lam^{[m]\ast[n]}_{m}, T|_{\Lam^{[m]\ast[n]}}\bigr).\]
		\item[C3]
		When \(f\) is the inclusion \((\Lam^m_0 \plus{\Del^{\{0,1\}}}\Del^0,\varnothing,\varnothing) \subseteq (\Del^m \plus{\Del^{\{0,1\}}}\Del^0,\varnothing,\varnothing)\) for \(m \geq 2\) the map~\eqref{e:pushout-join} is isomorphic to the map
		\begin{equation}\label{eq:pushout-join-3}
			\bigl(\Lam^{[m]\ast [n]}_0 \plus{\Del^{\{0,1\}\ast[n]}}\Del^{[0]\ast [n]},\varnothing\bigr) \to
			\bigl(\Del^{[m] \ast [n]} \plus{\Del^{\{0,1\}\ast[n]}}\Del^{[0]\ast[n]},\varnothing\bigr).
		\end{equation}
		which is a pushout of the outer horn inclusion \((\Lam^n_0)_{\flat} \subseteq \Delta^n_{\flat}\). In addition, for any subset \(S\) of \([n]\), the simplex \(\Delta^{\{0, 1\} \ast S}\) of \(\Lam^{[m]\ast [n]}_0 \plus{\Del^{\{0,1\}\ast[n]}}\Del^{[0]\ast [n]}\)
		is degenerate; in particular, in the domain of~\eqref{eq:pushout-join-3} the edge \(\Delta^{\{0,1\}}\) is degenerate and triangle \(\Delta^{\{0,1,n\}}\) is thin. It is therefore also the pushout of the map
\[		\bigl(\Lam^{[m]\ast [n]}_0 \plus{\Del^{\{0,1\}}}\Del^{[0]},\{\Del^{\{0,1,n\}}\}\bigr) \to
			\bigl(\Del^{[m] \ast [n]} \plus{\Del^{\{0,1\}}}\Del^{[0]},\{\Del^{\{0,1,n\}}\}\bigr).\]
		which is a map of type~\ref{item:scaled_anodyne_iii}.		
		\item[C4]
		When \(f\) is the inclusion \(\Del^2 \subseteq (\Del^2,\emptyset,\{\Del^2\})\),
		then the map~\eqref{e:pushout-join} is an isomorphism by Remark~\ref{rem:join_underlying_id}. \qedhere
	\end{description}
\end{proof}

\begin{define}
	The \ndef{marked core} of a marked-scaled simplicial set \(X\) is the (marked-)scaled simplicial subset
	spanned by those \(n\)-simplices whose 1-dimensional faces are marked in \(X\).
\end{define}

\begin{cor}\label{c:slice-3}
	Let \(p\colon X \to S\) be a map of scaled simplicial sets which satisfies the right lifting property with respect to maps of type~\ref{item:scaled_anodyne_i} 
	and~\ref{item:scaled_anodyne_iii} in Definition~\ref{anodyne defi}. Let \(f\colon K \to X\) be a map of scaled simplicial sets and let \(i\colon L \hrar K\)
	be an inclusion of scaled simplicial sets. Then the map 
	\[ \bar{q}\colon \ovl{X}_{/f} \to \ovl{X}_{/fi} \times_{\ovl{S}_{/pfi}} \ovl{S}_{/pf} \]
	is an outer fibration such that every marked edge of \(X_{/f}\) is \(\bar{q}\)-cartesian
  and every marked edge in the codomain admits a marked (and hence \(q\)-cartesian) lift.
  In particular, the base change of \(\bar{q}\)
	to the marked core of \(q \colon X_{/f} \to X_{/fi} \times_{S_{/pfi}} S_{/pf}\) is an outer cartesian fibration. 
\end{cor}

\begin{proof}
	For any monomorphism \(j \colon A \to B\) of marked-scaled simplicial sets,
	adjunction yields the following correspondence of lifting problems
	\[
		\begin{tikzcd}
			A \ar[r]\ar[d, "j"']	& X_{/f} \ar[d, "q"]				\\
			B \ar[r]\ar[ur, dotted]	& X_{/fi} \times_{S_{/pfi}} S_{/pf}
		\end{tikzcd}
		\quad \leftrightsquigarrow \quad
		\begin{tikzcd}
			A \ast K \cup_{A \ast L} B \ast L \ar[r]\ar[d, "j \hat{\ast} i"']	& X \ar[d, "p"]	\\
			B \ast K \ar[r] \ar[ru, dotted]							& S
		\end{tikzcd},
	\]
	where \(j \hat{\ast} i\) is the pushout-join of \(j\) with \(i\) (with respect to the join operation of Definition~\ref{d:join}), and the square on the right is considered in the category of scaled simplicial sets under \(K\). As applying to the right square the forgetful functor \((\Ss)_{K/} \to \Ss\) induces a bijection on the set of possible dotted lifts, we may equivalently view it as a lifting problem in scaled simplicial sets. The desired result now follows by combining Lemma~\ref{l:pushout-join} with Proposition~\ref{p:char}.
\end{proof}

The particular case of Corollary~\ref{c:slice-3} 
where \(S = \Delta^0\) and \(L = \emptyset\) gives the following:

\begin{cor}\label{c:slice}
	Let \(X\) be a scaled simplicial set which satisfies the extension property with respect to maps
  of type~\ref{item:scaled_anodyne_i} and~\ref{item:scaled_anodyne_iii} in Definition~\ref{anodyne defi}. 
  Given a map \(f\colon K \to X\) of scaled simplicial sets, the map \(q\colon \ovl{X}_{/f} \to X\)
  is an outer cartesian fibration, every marked edge in \(X_{/f}\) is \(q\)-cartesian and
  every edge of \(X\) has a marked lift. 
\end{cor}

\begin{proof}
	It is enough to notice that all the edges of \(X_{/\emptyset}\) are marked, that is it coincides with its marked core, and its underlying scaled simplicial set is \(X\).
\end{proof}

In the situation of Corollary~\ref{c:slice}, if \(K\) is a point and the image of \(f\) is the vertex \(y \in X\) then we will denote \(X_{/f}\) and \(\ovl{X}_{/f}\) simply by \(X_{/y}\) and \(\ovl{X}_{/y}\), respectively. In particular, the outer cartesian fibration \(q\colon\ovl{X}_{/y} \to X\) is the one we alluded to in the beginning of this section. We consider \(q\) as the outer cartesian fibration \emph{represented} by \(y\). A precise formulation and proof of this statement now appears in~\cite{GagnaHarpazLanariLaxLimits}, though a partial justification is already offered by Proposition~\ref{p:mapping-1} below.

\begin{rem}\label{r:closure}
	Let \(X\) be a scaled simplicial set and let \(y \in X\) a vertex. Then for \(i=1,2\), if the scaled simplicial set \(X\) satisfies the extension property with respect to the inclusion \((\Del^3,T_i) \to \Del^3_{\sharp}\) (where \(T_i\) denotes all triangles which contain the vertex \(i\)) then the marked-scaled simplicial set \(X_{/y}\) satisfies the extension property with respect to the inclusions \((\Del^2,E_i,\{\Del^2\}) \to (\Del^2,\Del^2,\{\Del^2\})\) (where \(E_i\) denotes all edges containing \(i\)). In particular, if \(X\) is a weak \(\infty\)-bicategory and \(\sig\colon\Del^2 \to X_{/y}\) is a thin triangle such that \(\sig|_{\Del^{\{1,2\}}}\) is marked then \(\sig|_{\Del^{\{0,1\}}}\) is marked if and only if \(\sig|_{\Del^{\{0,2\}}}\) is marked (see Remark~\ref{rmk:j's are anod}).
\end{rem}

\subsection{Mapping categories}\label{s:map}

Let \(\C\) be a weak \(\infty\)-bi\-cat\-egory with underlying simplicial set \(\ovl{\C}\) and let \(x,y \in \ovl{\C}\) be two vertices.
Recall the following explicit model for the mapping \(\infty\)-category from \(x\) to \(y\) in \(\C\) constructed in~\cite[\S 4.2]{LurieGoodwillie} (see also~\cite[Notation 4.1.5]{LurieGoodwillie} and~\cite[\S 4.2.1]{HTT} for the ``fattened'' slice construction used there).

\begin{define}
\label{def:mapping-oo-cat}
Let \(\Hom_{\C}(x,y)\) be the marked simplicial set whose \(n\)-simplices are given by maps \(f\colon\Del^1 \times \Del^n \to \ovl{\C}\)
such that \(f|_{\{0\} \times \Del^n}\) is constant on \(x\), \(f|_{\{1\} \times \Del^n}\) is constant on \(y\), and the triangle \(f|_{\Del^{\{(0,i),(1,i),(1,j)\}}}\) is thin 
for every \(0 \leq i\leq j \leq n\). The marked edges of \(\Hom_{\C}(x,y)\) are the edges corresponding to those maps \(\Del^1 \times \Del^1 \to \ovl{\C}\) which send both triangles of \(\Del^1 \times \Del^1\) to thin triangles.
\end{define}
As shown in~\cite[\S 4.2]{LurieGoodwillie}, the assumption that \(\C\) is a weak \(\infty\)-bicategory implies that the marked simplicial set \(\Hom_{\C}(x,y)\) is \emph{fibrant} in the marked categorical model structure, that is, it is an \(\infty\)-category whose marked edges are exactly the equivalences.

This construction can also be understood in terms of the \ndef{Gray product} of scaled simplicial sets.
We further study this tensor product in~\cite{GagnaHarpazLanariGrayLaxFunctors}, where
we also compare it with Vertity's Gray product of 2-complicial sets (see~\cite[\S2]{VerityComplicial} and~\cite{OzornovaRovelliVerityGray}),
showing that they induce the same tensor product on the \(\infty\)-category of \((\infty, 2)\)-categories.
For our purpose here let us consider the following limited variant:
\begin{define}
	Let \(K\) and \(L\) be two simplicial sets. The \ndef{Gray product} \(K \mgr L\) is the \emph{scaled} simplicial set whose underlying simplicial set is the cartesian product \(K \times L\) and such that a \(2\)-simplex \(\sig\colon \Del^2 \to K \times L\) is thin if and only if the following conditions hold:
	\begin{enumerate}
		\item
		The image of \(\sig\) in both \(K\) and \(L\) is degenerate.
		\item
		The \(1\)-simplex \(\sig|_{\Del^{\{1,2\}}}\) maps to a degenerate edge in \(K\) or \(\sig|_{\Del^{\{0,1\}}}\) maps to a degenerate edge in \(L\) (or both). 
	\end{enumerate}
\end{define}

In terms of this Gray product we can describe the \(n\)-simplices in \(\Hom_{\C}(x,y)\) as maps \(f\colon\Del^1 \mgr \Del^n \to \C\) such that \(f|_{\{0\} \times \Del^n}\) is constant on \(x\) and \(f|_{\{1\} \times \Del^n}\) is constant on \(y\). An edge \(\Del^1 \mgr \Del^1 \to \C\) is marked exactly when it factors through the map \(\Del^1 \mgr \Del^1 \to \Del^1_{\flat} \times \Del^1_{\flat}\). 

Now recall from \S\ref{s:join} that for \(y \in \C\), the map of scaled simplicial sets \(p\colon\ovl{\C}_{/y} \to \C\) is an outer cartesian fibration (Corollary~\ref{c:slice}) and all the marked edges of \(\C_{/y}\) are \(p\)-cartesian. If \(x \in \C\) is then another vertex then by Remark~\ref{r:fibers} the fiber \((\ovl{\C}_{/y})_x\) is a weak \(\infty\)-bicategory in which all triangles are thin, and every marked edge in \(({\C}_{/y})_x\) is an equivalence. 
Forgetting the scaling, we will denote by \(\Hom^{\triangleright}
_{\C}(x,y)\) the underlying \emph{marked} simplicial set of \((\C_{/y})_x\). In particular,  \(\Hom^{\triangleright}
_{\C}(x,y)\) is an \(\infty\)-category endowed with a marking, and all marked edges are equivalences.

\begin{const}\label{cn:map}
	We construct a natural map of marked simplicial sets
	\begin{equation}\label{e:i} 
	i\colon \Hom^{\triangleright}_{\C}(x,y) \to \Hom_{\C}(x,y) .
	\end{equation}
	Explicitly the \(n\)-simplices of \(\Hom^{\triangleright}_{\C}(x,y)\) are given by maps 
	\[g\colon (\Del^n \ast \Del^0)_{\flat} = \Del^{n+1}_{\flat} \to \C\]
	which send \(\Del^{\{0,\dots,n\}}\) to \(x\) and \(\Del^{\{n+1\}}\) to \(y\) (where an edge is marked exactly when it corresponds to a thin triangle), 
	while the \(n\)-simplices of \(\Hom_{\C}(x,y)\) correspond to maps \(\Del^1 \mgr \Del^n \to \C\) 
	(where an edge \(\Del^1 \mgr \Del^1 \to \C\) is marked exactly with it factors through 
	\(\Del^1_{\flat} \times \Del^1_{\flat}\)). The map~\ref{e:i} is then obtained by precomposing
	along the natural map 
\begin{equation}\label{e:product-to-join}
\Del^1 \mgr \Del^n \to (\Del^n \ast \Del^0)_{\flat} = \Del^{n+1}_{\flat}
\end{equation}
	which on vertices sends \((0, i)\) to \(i\) and \((1, i)\) to \(n+1\)
	(this map indeed sends the thin triangles \(\Del^{\{(0,i),(1,i),(1,j)\}}\) of \(\Del^1 \mgr \Del^n\) to degenerate triangles).
	Notice that by definition, if \(\sigma \colon \Delta^n_\flat \ast \Delta^0 \to \C\) is an \(n\)-simplex in \(\Hom^{\triangleright}_{\C}(x,y)\),
	then the precomposed map \(\Delta^1 \mgr \Delta^n\) is such that the restriction to \(\Delta^{\{0\}} \mgr \Delta^n\)
	is mapped to \(x\) and the restriction to \(\Delta^{\{1\}} \mgr \Delta^n\) is mapped to \(y\), so that \(i(\sigma)\)
	is indeed an \(n\)-simplex of \(\Hom_{\C}(x, y)\).
	It is straightforward to verify that this assignment is compatible with face and degeneracy maps, and with the markings on both sides.
\end{const}

\begin{rem}\label{r:fibrant}
	Inspecting Construction~\ref{cn:map} we see that the map~\eqref{e:i} detects marked edges.
	Since the marked edges in \(\Hom_{\C}(x,y)\) are exactly the equivalences, as \(\Hom_{\C}(x,y)\) is a fibrant marked simplicial set,
	it follows that every equivalence in \(\Hom^{\triangleright}_{\C}(x,y)\) is marked.
	In particular, the marked edges in \(\Hom^{\triangleright}_{\C}(x,y)\) are exactly the equivalences
	and \(\Hom^{\triangleright}_{\C}(x,y)\) is fibrant with respect to the marked categorical model structure.
\end{rem}

\begin{prop}\label{p:mapping-1}
	Let \(\C\) be a weak \(\infty\)-bicategory. Then the map
	\[ i\colon \Hom^{\triangleright}_{\C}(x,y) \to \Hom_{\C}(x,y) \]
	of Construction~\ref{cn:map} is a marked categorical equivalence of fibrant marked simplicial sets for every \(x,y \in \C\).
\end{prop}

The proof of Proposition~\ref{p:mapping-1} will require the following lemma:

\begin{lem}\label{l:core}
	Let \(\C\) be a weak \(\infty\)-bicategory and for \(m \geq 2\) consider a diagram of marked simplicial sets
	\begin{equation}\label{e:core} 
	\xymatrix{
		(\partial \Del^{\{1,\dots,m\}},\emptyset) \ar@{^{(}->}[d]\ar^-{f_0}[r] & \Hom^{\triangleright}_{\C}(x,y) \ar[d] \\
		(\Lam^m_0,\{\Del^{\{0,1\}}\}) \ar^-{h_0}[r] & \Hom_{\C}(x,y) \\
	}
	\end{equation}
	where \(\partial \Del^{\{1,\dots,m\}}\) and \(\Lam^m_0\) are considered as simplicial subsets of \(\Del^m = \Del^{\{0,\dots,m\}}\). 
	Then there exists an extension of~\eqref{e:core} to a diagram the form
	\begin{equation}\label{e:core-2} 
	\xymatrix{
		(\partial \Del^{\{1,\dots,m\}},\emptyset) \ar@{^{(}->}[d] \ar@{^{(}->}[r] & (\Del^{\{1,\dots,m\}},\emptyset) \ar@{^{(}->}[d]\ar^-{f}[r] & \Hom^{\triangleright}_{\C}(x,y) \ar[d] \\
		(\Lam^m_0,\{\Del^{\{0,1\}}\}) \ar@{^{(}->}[r] & (\Del^m,\{\Del^{\{0,1\}}\}) \ar^-{h}[r] & \Hom_{\C}(x,y) \\
	}
	\end{equation}
\end{lem}
\begin{proof}
	Let
	\[ Z_0 := \Del^1 \mgr \Lam^m_0 \coprod_{\partial \Del^1 \mgr \Lam^m_0} \partial \Del^1 \mgr \Del^m \subseteq \Del^1 \mgr \Del^m .\] 
	Unwinding the definitions, what we need to prove amounts to solving an extension problem of the form
	\[
	\xymatrix{
		Z_0 \ar[r]^{g_0}\ar[d] & \C \\
		\Del^1 \mgr \Del^m \ar@{.>}[ur]_-{g} & \\
	}
	\]
	under the assumptions that
	\begin{enumerate}[label=(\roman{*}), ref=(\roman{*})]
		\item \label{i:assume}
		the map \(g_0\) sends \(\Del^{\{0\}} \mgr \Del^m\) to the vertex \(x\) and \(\Del^{\{1\}} \mgr \Del^m\) to the vertex \(y\);
		\item \label{ii:assume}
		the restriction of \(g_0\) to \(\Del^1 \mgr \partial \Del^{\{1,\dots,m\}}\) factors through \(\partial \Del^{\{1,\dots,m\}} \ast \Del^0\); and
		\item \label{iii:assume}
		the map \(g_0\) sends \(\Del^{\{(0,0),(0,1),(1,1)\}}\) to a thin triangle in \(\C\);
	\end{enumerate}
	and with the additional constraint that 
	\begin{enumerate}
		\item[(\(\ast\))]\label{i:star}
		the restriction of \(g\) to \(\Del^1 \mgr \Del^{\{1,\dots,m\}}\) factors through \(\Del^{\{1,\dots,m\}} \ast \Del^0\).
	\end{enumerate}
	Here, we note that Condition~\ref{i:star} is compatible with Condition~\ref{ii:assume} since in both cases the factorization is along the map induced by the natural (and surjective) map~\eqref{e:product-to-join}. 
	Now for \(i=0,\dots,m-1\) let \(\sigma_i\colon \Del^m \to \Del^1 \times \Del^{\{1,\dots,m\}}\) be the \(m\)-simplex given on vertices by the formula
	\[ \sigma_i(j) = \left\{\begin{matrix} (0,j+1) & j \leq i \\ (1,j) & j > i \end{matrix}\right. .\]
	These are exactly all the non-degenerate \(m\)-simplices of the simplicial set \(\Del^1 \times \Del^{\{1,\dots,m\}}\). For \(i=0,\dots,m-2\) (that is, for all except the last one) our constraint~\hyperref[i:star]{(\(\ast\))} above requires that the desired lift \(g\)
	sends \(\sigma_i\) to the degenerate \(m\)-simplex in \(\C\) whose image is the \((i+1)\)-simplex \(g_0 \circ (\sigma_i|_{\Del^{\{0,\dots,i+1\}}})\).
	Let \(Z_1 \subseteq \Del^1 \mgr \Del^{m}\) be the union of \(Z_0\) and the simplices \(\sigma_i\) for \(i \leq m-2\). 
	We define a map \(g_1\colon Z_1 \to \C\) by setting \({g_1}|_{Z_0} = {g_0}|_{Z_0}\) and on the simplex \(\sig_i\) letting
	\(g_1 \circ \sigma_i\) be the composition \(\Del^{\{0,\dots,m\}} \to \Del^{\{0,\dots,i+1\}} \to \C\), where the first map collapses \(i+1,\dots,m\) to \(i+1\)
	and the second is given by \(g_0 \circ ({\sigma_i}|_{\Del^{\{0,\dots,i+1\}}})\). This is indeed well-defined on \(Z_0 \cap \im(\sig_i)\) since \(g_0\) is assumed to
	satisfy~\ref{i:assume} and~\ref{ii:assume}. We now observe that any map \(\Del^1 \otimes \Del^{\{1,...,m\}} \to \C\) whose restriction to \(Z_1 \cap (\Del^1 \otimes \Del^{\{1,...,m\}}) = \cup_{i=0}^{m-2} \im(\sig_i)\) agrees with \(g_1\) as constructed above will necessarily factor through \(\Del^{\{1,...,m\}} \ast \Del^0\). 
We are hence left with solving the unconstrained extension problem
	\[
	\xymatrix{
		Z_1 \ar[r]^{g_1}\ar[d] & \C \\
		\Del^1 \mgr \Del^m \ar@{.>}[ur]_-{g} & \ .
	}
	\]
	Note that \(Z_1\) is obtained from the boundary \(\partial (\Del^1 \times \Del^m) = \partial \Del^1 \times \Del^m \coprod_{\partial \Del^1 \times \partial \Del^m} \Del^1 \times \partial \Del^m\) by removing the \(m\)-simplex \(\sig_{m-1}\). Now for \(i=1,\dots,m+1\) let \(\tau_i\colon \Del^{m+1} \to \Del^1 \times \Del^{m}\) be the \((m+1)\)-simplex given on vertices by the formula:
	\[ \tau_i(j) = \left\{\begin{matrix} (0,j) & j < i \\ (1,j-1) & j \geq i \end{matrix}\right. .\]
	We note that these are exactly all the top-dimensional simplices of the product \(\Delta^1 \times \Delta^m\).
	For \(j=2,\dots,m+2\) let \(Z_j \subseteq \Del^1 \mgr \Del^m\) be the union of \(Z_1\) and the simplices \(\tau_i\) for \(i<j\). We then observe that \(Z_{m+2} = \Del^1 \mgr \Del^m\) and that for \(j=1,\dots,m\) we have a pushout square
	\[ \xymatrix{
		(\Lam^{m+1}_{j},\{\Del^{\{j-1,j,j+1\}}\}) \ar[r]\ar[d] & Z_{j} \ar[d] \\
		(\Del^{m+1},\{\Del^{\{j-1,j,j+1\}}\}) \ar[r] & Z_{j+1} \\
	}\]
	In particular \(Z_j \to Z_{j+1}\) is scaled anodyne for \(j=1,\dots,m\) and since \(\C\) is a weak \(\infty\)-bicategory we may extend \(g_1\) to a map \(g_{m+1}\colon Z_{m+1} \to \C\). Finally, note that the \((m+1)\)-simplex \(\tau_{m+1}\) contains the \(m\)-simplex \(\sig_{m-1}\) as the face opposite \(0\), and this is the only face of \(\tau_{m+1}\) which is not in the image of \(Z_{m+1} \to Z_{m+2}\). 
	We therefore have a pushout square
	\[ \xymatrix{
		(\Lam^{m+1}_{0})_{\flat} \ar[r]\ar[d] & Z_{m+1} \ar[d] \\
		\Del^{m+1}_{\flat} \ar[r] & Z_{m+2} \ ,
	}\]
	and by~\ref{i:assume} above the composite map \((\Lam^{m+1}_0)_{\flat} \to Z_{m+1} \to \C\) 
	collapses \(\Del^{\{0,1\}}\) to \(x\). 
		Now the left vertical map in the above square is not scaled anodyne.
	Nonetheless, since the map
	\[
		(\Lam^{m+1}_0 \coprod_{\Del^{\{0,1\}}}\Del^0,\{\Del^{\{0,1,m+1\}}\}) \to (\Del^{m+1} \coprod_{\Del^{\{0,1\}}}\Del^0,\{\Del^{\{0,1,m+1\}}\})
	\]
	is scaled anodyne we could finish the proof provided that we show that the following condition
	is satisfied: The composite map \((\Lam^{m+1}_0)_{\flat} \to Z_{m+1} \to \C\) 
	sends \(\Del^{\{0,1,m+1\}}\) to a thin triangle,
	or, equivalently, that \(g_{m+1}\) sends \(\Del^{\{(0,0),(0,1),(1,m)\}}\) to a thin triangle.
	We know that the map \(g_{m+1}\):
	\begin{itemize}
		\item sends the triangle \(\Del^{\{(0,0),(0,1),(1,1)\}}\) to a thin triangle by~\ref{iii:assume},

		\item sends every triangle of the form \(\Del^{\{(0,i),(1,i),(1,j)\}}\) to a thin triangle by the definition of the Gray product, and

		\item sends \(\{1\} \times \Del^m\) to a point.
	\end{itemize}
	 It suffices to apply Remark~\ref{rmk:j's are anod} to the \(3\)-simplex \(g_{m+1}(\Del^{\{(0,0),(1,0),(1,1),(1,m)\}})\)
	 to deduce that \(g_{m+1}\) sends \(\Del^{\{(0,0),(1,1),(1,m)\}}\) to a thin triangle
	 and then similarly to apply Remark~\ref{rmk:j's are anod} to the \(3\)-simplex \(g_{m+1}(\Del^{\{(0,0),(0,1),(1,1),(1,m)\}})\)
	 to deduce that \(g_{m+1}\) sends \(\Del^{\{(0,0),(0,1),(1,m)\}}\) to a thin triangle, as desired.
\end{proof}

\begin{proof}[Proof of Proposition~\ref{p:mapping-1}]
	We first note that the map \(i\colon\Hom^{\triangleright}_{\C}(x,y) \to \Hom_{\C}(x,y)\)
  is a morphism of naturally marked \(\infty\)-categories. We will thus drop the marking and consider only the underlying \(\infty\)-categories. Observe that this map is bijective
  on vertices. It will hence suffice to show that it is fully-faithful.
	
	Given vertices \(\alp,\bet \in \Hom^{\triangleright}_{\C}(x,y)\), let \(X_{\alp,\bet} = (\Hom^{\triangleright}_{\C}(x,y)_{/\beta})_{\alpha}\) and similarly \(Y_{\alp,\bet} = (\Hom_{\C}(x,y)_{/i\beta})_{i\alpha}\). As established in~\cite[\S\S 1.2.2]{HTT} these are Kan complexes which model the corresponding mapping spaces on both sides. We hence need to show that the map 
	\[ X_{\alp,\bet} \to Y_{\alp,\bet} \] 
	is a homotopy equivalence of Kan complexes. Now a morphism \(i \colon X\to Y\) of Kan complexes is a homotopy equivalence if and only if the homotopy fiber of \(i\) at any vertex \(z\) of \(Y\)
	is contractible (see, for instance, \cite[\href{https://kerodon.net/tag/00ZP}{Tag 00ZP}]{LurieKerodon} applied to a factorization of \(i\) as anodyne followed by a Kan fibration). In addition, since \(Y_{z/} \to Y\) is a Kan fibration between Kan complexes
  (this follows, for instance, from the dual of~\cite[Corollary~3.20]{JoyalQCatsApplications}) and \(Y_{z/} \simeq \ast\), the homotopy fiber of \(i\) over a point \(y \in Y\) can be modelled by the Kan complex \(X \times_{Y} Y_{z/}\).
	To finish the proof we now show that the Kan complex \((X_{\alp,\beta}) \times_{Y_{\alp,\beta}} (Y_{\alp,\beta})_{z/}\) is contractible for every \(z \in Y_{\alp,\beta}\)
  (see, for instance, \cite[\S 8]{JoyalQCatsApplications}). This amounts to showing that for any \(m \geq 2\), any diagram	
	\begin{equation}
	\label{diag:lifting_problem_homotopy_fiber}
	\begin{tikzcd}
		{\partial \Del^{\{1,\dots,m-1\}}} \ar[d, hook]\ar[r] & {X_{\alp,\bet}} \ar[d] \\
		{\Lam^{m-1}_0} \ar[r] & {Y_{i\alp,i\bet}}
	\end{tikzcd},
	\end{equation}
  where the image of \(\Delta^{\{0\}}\) of the bottom map is \(z\),
	extends to a diagram of the form
	\begin{equation}
	\label{diag:lift_homotopy_fiber}
	\begin{tikzcd}
		\partial \Del^{\{1,\dots,m-1\}} \ar[d, hook] \ar[r, hook] & \Del^{\{1,\dots,m-1\}} \ar[d, hook] \ar[r] & X_{\alp,\bet} \ar[d] \\
		\Lam^{m-1}_0 \ar[r, hook] & \Del^{m-1} \ar[r] & Y_{i\alp,i\bet}
	\end{tikzcd}.
	\end{equation}
  Indeed, diagram~\eqref{diag:lifting_problem_homotopy_fiber} corresponds to a
  map
  \(\partial \Delta^{\{1, \dots, m-1\}} \to X_{\alp,\bet} \times_{Y_{i\alpha, i\beta}} (Y_{i\alpha, i\beta})_{z/}\), while diagram~\eqref{diag:lift_homotopy_fiber} corresponds to
  an extension
  \[
    \begin{tikzcd}
      \partial \Delta^{\{1, \dots, m-1\}} \ar[r] \ar[d] &
          X_{\alp,\bet} \times_{Y_{i\alpha, i\beta}} (Y_{i\alpha, i\beta})_{z/} \\
      \Delta^{\{1, \dots, m-1\}} \ar[ur] & 
    \end{tikzcd}.
  \]
  Now, by adjunction diagram~\eqref{diag:lifting_problem_homotopy_fiber} corresponds to
  a diagram
  \[
    \begin{tikzcd}
    \Lambda^{\{1, \dots, m\}}_m \ar[d, hook]\ar[r, "{\ovl{f_0}}"] & \Hom^{\triangleright}_{\C}(x,y) \ar[d]\\
    \Lam^{m}_0 \ar[r, "h_0"] &  \Hom_{\C}(x,y)
    \end{tikzcd}.
  \]
  Setting \(f_0 \colon \partial \Delta^{\{1, \dots, m\}} \to \Hom^{\triangleright}_{\C}(x, y)\)
  to be the gluing of \(\ovl{f_0}\) together with \({f_0}_{|\Delta^{\{1, \dots, m-1\}}}\)
  degenerate at \(\alpha\) and noticing that \({h_0}_{|\Delta^{\{0, 1\}}}\) is degenerate at \(i\alpha\), the extension of diagram~\eqref{diag:lift_homotopy_fiber}
	is a direct consequence of Lemma~\ref{l:core} (applied to diagrams of the form~\eqref{e:core} for which \(f\) sends \(\Del^{\{1,\dots,m-1\}}\) to \(\alp\) and \(\Del^{\{m\}}\) to \(\beta\) and \(h\) sends \(\Del^{\{0,\dots,m-1\}}\) to~\(i\alp\)).
\end{proof}

We conclude this section by deducing a characterization of bicategorical equivalences in terms of \(\Hom^{\triangleright}_{\C}(x,y)\). We first recall the following fundamental result of Lurie:
\begin{lem}[{\cite[Lemma~4.2.3]{LurieGoodwillie}}]\label{l:lurie-hom}
Let \(\C\) be a weak \(\infty\)-bicategory, and \(x,y\) a pair of objects in it. Then \(\mathfrak{C}^{\sca}(\C)(x,y)\) and \(\Hom_{\C}(x,y)\) are weakly equivalent as marked simplicial sets.
\end{lem}
Since the functor \(\mathfrak{C}^{\sca}(\bullet)\) creates weak equivalences, and since the weak equivalences in \(\msCat\) are the Dwyer--Kan ones, we obtain the following corollary by combining Lemma~\ref{l:lurie-hom} with Proposition~\ref{p:mapping-1}:
\begin{cor}
\label{eq = ff + ess surj}
  Let \(f \colon \C \to \D\) be a map of weak \(\infty\)-bicategories. Then \(f\) is a bicategorical equivalence if and only if it is fully-faithful and essentially surjective. More precisely, if and only if:
  \begin{itemize}
    \item The induced map \(f_{x,y}\colon\Hom^{\triangleright}_{\C}(x,y) \to \Hom^{\triangleright}_{\D}(f(x),f(y))\) is an equivalence of marked simplicial sets for every pair of objects \((x,y)\) in \(\C\).
     \item For every object \(d \in \D\) there exists an object \(c \in \C\) and an invertible edge \(h\colon f(c) \overset{\sim}{\to} d\) in \(\D\).
     
   \end{itemize} 
\end{cor}

\subsection{Cartesian lifts of natural transformations}\label{s:lift}
In this section we prove a result concerning cartesian lifts of natural transformations. It will be used in the proof of Lemma \ref{l:saturation-anodyne}, but we present it now because some of the tools employed in proving it (e.g. the filtration of Construction~\ref{cn:filtration}) will be used throughout the paper. 
\begin{define}
\label{d: pointwise nat transf}
  Let \(A_0 \subseteq A,X\) be scaled simplicial sets, with \(X\) an \(\infty\)-bicategory. By a \ndef{natural transformation} of functors \(A \to X\) relative to \(A_0\) we will mean a map \(H\colon \Delta^1_{\flat} \times A \to X\) such that \(H|_{\Del^1 \times A_0}\) factors through the projection \(\Del^1 \times A_0 \to A_0\). In this case we will say that \(H|_{\Del^{\{0\}} \times A}\) is the \ndef{source} of \(H\) and \(H_{\Del^{\{1\}} \times A}\) is its \ndef{target}. We will say that \(H\) is
  \begin{itemize}
  \item \emph{constant} if it factors through the projection \(\Del^1 \times A \to A\); 
  \item \ndef{pointwise degenerate} 
  if the edge \(H(\Delta^1 \times \{a\})\) is degenerate for every vertex \(a \in A\); 
  \item \ndef{pointwise invertible} 
  if the edge \(H(\Delta^1 \times \{a\})\) is invertible for every vertex \(a \in A\);
  \item \ndef{pointwise \(p\)-cartesian} with respect to a weak fibration  \(p \colon X \to Y\) if the edge \(H(\Delta^1 \times \{a\})\) is \(p\)-cartesian for every vertex \(a \in A\).
  \end{itemize}
\end{define}

Our goal in this subsection is to prove the following:
\begin{prop}[Lifting natural transformations]\label{p:natural-lift}
	Let \(p\colon X \rightarrow Y\) be a weak fibration of scaled simplicial sets which detects thin triangles and let \(A \subseteq B\) be an inclusion of scaled simplicial sets. Consider a lifting problem of the form
	\[
    \begin{tikzcd}
    \Del^{\{1\}} \times B \displaystyle\mathop{\coprod}_{\Del^{\{1\}} \times A} \Del^1_{\flat} \times A  \ar[r, "f"]\ar[d] & X \ar[d, "p"] \\
    \Del^1_{\flat} \times B \ar[r, "H"]\ar[ur, dotted, "\wtl{H}"{description}] & Y
    \end{tikzcd}
  \]
	such that \(f\) sends every edge of the form \(\Del^1_{\flat} \times \{a\}\) for \(a \in A\) to a \(p\)-cartesian edge. Suppose given for every \(b \in B\) a \(p\)-cartesian edge \(e_b\colon \Del^1_{\flat} \times \{b\} \to X\) such that \(e_b(\Del^{\{1\}} \times \{b\}) = f(\Del^{\{1\}}\times \{b\})\), \(p(e_b) = H|_{\Del^1_{\flat} \times \{b\}}\) and \(e_a = f|_{\Del^1_{\flat} \times \{a\}}\) for every \(a \in A \subseteq B\). Then there exists a dotted lift \(\wtl{H}\colon \Del^1_{\flat} \times B \rightarrow X\) as indicated with \(\wtl{H}\vert_{\Del^1 \times \{b\}} = e_b\) for every \(b \in B\).  
\end{prop}

\begin{rem}
The statement of Proposition~\ref{p:natural-lift} can be loosely stated as follows: given a natural transformation \(\Del^1_{\flat} \times B \to Y\), any lift of its target to \(X\) can be extended to a lift of the entire natural transformation, in a manner that is furthermore pointwise \(p\)-cartesian. In addition, this lift can be chosen to respect any lift of the same nature fixed in advance along a given scaled simplicial subset \(A \subseteq B\), and the \(p\)-cartesian edges \(\wtl{H}(\Del^1_{\flat} \times \{b\})\) can be any prescribed collection of \(p\)-cartesian lifts of the edges \(H(\Del^1_{\flat} \times \{b\})\). 
\end{rem}

\begin{rem}
Inspecting the proof of Proposition~\ref{p:natural-lift} below shows that the condition that \(p\) detects thin triangles is only needed if \(B\) contains thin triangles that are not in \(A\). 
\end{rem}

In the proof of Proposition~\ref{p:natural-lift} we
will make use of the following type of filtration
that we will also employ in a few other proofs later on:

\begin{const}\label{cn:filtration}
	Fix \(n \geq  0\).
	For \(i=0,\dots,n\), let \(S_i\) be the collection of all triangles in \(\Del^{n+1}\) which are either degenerate or contain the edge \(\Del^{\{i,i+1\}}\), and let 
	\[\tau_i\colon (\Del^{n+1},S_i) \rightarrow \Del^1_{\flat} \times \Del^n_{\flat}\] be the map given on vertices by the formula
	\[ \tau_i(m) = \left\{\begin{matrix} (0,m) & m \leq i \\ (1,m-1) & m > i \end{matrix}\right. \]
	For \(k=0,\dots,n+1\) let \(Z^k \subseteq \Del^{1}_{\flat} \times \Del^n_{\flat}\) be the union of 
	\([\Del^1_{\flat} \times \partial \Del^{n}_{\flat}] \coprod_{\Del^{\{0\}}\times \partial \Del^{n}_{\flat}}[\Del^{\{0\}}\times \Del^{n}_{\flat}]\)	
	and the simplices \(\tau_i\) for \(n \geq i \geq k\)
	(defining \(\tau_{n+1}\) to be empty).
	We then have an ascending filtration of scaled simplicial sets
	\begin{equation}\label{e:filtration-0} 
	[\Del^1_{\flat} \times \partial \Del^{n}_{\flat}] \coprod_{\Del^{\{0\}}\times \partial \Del^{n}_{\flat}}[\Del^{\{0\}}\times \Del^{n}_{\flat}] = Z^{n+1} \subseteq Z^{n} \subseteq \dots \subseteq Z^0 =\Del^1_{\flat} \times   \Del^{n}_{\flat}
	\end{equation}
	and for each \(k=0,\dots,n\) we have a pushout square of scaled simplicial sets
	\[ 
	\xymatrix{
		(\Lam^{n+1}_k, {S_k}|_{\Lam^{n+1}_k}) \ar[r]\ar[d] & Z^{k+1} \ar[d] \\
		(\Del^{n+1},S_k) \ar[r] & Z^k \\
	}\]
	where the composed map \((\Del^{n+1},S_k) \rightarrow Z^k \rightarrow \Del^1_{\flat} \times \Del^{n}_{\flat}\) is the simplex \(\tau_k\).
  Notice that for \(k>0\) the set \(S_k\) contains the triangle \(\Delta^{\{k-1, k, k+1\}}\) by definition,
  so the left vertical map is a pushout of a morphism of type~\ref{item:scaled_anodyne_i} and hence a scaled anodyne map. In the final step \(k=0\) we obtain that the left vertical map is an outer horn inclusion, and in particular not scaled anodyne.
 
	Dually, for \(k=0,\dots,n+1\) let \(Z_k \subseteq \Del^{n}_{\flat} \times \Del^1_{\flat}\) be the union of 
	\[[\Del^1_{\flat} \times \partial \Del^{n}_{\flat}] \coprod_{\Del^{\{1\}}\times \partial \Del^{n}_{\flat}}[\Del^{\{1\}}\times \Del^{n}_{\flat}]\]
	and the simplices \(\tau_{i}\), for \(0 \leq i < k\).
	We then have an ascending filtration of scaled simplicial sets
	\begin{equation}\label{e:filtration-1} 
	[\Del^1_{\flat} \times \partial \Del^{n}_{\flat}] \coprod_{\Del^{\{1\}}\times \partial \Del^{n}_{\flat}}[\Del^{\{1\}}\times \Del^{n}_{\flat}]	
	= Z_0 \subseteq Z_1 \subseteq \dots \subseteq Z_{n+1} = \Del^1_{\flat} \times \Del^{n}_{\flat}  
	\end{equation}
	and for each \(k=0,\dots,n\) we have a pushout square of scaled simplicial sets
	\[ 
	\xymatrix{
		(\Lam^{n+1}_{k+1}, {S_k}|_{\Lam^{n+1}_{k+1}}) \ar[r]\ar[d] & Z_{k} \ar[d] \\
		(\Del^{n+1},S_k) \ar[r] & Z_{k+1} \\
	}\]
	where the composed map \((\Del^{n+1},S_k) \rightarrow Z_k \rightarrow \Del^1_{\flat} \times \Del^{n}_{\flat}\) is the simplex \(\tau_{k}\).
  Again, the left vertical maps are scaled anodyne morphisms, except for the final step \(k = n\). 

  On the underlying simplicial sets, these filtrations already appear in the literature, see for instance~\cite[Proposition 2.1.2.6]{HTT}.
\end{const}

\begin{rem}
	Although the above construction is very similar to that used
	in the proof of Lemma~\ref{l:core}, observe that the latter
	was a filtration for \(\Delta^1 \mgr \Delta^n\) and the
	starting simplicial subset there (regardless of whether one chooses what was called \(Z_0\)
	or \(Z_1\) in the proof) was different from \(Z_0\) of the above construction.
\end{rem}

\begin{proof}[Proof of Proposition~\ref{p:natural-lift}]
Let \(A_0,B_0\) be the sets of vertices of \(A\) and \(B\), respectively, considered as discrete scaled simplicial sets. Let \(A' = A \coprod_{A_0} B_0\). Then the \(p\)-cartesian edges \(\{e_b\}\) together with the map \(f\) assemble to form a map 
\[f'\colon \Del^{\{1\}} \times B \coprod_{\Del^{\{1\}} \times A'} \Del^1_{\flat} \times A'\to X\] 
which extends \(f\). Replacing \(A\) with \(A'\), we may hence assume without loss of generality that the inclusion \(A \subseteq B\) is bijective on vertices.  
We then have a bijective correspondence of lifting problems
  \[
    \begin{tikzcd}
    \Del^{\{1\}} \times B \displaystyle\mathop{\coprod}_{\Del^{\{1\}} \times A} \Del^1_{\flat} \times A  \ar[r, "f"]\ar[d] & X \ar[d, "p"] \\
    \Del^1_{\flat} \times B \ar[r, "H"]\ar[ur, dotted, "\wtl{H}"{description}] & Y
    \end{tikzcd}
    \quad\leftrightsquigarrow\quad
    \begin{tikzcd}
      A \ar[r] \ar[d]           & X^{\Delta^1} \ar[d] \\
      B \ar[r] \ar[ur, dotted]  & X^{\{1\}} \times_{Y^{\{1\}}} Y^{\Delta^1}
    \end{tikzcd}
  \]
  through which the statement of Proposition~\ref{p:natural-lift} translates to the statement that a dotted lift in the right square exists whenever the top horizontal map sends every vertex of \(A\) to a \(p\)-cartesian arrow. The collection of injective maps \(A \to B\) which are bijective on vertices and have this property is weakly saturated, and it will hence suffice to prove the claim for a suitable set of generators, which we can take to be the set consisting of the inclusion \(\partial \Del^n_{\flat} \hrar \Del^n_{\flat}\) for \(n \geq 1\) and the inclusion \(\Del^2_{\flat} \subseteq \Del^2_{\sharp}\). In the latter case the left vertical map in the left square is an isomorphism on the level of underlying simplicial sets, and so the desired lift exists by the assumption that \(p\colon X \to Y\) detects thin triangles. It is left to prove the case of \(\partial \Del^n_{\flat} \hrar \Del^n_{\flat}\) for \(n \geq 1\).
  We need to construct a lift in a square of the form
	\[
    \begin{tikzcd}
    \Del^{\{1\}} \times \Del^n_{\flat} \displaystyle\mathop{\coprod}_{\Del^{\{1\}} \times \partial \Del^n_{\flat}} \Del^1_{\flat} \times \partial \Del^n_{\flat}  \ar[r, "f"]\ar[d] & X \ar[d, "p"] \\
    \Del^1_{\flat} \times \Del^n_{\flat} \ar[r, "H"]\ar[ur, dotted, "\wtl{H}"{description}] & Y
    \end{tikzcd}
  \] 
  under the assumption that \(\Del^1_{\flat} \times \Del^{\{i\}}\) maps to a \(p\)-cartesian edge for every \(i\). In fact, we will only need to use this assumption for \(i=n\).
  For this we use the filtration~\eqref{e:filtration-1} of Construction~\ref{cn:filtration} to construct the lift step by step.
The steps \(k=0, \dots, n-1\) of the filtration
  can be lifted since these are pushouts of maps of type~\ref{item:scaled_anodyne_i}, and we assume that \(p\) is a weak fibration. The last step of the filtration is a pushout of the map
\begin{equation}\label{eq:outer-horn}
(\Lam^{n+1}_{n+1},S_n) \to (\Del^{n+1},S_n)
\end{equation}
when \(n \geq 2\) (where \(S_n\) is the set of all triangles containing the edge \(\Del^{\{n,n+1\}}\), which are all contained in \(\Lam^{n+1}_{n+1}\)) and of the map
\[ (\Lam^2_2)_{\flat} \to \Del^2_{\sharp} \]
when \(n=1\). Since \(S_n\) contains the triangle \(\Del^{\{0,n,n+1\}}\) we see that in either case the desired lift exists
since the edge corresponding to \(\Del^{\{n,n+1\}} \subseteq \Lam^{n+1}_{n+1}\) maps to a \(p\)-cartesian edge of \(X\) (it coincides with the image of \(\Del^1 \times \Del^{\{n\}}\) under \(f\), which is \(p\)-cartesian by assumption). 
\end{proof}

\begin{rem}
  In the second part of the proof, that is when we already reduced the inclusion \(A \subseteq B\) to the boundary inclusions,
  we use the condition for which \(f\) maps edges
  of the form \(\Delta^1 \times \{a\}\) to \(p\)-cartesian edges only for \(a = n\), \ie the last vertex of the simplex.
  But notice that in fact we need the full power of that assumption for the first part of the proof,
  that is in order to prove the weak saturation of the class of inclusions of simplicial sets satisfying that
  particular kind of left lifting property and therefore reduce to the boundary inclusions.
\end{rem}

\section{Thin triangles in weak \pdfoo-bicategories}\label{sec:thin}

In this section we will establish some useful properties of thin triangles which we will need in the subsequent sections. We begin with the following lemma which turns an arbitrary \(2\)-simplex into a \(2\)-morphism:

\begin{lem}\label{l:moving-1}
	\label{reduced form}
	Given a 2-simplex \(\alpha\) in a weak \(\infty\)-bicategory \(\C\), there exists a 3-simplex \(\Upsilon_{\alpha}\colon (\Delta^3,\Delta^{\{1,2,3\}}) \to \C\) of the form
	\begin{nscenter}
		\begin{tikzpicture}[scale=1.2]
		\squares{%
			/squares/label/.cd,
			0=$0$, 1=$1$, 2=$2$, 3=$3$,
			01={}, 12={$f$}, 23={$g$}, 03={$h$}, 02={$f$}, 13={$h'$},
			012={$=$}, 023={$\alpha$}, 013={$\hat{\alpha}$}, 123={$\simeq$},
			0123={$\Upsilon_{\alpha}$},
			/squares/arrowstyle/.cd,
			01={equal}, 012={phantom, description},
			/squares/labelstyle/.cd,
			012={anchor=center}
		}
		\end{tikzpicture}
	\end{nscenter} 
that is, the restriction of \(\Upsilon_{\alp}\) to \(\Del^{\{0,2,3\}}\) is \(\alp\) and its restriction to \(\Delta^{\{0,1,2\}}\) is degenerate along \(\Delta^{\{0,1\}}\). 
\end{lem}
\begin{proof}
	We first construct a thin triangle \(\Del^{\{1,2,3\}}_{\sharp} \to \C\) by extending the 2\hyp{}dimensional horn \((f,g)\colon \Lambda^2_1\rightarrow \C\) along the scaled anodyne map \(\Lambda^2_1 \subseteq \Del^2_{\sharp}\), thus obtaining the bottom right triangle in the right square above. This triangle together with \(\alp\) and the degenerate triangle in the left square determine a map \(H\colon(\Lambda^3_2,\{\Delta^{\{1,2,3\}}\}) \rightarrow \C\). 
	Extending along the scaled anodyne inclusion 
	\[(\Lambda^3_2,\{\Delta^{\{1,2,3\}}\}) \subseteq (\Del^3,\{\Del^{\{1,2,3\}}\})\] 
	we get the desired \(3\)-simplex \(\Upsilon_{\alpha}\).
\end{proof}

\begin{rem}\label{r:hat-alp}
	We identify \(\hat{\alp}\) with an edge \([\hat{\alp}]\) in the marked simplicial set \((\C_{/\alp(2)})_{\alp(0)}\), which in turn is weakly equivalent to the mapping \(\infty\)-category \(\Hom_{\C}(\alp(0),\alp(2))\) by Proposition~\ref{p:mapping-1}. It follows from Remark \ref{rmk:j's are anod} that \(\alpha\) is thin if and only if \(\hat{\alpha}\) is thin, that is, if and only if the edge \([\hat{\alp}]\) is marked.
\end{rem}

\begin{prop}\label{p:detects}
	Let \(f\colon\C \to \D\) be a bicategorical equivalence of weak \(\infty\)-bicategories and let \(\sig\colon \Del^2 \to \C\) be a triangle. Then \(\sig\) is thin in \(\C\) if and only if \(f\sig\) is thin in \(\D\).
\end{prop}
\begin{proof}
	Let us depict \(\sig\) as a diagram
	\[ \xymatrix{
		& y \ar^{g}[dr] & \\
		x\ar[ur]^{l}\ar[rr]_{h} & \twocell{u}{\alp}{0.4}{0}{0.12} & z \\
	}\]
	which is commutative up to a (not-necessarily-invertible \(2\)-cell) \(\alp\colon h \Rightarrow gl\). 
	Applying Lemma~\ref{l:moving-1} and Remark~\ref{r:hat-alp} we may reduce to the case where \(l\) is degenerate. Similarly, we can use \(f(\Upsilon_{\alpha})\) as \(\Upsilon_{f(\alpha)}\) (following the notation of the abovementioned lemma and remark). This allows us to consider \(\alp\) as encoding an edge \([\alpha]\) in \(\Hom^{\triangleright}_{\C}(x,z)\), which is marked if and only if \(\alp\) is a thin \(2\)-simplex. Similarly, \(f(\alpha)\) can be thought of as encoding an edge \([f(\alpha)]\) in \(\Hom^{\triangleright}_{\D}(f(x),f(z))\), which is marked if and only if \(f(\alp)\) is a thin \(2\)-simplex. Since \(f\) is a bicategorical equivalence, 
Corollary~\ref{eq = ff + ess surj} implies that the map
	\begin{equation}\label{e:f}
	\Hom^{\triangleright}_{\C}(x,z) \to \Hom^{\triangleright}_{\D}(f(x),f(z))
	\end{equation}
	is a marked categorical equivalence. Since \(\Hom^{\triangleright}_{\C}(x,z)\) and \(\Hom^{\triangleright}_{\D}(f(x),f(z))\) are fibrant marked simplicial sets (see Remark~\ref{r:fibrant}) their marked edges are exactly the respective equivalences, and hence~\eqref{e:f} detects marked edges. We may then conclude that \(\alp\) is thin in \(\C\) if and only if \(f\alp\) is thin in \(\D\), as desired.
\end{proof}

\begin{prop}\label{p:joyal-2}
	Let \(\C\) be a weak \(\infty\)-bicategory and \(\sig\colon (\Del^3,T) \to \C\) a map of scaled simplicial sets, where \(T = \{\Del^{\{0,1,3\}},\Del^{\{0,2,3\}}\}\). Then the following holds:
  \begin{enumerate}
    \item If \(\sig(\Del^{\{0,1\}})\) is an equivalence and \(\sig(\Del^{\{0,1,2\}})\) is thin then \(\sig(\Del^{\{1,2,3\}})\) is thin.
    \item If \(\sig(\Del^{\{2,3\}})\) is an equivalence and \(\sig(\Del^{\{1,2,3\}})\) is thin in \(\C\) then \(\sig(\Del^{\{0,1,2\}})\) is thin. 
  \end{enumerate} 
\end{prop}
\begin{proof}
We prove the first claim. The proof of the second is completely analogous. Since \(\fC^{\sca} \dashv \rN^{\sca}\) is a Quillen equivalence there exists a fibrant \(\Set^+_\Del\)-enriched category \(\E\) equipped with a bicategorical equivalence \(f\colon \C \to \rN^{\sca}(\E)\). By Proposition~\ref{p:detects} we have that \(\sig|_{\Del^{\{1,2,3\}}}\) is thin in \(\C\) if and only if \(f\sig|_{\Del^{\{1,2,3\}}}\) is thin in \(\rN^{\sca}(\E)\). We may hence reduce to the case where \(\C = \rN^{\sca}(\E)\). Let \(\sig^{\ad}\colon \fC^{\sca}(\Del^3,T) \to \E\) be the adjoint map. Then the restriction of \(\sig^{\ad}\) to \(\fC^{\sca}(\Del^{\{1,2,3\}}_{\flat})\) determines a diagram in \(\E\) of the form
	\[ \xymatrix{
		& y \ar^{f_{2,3}}[dr] & \\
		x\ar[ur]^{f_{1,2}}\ar[rr]_{f_{1,3}} & \twocell{u}{\alp}{0.4}{0}{0.12} & z \\
	}\]
	where the \(\alp\) corresponds to an edge in the marked simplicial set \(\Map_{\E}(x,z)\) going from \(f_{1,3}\) to \(f_{2,3} \circ f_{1,2}\). To show that \(\sig|_{\Del^{\{1,2,3\}}}\) is thin in \(\rN^{\sca}(\E)\) we need to show that \(\alp\) is a marked edge of \(\Map_{\E}(x,z)\). Now since \(\sig|_{\Del^{\{0,1\}}}\) is an equivalence in \(\rN^{\sca}(\E)\) we have that the edge \(f_{0,1}\colon x' \to x\) in \(\E\) determined by \(\sig^{\ad}|_{\fC(\Del^{\{0,1\}})}\) is an equivalence in \(\E\), i.e., admits an inverse up to homotopy. This implies that the pre-composition map
	\[\begin{tikzcd} \Map_\E(x,z) \ar[r, "{f_{0,1}^*}"] & \Map_\E(x',z) \end{tikzcd}\]
	is a categorical equivalence of (fibrant) marked simplicial sets. Such an equivalence detects marked edges (which coincide in this case with the collection of equivalences), and hence it will suffice to show that the edge \(f_{0,1}^*\alp\colon \Del^1 \to \Map_{\E}(x',z)\) is marked. Now the map \(\sig^{\ad}_*\colon \Map_{\fC^{\sca}(\Del^3,T)}(0,3) \to \Map_{\E}(x',z)\) determines a commutative square of the form
	\begin{equation}\label{e:square-map} 
	\xymatrix{
		f_{0,3} \ar[r]\ar[d] & f_{1,3} \circ f_{0,1} \ar^{f_{0,1}^*\alp}[d] \\
		f_{2,3} \circ f_{0,2} \ar[r] & f_{2,3} \circ f_{1,2} \circ f_{0,1} \\
	}
	\end{equation}
	in the marked simplicial set \(\Map_{\E}(x',z)\). Since \(\sig\) sends the triangles \(\Del^{\{0,1,3\}}, \Del^{\{0,2,3\}}\) and \(\Del^{\{0,1,2\}}\) to thin triangles in \(\Ne^{\sca}(\E)\) it follows that the two horizontal arrows and the left vertical arrow in~\eqref{e:square-map} are marked. Since \(\Map_{\E}(x',z)\) is fibrant it follows that \(f_{0,1}^*\alp\) is marked as well, and so the proof is complete.
\end{proof}

\begin{cor}\label{c:joyal-2}
Let \(\C\) be a weak \(\infty\)-bicategory and let \(\rho_0,\rho_1\colon\Del^2\to \C\) be two triangles. Let \(h\colon \rho_0 \to \rho_1\) be a natural transformation from \(\rho_0\) to \(\rho_1\), i.e., a map \(h\colon \Del^1_{\flat} \times \Del^2_{\flat} \to \C\) such that \(h|_{\Del^{\{i\}} \times \Del^2} = \rho_i\). Then the following holds:
\begin{enumerate}
\item
If \(h|_{\Del^1 \times \{0\}}\) is an equivalence in \(\C\) and \(\rho_0\) is thin then \(\rho_1\) is thin.
\item
If \(h|_{\Del^1 \times \{2\}}\) is an equivalence in \(\C\) and \(\rho_1\) is thin then \(\rho_0\) is thin.
\end{enumerate}
\end{cor}
\begin{proof}
For \(i=0,1,2,3\) let \(\alp_i\colon\Del^2 \to \Del^1 \times \Del^2\) be the \(2\)-simplex given by 
\[ \alp_i(j) = \left\{\begin{matrix} (0,j) & j < i \\ (1,j) & j \geq i \end{matrix}\right. \]
By the definition of the cartesian product of scaled simplicial sets all the triangles in \(\Del^1_{\flat} \times \Del^2_{\flat}\) are thin except the \(\alp_i\)'s. To prove (1), assume that \(\rho_0 = h|_{\alp_3}\) is thin in \(\C\). 
Applying Remark~\ref{rmk:j's are anod} to the \(3\)-simplex \(\Del^{\{(0,0),(0,1),(0,2),(1,2)\}}\) we get that \(h|_{\alp_2}\) is thin. Applying Remark~\ref{rmk:j's are anod} to the \(3\)-simplex \(\Del^{\{(0,0),(0,1),(1,1),(1,2)\}}\) we get that \(h|_{\alp_1}\) is thin. Finally, applying the first point of Proposition~\ref{p:joyal-2} to the \(3\)-simplex \(\Del^{\{(0,0),(1,0),(1,1),(1,2)\}}\) we get that \(\rho_1 = h|_{\alp_0}\) is thin. To prove (2), assume that \(\rho_1 = h|_{\alp_0}\) is thin in \(\C\). 
Applying Remark~\ref{rmk:j's are anod} to the \(3\)-simplex \(\Del^{\{(0,0),(1,0),(1,1),(1,2)\}}\) we get that \(h|_{\alp_1}\) is thin. Applying Remark~\ref{rmk:j's are anod} to the \(3\)-simplex \(\Del^{\{(0,0),(0,1),(1,1),(1,2)\}}\) we get that \(h|_{\alp_2}\) is thin. Finally, applying the second point of Proposition~\ref{p:joyal-2} to the \(3\)-simplex \(\Del^{\{(0,0),(0,1),(0,2),(1,2)\}}\) we get that \(\rho_0 = h|_{\alp_3}\) is thin.
\end{proof}

\section{The moving lemma}\label{sec:moving}
The goal of this section is to prove a crucial technical result, which we call the \emph{moving lemma}. It will be used in Section 7 to prove the existence of a Quillen adjunction between the bicategorical model structure on scaled simplicial sets and the model structure on stratified sets for saturated \(2\)-trivial complicial sets.\\
Throughout this section let us fix a surjective map of simplices \(\rho\colon \Del^n \to \Del^m\) and a section \(\sig\colon \Del^m \to \Del^n\) such that \(\sig(i)\) is minimal in \(\rho^{-1}(i)\) for every \(i \in [m]\). We will denote such a pair by \(\sig \dashv \rho\), since these maps constitute an adjunction between the posets \([n]\) and \([m]\).\\

\begin{define}\label{d:admissible}
Let \(A \subseteq \Del^n\) be a subsimplicial set. We will say that \(A\) is \emph{\((\sig \dashv \rho)\)\nbd-ad\-missible} if it satisfies the following two properties:
\begin{itemize}
\item
The image \(\sig\rho(A)\) is contained in \(A\).
\item
If \(A\) contains a vertex \(i \in [n]\) which is not in the image of \(\sig\) then \(A\) contains every simplex of \(\Del^n\) which has \(i\) as a terminal vertex.
\end{itemize}
We will refer to non-degenerate simplices in \(A\) via the corresponding subset \(I \subseteq [n]\). Similarly, we will refer to the non-degenerate simplices in \(\Del^1 \times A\) via the corresponding subset \(J \subseteq [1] \times [n]\). Given such a \(J\) we denote \(J_0 := \{j \in [n]\:|\: (0,j) \in J\}\) and \(J_1 := \{j \in [n] \:|\: (1,j) \in J\}\), and when \(J_0 \neq \emptyset\) we call \(m(J) := \max(J_0)\) the \emph{index} of \(J\).
\end{define}

\begin{examples}\label{ex:admissible}\
\begin{enumerate}
\item
The \(n\)-simplex \(\Del^n\) itself and the empty set \(\emptyset \subseteq \Del^n\) are \((\sig \dashv \rho)\)-admissible for every \(\sig \dashv \rho\).
\item
The boundary \(\partial \Del^n \subseteq \Del^n\) of the \(n\)-simplex is \((\sig \dashv \rho)\)-admissible as soon as \(n\) is in the image of \(\sig\). We will use this case in the proof of Proposition~\ref{p:step-1} below.
\item
The 0-th horn \(\Lam^n_0 \subseteq \Del^n\) is \((\sig \dashv \rho)\)-admissible as soon as \(n\) is in the image of \(\sig\). We will use this case in the proof of Lemma~\ref{l:outer-anodyne} below.
\end{enumerate}
\end{examples}

\begin{define}\label{d:transformation} 
Let \((X,T_X)\) be a scaled simplicial set and \(A \subseteq \Del^n\) a
\((\sig \dashv \rho)\)\nbd-ad\-missible subsimplicial set. By a \emph{\((\sig \dashv \rho)\)-transformation} we will mean a map of simplicial sets \(h \colon \Del^1 \times A \to X\) which satisfies the following properties:
\begin{enumerate}
\item\label{i:trans-1}
Suppose that \(J \subseteq [1] \times [n]\) is such that \(J_0 \neq \emptyset\), \(\Del^{J_0} \subseteq A\), \(i:=\max(J_0)\) is in the image of \(\sig\) and \(J_1 = \{(1,i)\}\). Then \(h|_{\Del^J}\) factors through the retraction \(\Del^{J}\to \Del^{J \setminus \{(0,i)\}}\) which maps \((0,i)\) to \((1,i)\).
\item\label{i:trans-2} 
Suppose that \(J \subseteq [1] \times [n]\) is such that \(J_0 \neq \emptyset\), \(\Del^{J_0} \subseteq A\), and \(i:=\max(J_0)\) is not in the image of \(\sig\). Assume also that \(J_0\) contains \(i-1\), and \(J_1\) is either \(\emptyset\) or \(\{(1,i)\}\). Then \(h|_{\Del^J}\) factors through the retraction \(\Del^{J}\to \Del^{J \setminus \{(0,i)\}}\) which maps \((0,i)\) to \((0,i-1)\).
\item\label{i:trans-3}
For each edge \(\Del^{\{i,j\}}\subseteq A\) the triangle \(h(\Del^{\{(0,i),(1,i),(1,j)\}})\) is thin in \(\C\). 
\end{enumerate}
\end{define}

\begin{rem}\label{rem:trans_deg_edges}
We make the following observations:
\begin{enumerate}
  \item Condition~\eqref{i:trans-1} implies that the map \(h\) sends the edge \(\Del^1 \times \{\sig(k)\} \subseteq \Del^1 \times A\) to a degenerate edge in \(X\) for every \(k \in [m]\).
  \item Given a \((\sig \dashv \rho)\)-transformation \(h \colon \Del^1 \times A \to X\) and any map \(f\colon (X,T_X) \to (Y,T_Y)\) of scaled simplicial sets, the composite \(f\circ h \colon \Del^1 \times A \to Y\) is still a \((\sig \dashv \rho)\)-transformation.
\end{enumerate}

\end{rem}

\begin{rem}\label{rem:transdeg_triangles}
Given an \(\infty\)-bicategory \(\C := (\ovl{\C},T_{\C})\), the notion of a \((\sig \dashv \rho)\)-trans\-for\-ma\-tion \(h\colon\Del^1 \times A \to \ovl{\C}\) is weaker than that of a natural transformation in \(\C\), as \(h\) may fail to extend to a map of scaled simplicial sets \(\Del^1_{\flat} \times A_{\flat} \to \C\). Nonetheless, by Condition~\eqref{i:trans-3} \(h\) sends every triangle of the form \(\Del^{\{(0,i),(1,i),(1,j)\}}\) to a thin triangle in \(\C\), and hence extends to a scaled map \(\Del^1 \mgr A \to \C\). We consequently consider it as a \emph{lax natural transformation}, a concept we will investigate further in~\cite{GagnaHarpazLanariGrayLaxFunctors} and~\cite{GagnaHarpazLanariLaxLimits}. We note that this lax natural transformation is of a rather special kind, due to the degeneracy conditions imposed by Conditions~\eqref{i:trans-1} and~\eqref{i:trans-2}. These conditions imply, in particular, that for an edge \(\Del^{\{i,j\}} \subseteq A\), if either:
\begin{enumerate}
   \item \(j\) is in the image of \(\sig\), or
   \item \(j=i+1\) and \(\rho(i)=\rho(j)\) 
 \end{enumerate}
 then the triangle \(h(\Del^{\{(0,i),(0,j),(1,j)\}})\) is degenerate and hence thin as well, so that \(h|_{\Del^1 \times \Del^{\{i,j\}}}\) extends to a natural transformation \(\Del^1_{\flat} \times \Del^{\{i,j\}}_{\flat} \to \C\). Note that the equality \(\rho(i)=\rho(j)\) in the second point above means in particular that \(j\) does not belong to the image of \(\sigma\).
\end{rem}

\begin{notate}\label{n:L-A}
Given a \((\sig \dashv \rho)\)-admissible simplicial set  \(A\subseteq \Del^n\), we shall denote by \(L_A\) the collection of triangles of
\(\Del^1 \times A\) which are either of the form \(\Del^{\{(0,i),(1,i),(1,j)\}}\) for \(i < j \in [n]\) or of the form \(\Del^{\{(0,i),(0,j),(1,j)\}}\) if either \(j\) is in the image of \(\sig\) or \(j=i+1\) and \(\rho(i)=\rho(j)\). In particular, if \((X,T_X)\) is a scaled simplicial set then by Remark~\ref{rem:transdeg_triangles} every \((\sig \dashv \rho)\)-transformation \(h\colon \Del^1 \times A \to X\) extends to a map of scaled simplicial sets \((\Del^1 \times A, L_A) \to (X,T_X)\).
\end{notate}

\begin{rem}\label{r:degenerate-moving}
By induction, Condition~\eqref{i:trans-2} implies, for example, that for every \(j \in [m]\) the simplex \(g|_{\Del^{\{0\} \times \rho^{-1}(j)}}\) degenerates to a point, provided \(\mathrm{max}(\rho^{-1}(j))\in A\). Note that this last condition is equivalent to having \(\Delta^{\rho^{-1}(j)} \subset A\) since \(A\) is \((\sig \dashv\rho)\)-admissible. We warn the reader however that the map \(g|_{\{0\} \times A}\) does not in general factor through the image of \(A\) in \(\Del^m\).
\end{rem}

\begin{rem}\label{ob:filtration}
Let \(\sig \dashv \rho\) be as above and assume that \(n \geq 2\) and that \(n\) is in the image of \(\sig\) (in particular, both \(\Del^n\) and \(\partial \Del^n\) are \((\sig \dashv \rho)\)-admissible, see Examples~\ref{ex:admissible}).
We may then consider the set of triangles \(L_{\Del^n}\) of \(\Del^1 \times \Del^n\) defined in Notation~\ref{n:L-A}. Since \(n \geq 2\) all these triangles are contained in \(\Del^1 \times \partial \Del^n\), and so we also have \(L_{\partial \Del^n} = L_{\Del^n}\).
Consider the filtration
\[
[(\Del^1 \times \partial \Del^{n},L_{\Del^n})] \coprod_{\Del^{\{0\}}\times \partial \Del^{n}_{\flat}}[\Del^{\{0\}}\times \Del^{n}_{\flat}] = Z^{n+1} \subseteq Z^{n} \subseteq \dots \subseteq Z^0 = (\Del^1 \times \Del^{n},L_{\Del^n})
\]
whose underlying filtration of simplicial sets is as in the filtration~\eqref{e:filtration-0} of Construction~\ref{cn:filtration}, and the scaling on \(Z^k\) is given by restricting \(L_{\Del^n}\). 
Since \(L_{\Del^n}\) contains every triangle of the form \(\Del^{\{(0,i),(0,i+1),(1,i+1)\}}\) (because either \(i+1\) is in the image of \(\sigma\) or we must have \(\rho(i) = \rho(i+1)\)), we have that the inclusion \(Z^{k+1} \subseteq Z^k\)
is a pushout of a map in the set \eqref{item:scaled_anodyne_i} for \(k \in \{1,\dots,n\}\), so it is an anodyne map.
In the last step of the filtration, since \(L_{\Del^n}\) contains all the triangles \(\Del^{\{(0,0),(1,0),(1,i)\}}\) we obtain a pushout square of the form
\[\begin{tikzcd}
(\Lam^{n+1}_0,S_0) \ar[r]\ar[d] & Z^1 \ar[d] \\
(\Del^{n+1}_0,S_0) \ar[r] & Z^0 
\end{tikzcd} \]
where \(S_0\) is the set of all triangles which contain the edge \(\Del^{\{0,1\}}\). Now suppose that \(\C = (\ovl{\C},T_{\C})\) be a weak \(\infty\)-bicategory and \(h\colon \Del^1 \times \partial \Del^n\to \ovl{\C}\) be a \((\sig \dashv \rho)\)-transformation.
It then follows that any extension of \(h|_{\Del^{\{0\}} \times \partial \Del^n}\) to \(\{0\} \times \Del^n\) can be prolonged along the above filtration all the way to
\(\Del^1 \times \Del^n\): indeed, in the last step of the filtration the extension exists since 
the edge \(h_{|\Del^1\times \{0\}}\) must be degenerate by Remark~\ref{rem:trans_deg_edges} and the map \((\Lam^{n+1}_0,S_0) \coprod_{\Del^{\{0,1\}}_{\flat}} \Del^0 \to (\Del^{n+1}_0,S_0)_{\Del^{\{0,1\}}_{\flat}} \Del^0\) is scaled anodyne (it is a pushout of a generating scaled anodyne map of type~\ref{item:scaled_anodyne_iii}). 
\end{rem}

The following is the key lemma of this section. Unlike the situation in Remark~\ref{ob:filtration}, it allows us to extend an \((\sig \dashv \rho)\)-transformations given an extension of its value at \(1\) (as opposed to \(0\)). The combination of Remark~\ref{ob:filtration} and Lemma~\ref{l:moving-lemma} is what gives the notion of an \((\sig \dashv \rho)\)-transformation its power in practice. 

\begin{lem}[The moving lemma]\label{l:moving-lemma}
Let \((\sig \dashv \rho)\) be as above and let \(A \subseteq B \subseteq \Del^n\) be an inclusion of \((\sig \dashv \rho)\)-admissible subsimplicial sets of \(\Del^n\). Let \(\C = (\ovl{\C},T_{\C})\) be a scaled simplicial set which satisfies the right lifting property with respect to scaled inner horn inclusions (that is, the horn inclusions of type~\ref{item:scaled_anodyne_i} in Definition~\ref{anodyne defi}). 
Suppose that we are given a map 
\[ g\colon \Del^1 \times A \coprod_{\Delta^{\{1\}} \times A} \Delta^{\{1\}} \times B \to \ovl{\C} \]
whose restriction to \(\Del^1 \times A\) is a \((\sig \dashv \rho)\)-transformation. 
Then \(g\) extends to a \((\sig \dashv \rho)\)-transformation \(h\colon\Del^1 \times B \to \ovl{\C}\). 
\end{lem}

\begin{rem}\label{r:gray}
Considering for simplicity the case where \(A\) is empty, 
Lemma~\ref{l:moving-lemma} enables one to take a given map \(B \to \ovl{\C}\) and modify it in such a way that it becomes degenerate in certain specific ways (see, e.g., Remark~\ref{r:degenerate-moving}). We will use this ``moving trick'' in the proofs of Proposition~\ref{p:step-1} and Lemma~\ref{l:outer-anodyne} below.
\end{rem}

\begin{rem}
Though we formulate Lemma~\ref{l:moving-lemma} in a rather generic manner, we will only apply it in cases where \(B\) is either \(\Del^n, \partial \Del^n\) or \(\Lam^n_0\) and \(A\) is either \(\emptyset,\partial \Del^n\) or~\(\Lam^n_0\). The reader who wishes to have a concrete picture in mind while reading the proof below is invited to take \(A=\emptyset\) and \(B=\Del^n\).
\end{rem}

\begin{proof}[Proof of Lemma~\ref{l:moving-lemma}]
Let \(\B\) be the collection of non-degenerate simplices of \(\Del^1 \times B\) which are not contained in \(\Del^1 \times A \coprod_{\Delta^{\{1\}} \times A} \Delta^{ \{1\}} \times B\). We will follow the notation for non-degenerate simplices introduced in Definition~\ref{d:admissible}.
Let \(\B_0 \subseteq \B\) be the subset of those simplices \(J \subseteq [1] \times [n]\) belonging to one of the following mutually exclusive cases:
\begin{enumerate}[label=(\roman{*}), ref=(\roman{*})]
\item
The index \(m(J)\) belongs to the image of \(\sig\) and \(J_1 = \{m(J)\}\). 
\item
The index \(m(J)\) does not belong to the image of \(\sig\), \(J\) contains \((0,m(J)-1)\) and \(J_1 \subseteq \{m(J)\}\).
\item
\(|J_1| \geq 2\) and \(J\) contains the edge \(\{(0,m(J)),(1,m(J))\}\).
\end{enumerate}
Choose a total ordering on \(\B_0\) such that \(J' < J\) whenever the dimension of \(J'\) is smaller than that of \(J\), or whenever they have the same dimension but the index of \(J'\) is smaller than that of \(J\), or whenever they have the same dimension and the same index but \(|J'_1| < |J_1|\). More precisely, we arbitrarily extend to a total order that set of relations, by choosing a total order on those simplices who have the same dimension, index and cardinality of vertices at height \(1\). For \(J \in \B_0\) let \(Z_{<J} \subseteq \Del^1 \times B\) be the subsimplicial set given by the union of \(\Del^1 \times A \coprod_{\Delta^{\{1\}} \times A} \Delta^{\{1\}} \times B\) with all the simplices \(J' \in \B_0\) such that \(J' < J\), and similarly let \(Z_{\leq J} \subseteq \Del^1 \times B\) be the subsimplicial set given by the union of \(\Del^1 \times A \coprod_{\Delta{\{1\}} \times A} \Delta^{\{1\}} \times B\) with all the simplices \(J' \in \B_0\) such that \(J' \leq J\).

We note that out of the simplices in \(\B\), Conditions (1)--(3) of Definition~\ref{d:transformation} only concern simplices which are in \(\B_0\) (more specifically, Condition (1) concerns simplices of type (i), Condition (2) simplices of type (ii) and Condition (3) simplices of type (iii)). We now show by induction that for every \(J \in \B_0\) the map \(g\) extends to a map \(g_{\leq J}\colon Z_{\leq J} \to \ovl{\C}\) such that Conditions (1)--(3) hold for the simplices in \(Z_{\leq J}\). We note that every simplex in \(\B\) which is not in \(\B_0\) is a face of a simplex in \(\B_0\) and so \(Z_{\leq J} = \Del^1 \times B\) when \(J\) is the maximal element of \(\B_0\).

Let now \(J\) be a simplex in \(\B_0\) and assume we have already extended the map \(g\) to a map \(g_{< J}\colon Z_{<J} \to \ovl{\C}\) in a way that Conditions (1)--(3) of Definition~\ref{d:transformation} hold. To extend \(g_{< J}\) to a map \(g_{\leq J}\colon Z_{\leq J} \to \ovl{\C}\) we deal with each of the cases (i)--(iii) separately:
\begin{enumerate}[label=(\roman{*}), ref=(\roman{*})]

\item
Suppose that \(J\) is a simplex of type (i) and index \(i := m(J)\). We first note that \(Z_{<J}\) already contains all the maximal faces of \(J\) except the face opposite to the vertex \((1,i)\). Indeed, the faces opposite to vertices of the form \((0,j)\) for \(j < i\) are either in \(\Del^1 \times A\) or are simplices in \(\B_0\) of type (i) of lower dimension. The face opposite to the vertex \((0,i)\) is either in \(\Delta^{\{1\}} \times B\) or is a maximal face of a simplex \(J'\) in \(\B_0\) of type (iii) and of index \(j = \max(J_0 \setminus \{i\}) < i\), and is hence contained in \(Z_{<J}\). On the other hand, the face opposite \((1,i)\) can neither belong to \(\Del^1 \times A\) nor to \(\Delta^{\{1\}} \times B\), does not belong to \(\B_0\), and is a maximal face of a unique simplex in \(\B_0\), namely \(J\). It is hence not contained in \(Z_{< J}\). 
In summary, we have a pushout square of simplicial sets of the form:
\[\xymatrix{
\Lam^J_{(1,i)} \ar[r]\ar[d] & Z_{< J} \ar[d] \\
\Del^J \ar[r] & Z_{\leq J} \\
}\]
To respect Condition~\eqref{i:trans-1} we now define \(g_{\leq J}\colon Z_{\leq J} \to \ovl{\C}
\) by mapping \(\Del^J\) in a degenerate manner via the retraction \(\Del^J \to \Del^{J \setminus \{(0,i)\}}\) which maps \((0,i)\) to \((1,i)\) (so that \((g_{\leq J})|_{\Del^{J}}\) is determined by \((g_{< J})|_{\Del^{J \setminus \{(0,i)\}}}\)). We note that this definition is compatible with all the faces which are already in \(Z_{<J}\): this is automatically true for the face opposite \((0,i)\), and holds for the other ones thanks to our assumption that \(g_{< J}\) satisfies Condition~\eqref{i:trans-1}.

\item
Suppose that \(J\) is a simplex of type (ii) and index \(i := m(J)\). 
Then we claim that \(Z_{<J}\) already contains all the maximal faces of \(J\) except the one opposite to the vertex \((0,i-1)\). 
Indeed, the faces opposite to vertices of the form \((0,j)\) for \(j < i-1\) are either in \(\Del^1 \times A\) or are simplices in \(\B_0\) of type (ii) of lower dimension. The face opposite \((1,i)\) (if \((1,i) \in J\)) is itself a simplex of type (ii) and a lower dimension.
For the face opposite the vertex \((0,i)\), if \(J_1=\{(1,i)\}\) then this face is also a maximal face of a simplex in \(\B_0\) of type (iii) and of index \(i-1\). If \(J_1=\emptyset\) then we argue as follows: if \(i-1\) is in the image of \(\sigma\) then we view this face as the one opposite to \((1,i-1)\) in \(J\setminus\{(0,i)\}\cup\{(1,i-1)\}\), which is of type (i) and smaller index than \(J\).
If \(i-1\) is not in the image of \(\sigma\) and \(i-2\in J_0\), then we view this face as the one opposite the vertex \((1,i-1)\) of \(J\setminus\{(0,i)\}\cup \{(1,i-1)\}\), which is of type (ii) and smaller index than \(J\).
Finally, if \(i-1\) is not in the image of \(\sigma\) and \(i-2 \notin J_0\), then we view this face as the one opposite to the vertex \((0,i-2)\) in \(J\setminus \{(0,i)\}\cup\{(0,i-2)\}\), which is of type (ii) and a smaller index than \(J\). 

On the contrary, the face opposite to \((0,i-1)\) cannot belong to \(\Delta^{\{1\}} \times B\) nor to \(\Del^1 \times A\) (indeed, since \(A\) is \((\sig \dashv \rho)\)-admissible, if \(A\) contained the simplex spanned by \(J_0 \setminus \{i-1\}\) then it would contain \(J_0\)), and is not an element of \(\B_0\). If we consider which simplices in \(\B_0\) other than \(J\) contain this face as a maximal face we see that they must either be of type (iii) and the same index as \(J\) or of type (ii) and a bigger index. They are hence necessarily bigger than \(J\) in the total order we chose, and hence this face is not contained in \(Z_{<J}\).
In summary, there is a pushout square of simplicial sets of the form:
\[\xymatrix{
\Lam^J_{(0,i-1)} \ar[r]\ar[d] & Z_{< J} \ar[d] \\
\Del^J \ar[r] & Z_{\leq J} \\
}\]
To respect Condition~\eqref{i:trans-2} we now define \(g_{\leq J}\colon Z_{\leq J} \to \ovl{\C}\) by mapping \(\Del^J\) in a degenerate manner via the retraction \(\Del^J \to \Del^{J \setminus \{(0,i)\}}\)
 which maps \((0,i)\) to \((0,i-1)\). We note that this definition is compatible with all the faces which are already in \(Z_{<J}\): this is automatically true for the face opposite \((0,i)\), and holds for the other ones thanks to our assumption that \(g_{< J}\) satisfies Condition~\eqref{i:trans-2}.

\item
Suppose that \(J\) is a simplex of type (iii) and index \(i := m(J)\). We first note that \(Z_{<J}\) already contains all the maximal faces  of \(J\) except the face opposite to the vertex \((1,i)\). Indeed, the faces opposite vertices of the form \((0,j)\) for \(j < i\) are either in \(\Del^1 \times A\) or are simplices of type (iii) of a smaller dimension, and the same holds for faces opposite vertices of the form \((1,j)\) for \(j > i\) such that \((1,j)\) is not maximal in \(J\). The face opposite to the vertex \((0,i)\) is either in \(\Delta^{\{1\}} \times B\) or is a maximal face of a simplex in \(\B_0\) of type (iii) and of index \(i' = \max(J_0 \setminus \{i\}) < i\), and is hence contained in \(Z_{<J}\). Call \((1,k)\) the maximal vertex in \(J\), then the face opposite \((1,k)\) is either in \(\Del^1 \times A\) or in \(\B_0\) or it must be the case that the index \(i\) is not in the image of \(\sig\), \(J_0\) does not contain \(i-1\) and \(J_1 = \{(1,i),(1,k)\}\). Indeed, if \(\vert J_1\vert >2\), then the face opposite to the vertex \((1,k)\) would still be of type (iii), and thus in \(\B_0\). This forces \(J_1 = \{(1,i),(1,k)\}\), which in turn implies that the index \(i\) is not in the image of \(\sig\) and \(J_0\) does not contain \(i-1\), otherwise the face opposite to the vertex \((1,k)\) would be, respectively, of type (i) or (ii). Given this, the face opposite \((1,k)\) is the maximal face of the simplex \((J \setminus \{(1,k)\}) \cup \{(0,i-1)\}\), which belongs to \(\B_0\) since \(B\) is \((\sig \dashv \rho)\)-admissible, is of the same index and dimension as \(J\) but is still smaller than \(J\) with respect to our linear order since its intersection with \(\{1\} \times [n]\) has less elements than that of \(J\). 
We may then conclude that this face is also contained in \(Z_{<J}\). On the other hand, the face opposite \((1,i)\) is not in \(\B_0\) and 
is the maximal face of at most two simplices in \(\B_0\), the simplex \(J\) and potentially one more simplex \(J'\) whose index in this case is \(\min(J_1 \setminus \{i\}) > i\).
This face is consequently not contained in \(Z_{< J}\). 
We may thus conclude that we have a pushout square of simplicial sets of the form
\[\xymatrix{
\Lam^J_{(1,i)} \ar[r]\ar[d] & Z_{< J} \ar[d] \\
\Del^J \ar[r] & Z_{\leq J} \\
}\]
where the vertical map is an inner horn inclusion since \((1,i)\) is not maximal nor minimal in \(J\). Now if \(|J| = 3\) then, by our assumption on \(\C\), we can extend \(g_{< J}\) to \(g_{\leq J}\) in such a way that \(\Del^J\) is sent to a thin triangle, 
thus assuring that Condition~\eqref{i:trans-3} continues to hold for \(g_{\leq J}\). On the other hand, if \(|J| > 3\) then, since we assumed that \(g_{< J}\) satisfies Condition~\eqref{i:trans-3}, we have that \(g_{< J}\) maps \(\Del^{\{(0,i),(1,i),(1,l)\}}\) to a thin triangle, where \(l = \min(J_1 \setminus \{i\})\). By our assumption on \(\C\) we can then extend \(g_{< J}\) to a map \(g_{\leq J}\colon Z_{\leq J} \to \ovl{\C}\).\qedhere
\end{enumerate} 
\end{proof}

\begin{cor}\label{c:moving-scaled}
Let \((\sig \dashv \rho)\) and \(A \subseteq B \subseteq \Del^n\) be as
in the previous lemma and \(L_A,L_B\) as in Notation~\ref{n:L-A}. Let \(\C = (\ovl{\C},T_{\C})\) be a scaled simplicial set which satisfies the right lifting property 
with respect to the maps of type~\ref{item:scaled_anodyne_i} in Definition~\ref{anodyne defi}. 
Suppose that we are given a map 
\[ g\colon (\Del^1 \times A,L_A) \coprod_{\Delta^{\{1\}} \times A_{\flat}} \Delta^{\{1\}} \times B_{\flat} \to \C \]
such that the underlying simplicial map of its restriction to \(\Del^1 \times A\) is a \((\sig \dashv \rho)\)-transformation. 
Then \(g\) extends to a map of scaled simplicial sets
\(h\colon (\Del^1 \times B, L_B) \to \C\)
whose underlying simplicial map is a \((\sig \dashv \rho)\)-transformation. 
\end{cor}

\begin{proof}
Combine the Lemma~\ref{l:moving-lemma} and Remark~\ref{rem:transdeg_triangles}.
\end{proof}

\section{Weak \pdfoo-bicategories are \pdfoo-bicategories}\label{sec:weak-is-strong}
Our goal in this section is to prove the following result:
\begin{thm}\label{t:fibrant}
	Let \(\C\) be a weak \(\infty\)-bicategory. Then \(\C\) is an \(\infty\)-bicategory, \ie, \(\C\) is fibrant in \(\Set^{\sca}_\Del\).
\end{thm}

We advance towards Theorem~\ref{t:fibrant} in a sequence of lemmas. We begin with an \(\infty\)-bicategorical extension of Joyal's result on special outer horns~\cite[Theorem 3.4]{JoyalQCatsAndKan}.

\begin{lem}[Special outer horns]\label{l:joyal}
	Let \(\C\) and \(\D\) be weak \(\infty\)-bicategories and let 
	\[p\colon \C \to \D\] be a map which satisfies the right lifting property with respect to maps of type~\ref{item:scaled_anodyne_i} and \ref{item:scaled_anodyne_iii} in Definition~\ref{anodyne defi}. 
	For \(n \geq 2\) let \(T\) be the collection of all triangles in \(\Del^n\) which are either degenerate or contain the edge \(\Del^{\{0,1\}}\), and let \(T_0 \subseteq T\) be the subset of those triangles which are contained in \(\Lam^n_0\). Then the dotted lift exists in any square of the form
	\begin{equation}\label{e:weak-problem} 
	\begin{tikzcd}
		(\Lam^n_0,T_0) \ar[r, "f"]\ar[d] & \C \ar[d, "p"] \\
		(\Del^n,T) \ar[ur, dotted] \ar[r, "g"] & \D
  \end{tikzcd}
	\end{equation}
	such that \(f|_{\Del^{\{0,1\}}}\) is an equivalence in \(\C\).
\end{lem}

\begin{proof}
	Let \(f_0 := f|_{\partial \Del^{\{2,\dots,n\}}}\) and let \(f_1:= f|_{\Del^{\{2,\dots,n\}}}\). By Corollary~\ref{c:slice-3} the projection 
	\[ q\colon \ovl{\C}_{/f_1} \to \ovl{\C}_{/f_0} \times_{\ovl{\D}_{/pf_0}}\ovl{\D}_{/pf_1} \] 
	is an outer fibration whose base change to the marked core of \(\C_{/f_0} \times_{\D_{/pf_0}}\D_{/pf_1}\) is an outer cartesian fibration.
	Unwinding the definitions, we see that finding a solution to the lifting problem~\eqref{e:weak-problem} is equivalent to finding a lift in a diagram of marked-scaled simplicial sets of the form
	\begin{equation}\label{e:weak-problem-2} 
	\begin{tikzcd}
		\Del^{\{0\}} \ar[r]\ar[d] & \C_{/f_1} \ar[d, "q"] \\
		(\Del^1,\{\Del^1\},\emptyset) \ar[r, "\eta"] & \C_{/f_0} \times_{\D_{/pf_0}}\D_{/pf_1}
  \end{tikzcd}
	\end{equation}
	where \(\eta\) is a marked edge of \(\C_{/f_0} \times_{\D_{/pf_0}}\D_{/pf_1}\) whose image in \(\C\) is an equivalence. By Corollary~\ref{c:slice} the map
	\[q_{\C}\colon\ovl{\C}_{/f_0} \to \C\] 
	is an outer cartesian fibration and the edge \(e_{\eta}\) of \(\ovl{\C}_{/f_0}\) determined by the corresponding component of \(\eta\) is \(q_\C\)-cartesian (since it is marked in \(\C_{/f_0}\)). By Remark~\ref{r:weak-scaled} we have that \(\ovl{\C}_{/f_0}\) is a weak \(\infty\)-bicategory, and by applying Remark~\ref{r:cart-equiv} to the map \(q_{\C}\) we see that the edge \(e_{\eta}\) is an equivalence. Now consider the pullback square
\[
\xymatrix{
\ovl{\C}_{/f_0} \times_{\ovl{\D}_{/pf_0}}\ovl{\D}_{/pf_1} \ar[r]\ar[d]^{r} & \ovl{\D}_{/pf_1} \ar[d]^{r'} \\
\ovl{\C}_{/f_0} \ar[r] & \ovl{\D}_{/pf_0}
}\]
By	Corollary~\ref{c:slice-3} the right vertical map \(r'\) is an outer fibration and the marked arrows of \(\D_{/pf_1}\) are \(r'\)-cartesian. By Remark~\ref{r:base-change} we then have that the left vertical map \(r\) is also an outer fibration and every marked edge in \(\C_{/f_0} \times_{\D_{/pf_0}} \D_{/pf_1}\) is \(r\)\nbd-car\-tesian. In particular, \(\eta\) is a \(r\)-cartesian lift of \(e_{\eta}\). By Remark~\ref{r:weak-scaled} we then have that \(\ovl{\C}_{/f_0} \times_{\ovl{\D}_{/pf_0}}\ovl{\D}_{/pf_1}\) is a weak \(\infty\)-bicategory, and by Remark~\ref{r:cart-equiv} we have that \(\eta\) is an equivalence.

	Now the marked edges in \(\C_{/f_0} \times_{\D_{/pf_0}}\D_{/pf_1}\) are closed under composition. Indeed, for \(n \geq 3\), any edge of \(\D_{/pf_1}\) mapping to
  a marked edge in \(\D_{pf_0}\) is already marked and so
  we have that the marked edges of \(\C_{/f_0} \times_{\D_{/pf_0}}\D_{/pf_1}\) are precisely those that map to marked edges in \(\C_{/f(i)}\) for every \(i=2,...,n\). When \(n=2\), they are precisely those that map to marked edges in \(\D_{/pf(2)}\), and so closure under composition follows from Remark~\ref{r:closure}.  
	The marked core of \(\C_{/f_0} \times_{\D_{/pf_0}}\D_{/pf_1}\) is consequently also a weak \(\infty\)-bicategory. 
Since the base change of \(q\) to the marked core of  \(\C_{/f_0} \times_{\D_{/pf_0}}\D_{/pf_1}\) is an outer cartesian fibration, this base change is also an isofibration by Remark~\ref{r:iso}. We may thus conclude that the dotted lift in~\eqref{e:weak-problem-2} exists, as desired.
\end{proof}

\begin{prop}\label{p:step-1}
	Let \(p\colon\C \rightarrow \D\) be a bicategorical equivalence of weak \(\infty\)-bicategories and consider a square of the form
	\[ \xymatrix{
		\partial \Del^n_{\flat} \ar[r]^{f}\ar[d] & \C \ar[d]^{p} \\
		\Del^n_{\flat} \ar[r]^{g} & \D \\
	}\]
	with \(n \geq 1\). Then there exists an extension \(\ovl{f}\colon\Del^n_{\flat} \rightarrow \C\) and a natural transformation \(T\colon \Del^1_{\flat} \times \Del^n_{\flat} \rightarrow \D\) from \(g\) to \(p\ovl{f}\) relative to \(\partial \Del^n_\flat\) 
	(see Definition \ref{d: pointwise nat transf}).
\end{prop} 
\begin{proof}
If \(n=1\) the claim amounts to the induced map of (naturally marked) \(\infty\)-categories \(\Hom_{\C}(x,y) \to \Hom_{\D}(p(x),p(y))\) being essentially surjective for every \(x,y \in \C\): indeed, a marked arrow in \(\Hom_{\D}(p(x),p(y))\) from \(\alp\) to \(\beta\) precisely corresponds to a natural transformation \(\Del^1_{\flat} \times \Del^1_{\flat} \to \D\) relative to \(\partial \Del^1_{\flat}\) with source \(\alp\) and target \(\beta\). 
We may hence assume that \(n \geq 2\). 
We now proceed in several steps as follows:
\smallskip \\ {[\textbf{Step 1}]} \smallskip \\
We begin by ``moving'' \(f\) so that it sends all the vertices except the last one to \(x\). More precisely, let \(\ovl{\C}\) be the underlying simplicial set of \(\C\). Applying Corollary~\ref{c:moving-scaled} 
with \(A=\emptyset, B=\partial \Del^n\) and \(\rho\colon\Del^n\to \Del^1\) the map which sends \(\{0,\dots,n-1\}\) to \(0\) and \(n\) to \(1\) (in which case \(A\) and \(B\) are indeed \((\sig \dashv \rho)\)-admissible, see Examples~\ref{ex:admissible}) we find a map of scaled simplicial sets
\begin{equation}\label{e:h} 
h\colon (\Del^1 \times \partial\Del^n,L_{\partial \Del^n}) \to \C ,
\end{equation}
whose underlying simplicial map \(\Del^1 \times \partial\Del^n \to \ovl{\C}\) is a \((\sig \dashv \rho)\)-transformation, and such that \(h|_{\{1\} \times \partial\Del^n_{\flat}} = f\) (here the scaling \(L_{\partial \Del^n}\) is as in Notation~\ref{n:L-A}).
In addition, by Remark~\ref{rem:trans_deg_edges}, we have that \(h|_{\Del^1 \times \{0\}}\) and \(h|_{\Del^1 \times \{n\}}\) are degenerate on \(x\) and \(y\) respectively, and \(h|_{\{0\} \times \Del^{\{0,\dots,n-1\}}_{\flat}}\) is degenerate on \(x\) (see Remark~\ref{r:degenerate-moving}). Let us set \(f' := h|_{\{0\} \times \partial\Del^n_{\flat}}\).
The map \(f'\) then determines a commutative square of scaled simplicial sets
\begin{equation}\label{e:C-over-y}
\begin{tikzcd}
\partial\Del^{\{0,\dots,n-1\}}_{\flat} \ar[r]\ar[d] & \ovl{\C}_{/y} \ar[d, "{\pi_{\C}}"] \\
\Del^{\{0,\dots,n-1\}}_{\flat} \ar[r] & \C
\end{tikzcd}
\end{equation}
in which \(y= h(0,n)\) and the bottom horizontal map is given by the restriction of \(f'\) to \(\{0\} \times \Del^{\{0,\dots,n-1\}}_{\flat}\), and is hence constant with image the vertex \(x\). In particular, the square~\eqref{e:C-over-y} determines a map \(f'_x\colon \partial \Del^{\{0,\dots,n-1\}}_{\flat} \to (\ovl{\C}_{/y})_x\).
\smallskip \\ {[\textbf{Step 2}]} \smallskip \\
Projecting down to \(\D\) we may now consider the map
\begin{equation}\label{e:map-to-D} 
(\Del^1 \times \partial\Del^n,L_{\partial \Del^n}) \coprod_{\Del^{\{1\}} \times \partial \Del^n_{\flat}} \Del^{\{1\}} \times \Del^n_{\flat} \to \D 
\end{equation}
determined by \(ph\) and \(g\). 
Applying again Corollary~\ref{c:moving-scaled} with respect to \(A=\partial\Del^n\)
and \(B=\Del^n\) we can extend~\eqref{e:map-to-D} to a map
\[ H\colon (\Del^1 \times\Del^n,L_{\Del^n}) \to \D, \]
whose underlying simplicial map is a \((\sig \dashv \rho)\)-transformation, so that, in particular, \(H\) maps \(\{0\} \times \Del^{\{0,\dots,n-1\}}\) to \(p(x)\). If we then write \(g' := H|_{\{0\} \times \Del^n}\) for the restriction of \(H\) to \(\{0\} \times \Del^n\) we obtain that \(g'\) determines a map  
\(g'_x\colon \Del^{\{0,\dots,n-1\}} \to (\ovl{\D}_{/p(y)})_{p(x)}\) which fits in a commutative square
\begin{equation}\label{e:on-mapping-cats}
\begin{tikzcd}
\partial \Del^{\{0,\dots,n-1\}}_{\flat} \ar[r, "{f'_{x}}"] \ar[d] & (\ovl{\C}_{/y})_x  \ar[d, "{p_\ast}"] \\
\Del^{\{0,\dots,n-1\}}_{\flat} \ar[r, "{g'_{x}}"] & (\ovl{\D}_{/p(y)})_{p(x)}
\end{tikzcd}
\end{equation}
together with the map \(f'_{x}\colon \partial \Del^{\{0,\dots,n-1\}} \to (\ovl{\C}_{/y})_x
\) constructed at the end of the previous step.
\smallskip \\ {[\textbf{Step 3}]} \smallskip \\
We now invoke (for the first and only time) the assumption that \(p\colon\C \rightarrow \D\) is a bicategorical equivalence of weak \(\infty\)-bicategories. Proposition~\ref{p:mapping-1} now implies that the map \((\ovl{\C}_{/y})_{x} \rightarrow (\ovl{\D}_{/p(y)})_{p(x)}\) is an equivalence of weak \(\infty\)-bicategories in which every triangle is thin. It then follows that in the square of Kan complexes
\[ \xymatrix{
\Fun^{\simeq}(\Del^{\{0,\dots,n-1\}}_{\flat}, (\ovl{\C}_{/y})_x) \ar[r]\ar[d] & \Fun^{\simeq}(\Del^{\{0,\dots,n-1\}}_{\flat}, (\ovl{\D}_{/p(y)})_{p(x)}) \ar[d] \\
\Fun^{\simeq}(\partial\Del^{\{0,\dots,n-1\}}_{\flat}, (\ovl{\C}_{/y})_x) \ar[r] & \Fun^{\simeq}(\partial\Del^{\{0,\dots,n-1\}}_{\flat}, (\ovl{\D}_{/p(y)})_{p(x)}) \\
}\]
the horizontal maps are equivalences and the vertical maps are categorical fibrations, which implies that this square induces equivalences on vertical fibers. We may therefore conclude that the map \(f'_{x}\colon \partial\Del^{\{0,\dots,n-1\}}_{\flat} \to (\ovl{\C}_{/y})_x\) extends to a map 
\[ \ovl{f}'_{x}\colon \Del^{\{0,\dots,n-1\}}_{\flat} \to (\ovl{\C}_{/y})_x, \]  
such that we have a pointwise invertible natural transformation 
\[\eta_{x}\colon \Del^1_{\flat} \times \Del^{n-1}_{\flat} \to (\ovl{\D}_{/p(y)})_{p(x)}\] 
from \(g'_{x}\) to \(p_\ast\ovl{f}'_{x}\), 
whose restriction 
to \(\partial \Del^{n-1}_{\flat}\) is constant on \(p_\ast f'_{x}\) (which makes sense thanks to the commutativity of \eqref{e:on-mapping-cats}). 
The map \(\ovl{f}'_{x}\) then determines a map \(\ovl{f}'\colon \Del^n_{\flat} \to \C\), and upon composing with the map
\[  \Delta^1 \times \Delta^n \cong \Delta^1 \times (\Delta^{n-1}\ast \Delta^{\{n\}}) \to \left(\Delta^1 \times \Delta^{n-1}\right)\ast \Delta^{\{n\}}\]
given on vertices by \((i,j) \mapsto (i,j)\) for \((i,j)\in [1] \times [n-1]\) and \((i,n) \mapsto n\) for \(i \in [1]\) ,
the natural equivalence \(\eta_{x}\) determines a natural transformation \(\eta\colon \Del^1_{\flat} \times \Del^n_{\flat} \to \D\) from \(g'\) to \(p\ovl{f}'\) relative to \(\partial \Del^n_\flat\).
\smallskip \\ {[\textbf{Step 4}]} \smallskip \\
To facilitate the argument that follows let us introduce the shorthand notation 
\[X \colon= (\Del^1 \times \partial \Del^n,L_{\partial \Del^n}) \coprod_{\Del^{\{0\}} \times \partial \Del^n_{\flat}} \Del^{\{0\}} \times \Del^n_{\flat} \to (\Del^1 \times \Del^n,L_{\Del^n}) =\colon Y,\]
The map \(h\) of~\eqref{e:h} together with \(\ovl{f}'\) constructed in the previous step together determine a map \(\vphi \colon X \to \C\) which sends the edge \(\Del^1 \times \{0\}\) to a degenerate edge (since so does \(h\) by Remark~\ref{rem:trans_deg_edges}). Using the filtration of Remark~\ref{ob:filtration} we now see that the map \(X \coprod_{\Del^1_{\flat} \times \{0\}} \Del^0 \to Y \coprod_{\Del^1_{\flat} \times \{0\}} \Del^0\) is scaled anodyne and hence the map \(\vphi\) extends to a map \(\ovl{\vphi}\colon Y \to \C\). Restricting \(\ovl{\vphi}\) to \(\Del^{\{1\}} \times \Del^n\) we now obtain an extension of our original map \(f\colon \partial \Del^n_{\flat} \to \C\) to a map \(\ovl{f}\colon\Del^n_{\flat} \to \C\).
\smallskip \\ {[\textbf{Step 5}]} \smallskip \\
To finish the proof we now need to construct a natural transformation from \(g\) to \(p\ovl{f}\) relative to \(\partial \Del^n_\flat\). 
Consider the following two generally different maps 
\[p\ovl{\vphi},H \colon Y \to \D,\] 
whose restriction to \(X\) gives the two maps \(p\vphi,H|_{X}\colon X \to \D\). By construction, the map \(p\vphi\)
is determined by the pair \((ph, p\ovl{f}')\),
while the map \(H_{|X}\) is determined by the pair \((ph, g')\). 
These two maps are related via a levelwise invertible natural transformation \(\ovl{\eta} \colon \Del^1_{\flat} \times X \to \D\) glued from the natural transformation \(\eta\) on 
\(\Del^{\{0\}} \times \Del^n_{\flat} \subseteq X\)
and the constant natural transformation on \((\Del^1 \times \partial \Del^n_{\flat},L_{\partial \Del^n}) \subseteq X\) (from \(ph\) to itself). We may hence consider the resulting extension problem 
\begin{equation}\label{e:last-extension}
\begin{tikzcd}
\partial \Del^1 \times Y \displaystyle\mathop{\coprod}_{\partial \Del^1 \times X} \Del^1_{\flat} \times X \ar[d]\ar[rr, "{\bigl((H,p\ovl{\vphi}), \ovl{\eta}\bigr)}"] && \D \\
\Del^1_{\flat} \times Y \ar[dotted, urr] &&
\end{tikzcd}.
\end{equation}
By the above, the top horizontal map sends \(\Del^1_{\flat} \times \Del^1_{\flat} \times \{0\} \subseteq \Del^1_{\flat} \times X\) to a point. Since the map \(X \coprod_{\Del^1_{\flat} \times \{0\}} \Del^0 \to Y \coprod_{\Del^1_{\flat} \times \{0\}} \Del^0\) is scaled anodyne and scaled anodyne maps are closed under pushout-products with arbitrary inclusions~\cite[Proposition~3.1.8]{LurieGoodwillie} the extension problem~\eqref{e:last-extension} admits a solution \(\psi\colon \Del^1_{\flat}\times Y \to \D\). Restricting 
the natural transformation \(\psi\colon \Del^1_{\flat} \times Y \to \D\) to \(\Del^{\{1\}} \times \Del^n_{\flat} \subseteq Y\) now yields 
a natural transformation 
from \(g\colon \Del^n_{\flat} \to \D\) to \(p\ovl{f} \colon \Del^n_{\flat} \to \D\) relative to \(\partial \Del^n_\flat\), as desired.
\end{proof}

\begin{prop}\label{p:step-2}
	Let \(p\colon \C \rightarrow \D\) be a bicategorical equivalence of weak \(\infty\)-bicategories and let \(j\colon A \rightarrow B\) an injective map of scaled simplicial sets. Suppose that we are given a commutative diagram of the form
	\begin{equation}\label{e:diagram3}
	\xymatrix{
		\Del^{\{1\}} \times A \ar[d]\ar^-{f_0}[r] & \C \ar^{p}[d] \\
		\Del^{\{0\}} \times B \displaystyle\mathop{\coprod}_{\Del^{\{0\}} \times A }\Del^1_{\flat} \times A \ar^-{h_0}[r] & \D \\
	}
	\end{equation}
	such that \((h_0)|_{\Del^1 \times \{a\}}\) is an equivalence in \(\D\) for every \(a \in A\). Then~\eqref{e:diagram3} extends to a diagram of the form
	\begin{equation}\label{e:diagram4}
	\xymatrix{
		\Del^{\{1\}} \times A \ar[d]\ar[r] & \Del^{\{1\}}\times B  \ar^-{f}[r]\ar[d] & \C \ar^{p}[d] \\
		\Del^{\{0\}} \times B \displaystyle\mathop{\coprod}_{\Del^{\{0\}} \times A}\Del^1_{\flat} \times A \ar[r] & \Del^1_{\flat} \times B \ar^-{h}[r] & \D \\
	}
	\end{equation}
	whose external rectangle is~\eqref{e:diagram3} and such that \(h|_{\Del^1 \times \{b\}}\) is an equivalence in \(\D\) for every \(b \in B\).
\end{prop}
\begin{proof}
	Working simplex by simplex, it will suffice to prove the case where the map \(j\colon A \hrar B\) is one of the inclusions \(j_n\colon\partial \Del^n_{\flat} \subseteq \Del^n_{\flat}\) or the inclusion \(\Del^2_{\flat} \subseteq \Del^2_{\sharp}\). We then note that for \(j_0\colon\varnothing \subseteq \Del^0\) the desired statement is equivalent to \(p\) being essentially surjective. When \(j\) is the inclusion \(\Del^2_{\flat} \subseteq \Del^2_{\sharp}\) the horizontal maps in the left square of~\eqref{e:diagram4} are both isomorphisms on the underlying simplicial sets. In this case the result can be obtained by invoking the fact that \(p\) detects thin triangles (Proposition~\ref{p:detects}) and that the collection of thin triangles is closed under levelwise invertible natural transformations (Corollary~\ref{c:joyal-2}). We may hence assume that \(j=j_n\) for some \(n \geq 1\).
	
	Let
	\begin{equation}\label{e:filt-0-again} 
	[\Del^1_{\flat} \times \partial \Del^{n}_{\flat}] \coprod_{\Del^{\{0\}}\times \partial \Del^{n}_{\flat}}[\Del^{\{0\}}\times \Del^{n}_{\flat}] = Z^{n+1} \subseteq Z^{n} \subseteq \dots \subseteq Z^0 =\Del^1_{\flat} \times   \Del^{n}_{\flat}
	\end{equation}
	be the filtration~\eqref{e:filtration-0} of Construction~\ref{cn:filtration}.	
	Then the inclusions \(Z^{k+1} \subseteq Z^k\) are scaled anodyne for \(k \geq 1\) and so we may extend the map \(h_0\) along the filtration~\eqref{e:filt-0-again} all the way to a map \(h'\colon Z^1 \rightarrow \D\). 
	In the last filtration step we have a pushout diagram of the form
	\[ \xymatrix{
		(\Lam^{n+1}_0,{S_0}|_{\Lam^{n+1}_0}) \ar[r]\ar[d] & Z^1 \ar[d] \\
		(\Del^{n+1},S_0) \ar[r] & Z^0 \\
	}\]
	where the composed map \(\Del^{n+1} \rightarrow Z_0 \rightarrow \Del^n_{\flat} \times \Del^1_{\flat}\) is the simplex \(\tau_0\). Our assumption that \({h_0}|_{\{0\} \times \Del^1}\) is an equivalence in \(\D\) implies that \(h' \circ \tau_0\colon\Lam^{n+1}_0 \to \D\) sends \(\Del^{\{0,1\}}\) to an equivalence, and hence by Lemma~\ref{l:joyal} we may extend \(h'\colon Z^1 \rightarrow \D\) to a map 
	\[h''\colon Z^0 = \Del^1_{\flat} \times \Del^n_{\flat} \rightarrow \D.\] 
	We have thus constructed a natural transformation 
	from \({h_0}|_{\{0\} \times \Del^n_{\flat}}\) to some \(n\)-simplex \(\sig = h''\vert_{\Del^{\{1\}}\times \Delta^n}\), extending the given natural transformation \(h_0\) on \(\partial \Del^n_{\flat} \subseteq \Del^n_{\flat}\). In particular, \(\sig|_{\partial \Del_{\flat}^n} = pf_0\). Applying Proposition~\ref{p:step-1} to \(f_0\) and \(\sig\) we may conclude that there exists a natural transformation \(H\colon \Delta^1_{\flat}\times \Delta^n_{\flat} \to \D\) relative to \(\partial \Delta^n_{\flat}\) satisfying \(H\vert_{\Delta^{\{0\}}\times \Delta^n_{\flat}}  = \sig\) and \(H\vert _{\Delta^{\{1\}}\times \Delta^n_{\flat}} = p\tau\) for some \(\tau \colon \Delta^n_{\flat} \to \C\) such that \(\sig|_{\partial \Del^n} = f_0\).  
	In particular, \(H\) is levelwise invertible. We now construct a composite natural transformation \(H \circ h''\)  
	by solving the lifting problem 
  \[\xymatrix{
    \left(\Lambda^2_1\right)_{\flat} \ar^{(h'',H)}[rr]\ar[d] && \D^{\Delta^n_{\flat}} \ar[d] \\
    \Delta^2_{\sharp} \ar^{\sigma^1}[r]\ar@{.>}[urr]^{\Gamma} & \Delta^1 \ar^{h_0}[r] & \D^{\partial\Delta^n_{\flat}} \\
  }\]
where \(\sigma^1\) is the surjective map that sends \(1\) to \(1\), 
 and the right-hand side vertical map is the bicategorical fibration given by restriction along  
 \(\partial \Delta^n_{\flat} \to \Delta^n_{\flat}\) (recall that the bicategorical model structure is cartesian, see Remark~\ref{r:cartesian}).  
 We now obtain the desired extension by setting \(f = \tau\) and \(h = H\circ h'':= \Gamma_{\vert \Delta^{\{0,2\}}}\). 
\end{proof}

\begin{proof}[Proof of Theorem~\ref{t:fibrant}]
	Let \(\C\) be a weak \(\infty\)-bicategory. Choose a trivial cofibration \(p\colon\C \hrar \D\) where \(\D\) is fibrant. Then in particular \(\D\) is a weak \(\infty\)-bicategory. Applying Proposition~\ref{p:step-2} to the diagram
	\[ \xymatrix{
		\C \times \Del^{\{1\}} \ar^{\cong}[r]\ar[d] & \C \ar^{p}[d] \\
		\D \displaystyle\mathop{\coprod}_{\C \times \Del^{\{0\}}}\Del^1 \times \C  \ar[r] & \D \\
	}\]
	in which the bottom map restricts to the identity on \(\D\) and to the identify transformation from \(p\) to itself on \(\Del^1 \times \C\), we may conclude that \(\C\) is a deformation retract of \(\D\) and hence in particular fibrant.
\end{proof}

Since the notion of a bicategorical equivalence is invariant under the duality operator \(X \mapsto X^{\op}\),
Theorem~\ref{t:fibrant} implies that the notion of a weak \(\infty\)\nbd-bi\-cat\-egory is self-dual. We may summarize the situation as follows:
\begin{cor}
	\label{weak are strong}
	Let \(X\) be a scaled simplicial set. Then the following conditions are equivalent:
	\begin{enumerate}[leftmargin=*]
		\item
		\(X\) is a weak \(\infty\)-bicategory.
		\item
		\(X^{\op}\) is a weak \(\infty\)-bicategory.
		\item
		\(X\) is fibrant in \(\Ss\).
	\end{enumerate}
\end{cor}

\section{The Cisinski model structure for \pdfoo-bicategories}\label{sec:cis}

In this section, we give a different construction of Lurie's model structure for \(\infty\)\nbd-bi\-cat\-e\-gories on the category \(\Ss\) of scaled simplicial sets using the machinery of Cisinski--Olschok recalled in the appendix.
We choose as our subset of monomorphisms the set \(\bS\) the generating scaled anodyne maps of Definition \ref{anodyne defi}.
As interval, we choose \(J_{\sharp}\), whose underlying simplicial set is the \(1\)\nbd-cat\-egorical nerve of the free living groupoid on an invertible arrow (alternatively, it can be described as \(J=\mathrm{Cosk}_0(\{0,1\})\)), \ie the 0-coskeleton of the set with two elements, whose two non-degenerate 2-simplices are marked. The pair \((J_{\sharp},\{0,1\} \to J_{\sharp})\) is then indeed a cartesian cylinder object in the sense of Definition~\ref{d:cylinder}, see also Remark~\ref{r:easy-cylinder}.

\begin{define}
	We will call the \emph{Cisinski model structure} on \(\s^{\sca}\) the model structure of Theorem~\ref{thm:cisinski_model} associated to the set of maps \(\bS\) and the interval object \(J_{\sharp}\). Note that Assumption~\ref{a:standing} holds for \(\s^{\sca}\) thanks to Remark~\ref{alternative def} and Remark~\ref{r:reflective}.
\end{define}

By definition, the cofibrations of the Cisinski model structure on \(\Set^{\sca}_{\Del}\) are the monomorphisms, and the fibrant objects are the scaled simplicial sets \(X\) which admit extensions for the generating anodyne maps of Notation~\ref{nota:anodyne}.
The following result justifies our choice of an interval object:

\begin{prop}
	\label{interval J}
	The inclusions \(i_0\colon\{0\}\rightarrow J_{\sharp}\) and \(i_1\colon\{1\}\rightarrow J_{\sharp}\) are \(\infty\)-bicategorical equivalences of \(\infty\)-bicategories.
\end{prop}

\begin{proof}
	It will suffice to show that the terminal map \(J_{\sharp} \to \Del^0\) is a trivial fibration, i.e., satisfies the right lifting property with respect to all inclusions of scaled simplicial sets. Since every triangle in \(J_{\sharp}\) is thin it will suffice to check that the underlying simplicial set \(J\) satisfies the right lifting property with respect to all inclusions of simplicial sets. Better yet, since \(J\) is defined to be the \(0\)-coskeleton of \(\{0,1\}\) it will suffice to check that \(\{0,1\}\) has the right lifting property with respect to all inclusions of sets. Indeed, every set has this property.
\end{proof}

\begin{cor}
	The class of fibrant objects in the Cisinski model structure contains all \(\infty\)-bicategories and is contained in the class of all weak \(\infty\)-bicategories. It thus coincides with the class of \(\infty\)-bicategories by Theorem~\ref{t:fibrant}.
\end{cor}
\begin{proof}
	Since \(\bS\) generates the class of scaled anodyne maps it follows that every Cisinski-fibrant object is a weak \(\infty\)-bicategory. On the other hand, combining Proposition~\ref{interval J} with the fact that Lurie's model structure on scaled simplicial sets is cartesian (Remark~\ref{r:cartesian})  
we may conclude that every generating anodyne in the Cisinski model structure is a trivial cofibration in the bicategorical model structure, and hence every \(\infty\)-bicategory is Cisinski-fibrant. 
	
\end{proof}
Since model structures are determined by the class of cofibrations and fibrant objects (\cite{JoyalQCatsApplications}, Proposition E.1.10) we may conclude that:
\begin{cor}
	The Cisinski model structure on the category \(\Ss\) of scaled simplicial sets with generating set \(\bS\) and interval given by \(J_{\sharp}\) coincides with Lurie's model structure for \(\infty\)-bicategories.
\end{cor}

\section{The main equivalence}\label{sec:equivalence}

In this section we define an adjunction of the form:
\[ \xymatrixcolsep{1pc}
\vcenter{\hbox{\xymatrix{
			**[l]\Ss \xtwocell[r]{}_{U}^{\iota
			}{'\perp}& **[r] \St_{2}}}}\]
which we show to be a Quillen one. After having established this, we prove it is a Quillen equivalence by making use, among other things, of an explicit fibrant replacement for \(\iota X\), when \(X\) is an \(\infty\)\nbd-bi\-cat\-egory.
\begin{define}
\label{d: iota}
	Let \(U\colon \St \to \Ss\) be the functor \(U(X,M_X) = (X,M_X \cap X_2)\) which forgets the marking in dimension \(\neq 2\). We will call \(U(X,M)\) the \emph{underlying scaled simplicial set} of the stratified set \((X,M)\).
In the other direction, we consider the functor \(\iota\colon \Ss\rightarrow \St\) which sends a scaled simplicial set \((X,T_X)\) to the stratified set \(\iota(X,T_X):=(X,T_X \cup \deg(X))\),
	where \(\deg(X)\) stands for the set of degenerate simplices of \(X\).
	By definition, this is a stratified set whose only non-degenerate marked \(n\)-simplices have \(n=2\).
\end{define}

By construction we have an evident natural isomorphism 
\[(X,T_X) \xrightarrow{\cong} U(X,T_X \cup \deg(X)) = U(\iota(X,T_X)),\] and direct inspection shows that this transformation exhibits \(\iota\) as left adjoint to \(U\).

\begin{lem}
	\label{basic obs}
	The functor \(\iota\colon \Ss \rightarrow \St_{2}\) is fully faithful, preserves monomorphisms and preserves pushouts. In addition, the natural map 
\[ \iota(X\times Y)\rightarrow \iota X \times   \iota Y \]
is an anodyne extension in \(\St_2\), and similarly the natural map
\[\iota\left(X'\times Y \cup_{X \times Y} X\times Y'\right) \to \iota X' \times \iota Y \cup_{\iota X \times \iota Y} \iota X \times \iota Y'\] is a weak equivalence in \(\St_{2}\) for every given pair of inclusions \(X\hookrightarrow X', \ Y \hookrightarrow Y'\).
\end{lem}

Informally, the previous lemma states that \(\iota\) preserves products and pushout-products ``up to homotopy''.

\begin{proof}
	We first note that monomorphisms of stratified sets are detected on the level of the underlying simplicial sets, and hence also on the level of the underlying scaled simplicial sets.
	The first part is then a consequence of the fact that \(\iota\) is a left adjoint functor and the associated unit map \(X \to U(\iota(X))\) is an isomorphism. Since \(U\) preserves products this also shows that
	\(\iota(X\times Y)\) and \(\iota X \times 	\iota Y\) have 
	the same underlying scaled simplicial set, so that 
	the marking of \(\iota(X\times Y)\) and \(\iota X \times 	\iota Y\) only differs in dimension strictly greater than 2. The comparison map (both for the case of products and pushout-products) is therefore in the weak saturation of the set \(\{(\Delta^n,\emptyset)\rightarrow (\Delta^n,\{\Del^n\})\}_{n >2}\), and is hence an anodyne extension.
\end{proof}

The previous lemma tells us that \(\iota\) preserves  cofibrations, since they coincide with monomorphisms for both model structures involved. 
\begin{define}
Given a stratified set \(X\in \St\) denote by \(eq_X\) the set of equivalences of \(X\), \ie 1-simplices \(v\colon x\rightarrow y\) that are equivalences in the underlying scaled simplicial set \(UX\) in the sense of Definition \ref{d: equivalences}.
\end{define}

By definition, the set of equivalences of a stratified set only depends on the underlying scaled simplicial set.

Before delving into the proof of Proposition \ref{Quillen adj}, we need a preliminary lemma. 
\begin{lem}
	\label{fib repl}
	Given an \(\infty\)-bicategory \((X,T_X)\) in \(\Ss\) we have that 
	\[
	\widetilde{X}:=\thb_2 (X,T_X\cup eq_X)
	\]
	is a \(2\)-trivial saturated complicial set and the map
	\(\iota(X,T_X) \to \widetilde{X}\) is a trivial cofibration of stratified sets. In particular, \(\widetilde{X}\) is a fibrant replacement
	of~\(\iota(X,T_X)\).
\end{lem}

\begin{proof}	
	We begin by showing that the inclusion \(\iota(X,T_X)\rightarrow \widetilde{X}\) is an anodyne extension. We can factor it as the composite of the obvious maps 
	\[\iota(X,T_X)\rightarrow \thb_2\iota(X,T_X)\rightarrow \widetilde{X} ,\] 
	the first one clearly being an anodyne morphism. Turning to the second one, we will show that it belongs to the weak saturation of the set of maps given by \(\{\Delta^3_{\eq}\rightarrow{\thr(\Delta^3)}\}\). For every equivalence \(v \in X_1\), extend the defining data of Remark \ref{rem:eq} to an inner horn \(\Lam^3_2 \to X\) of the form:
	\begin{center}
		\begin{tikzpicture}
		\squares{
			/squares/label/.cd,
			0=$0$,1=$1$,2=$2$,3=$3$,
			01=$v$, 12=$w$, 23=$v$, 03=$v$,
			012={$\simeq$}, 023={$=$}, 123={$\simeq$},
			0123={$\hat{\alpha}_v$},
			/squares/arrowstyle/.cd,
			02={equal}, 13={equal},
			012={phantom, description}, 023={phantom, description},
			013={RightarrowDashed}, 123={phantom, description},
			0123={tripleDashed},
			/squares/labelstyle/.cd,
			012={anchor=center}, 123={anchor=center}
		}
		\end{tikzpicture}
	\end{center}
	where the \(2\)-simplices denoted by ``\(\simeq\)'' are those provided by Remark \ref{rem:eq}. These data admit an extension to a \(3\)-simplex \(\hat{\alpha}_v\) by 2-triviality of \(\widetilde{X}\) and the fact that \((X,T_X)\) is an \(\infty\)-bicategory, and the resulting \(2\)-simplex \(\hat{\alpha}_v|_{\Del^{\{0,1,3\}}}\) has to be marked, thanks to Remark~\ref{rmk:j's are anod}.
	We thus get a map \(\alpha_v^{\sharp}\colon\Delta^3_{\eq}\rightarrow \thb_2\iota(X,T_X)\). We now observe that the inclusion \(\thb_2\iota(X,T_X)\rightarrow \tilde{X}\) is obtained as a transfinite composite of pushouts of the maps \(\alpha_v^{\sharp}\) along the map  \(\Delta^3_{\eq}\rightarrow{\thr(\Delta^3)}\), as \(v\) varies through all the equivalences of \((X,T_X)\), and is thus an anodyne map. Note that with this procedure we also marked \(w\), and that this is correct since by construction \(w\) is also invertible. 
	
	Let us now show that \(\widetilde{X}\) is fibrant, i.e., admits all lifts against generating anodyne morphisms (Definition~\ref{d:complicial-anodyne}). Since \(\widetilde{X}\) is \(2\)-trivial by construction and is marked by its equivalences,~\cite[Theorem 56]{VerityWeakComplicialI} tells us that it is enough to check that \(X\) admits lifts against inner complicial horn inclusions, inner thinness extensions of dimension \(\geq 3\) and saturation maps. Now for complicial inner horns \((\Lambda^n,M_i|_{\Lam^n_i})\rightarrow (\Delta^n,M_i)\), these lifts already exist at the level of the underlying scaled simplicial set \((X,T_X)=U(\widetilde{X})\) since it is assumed to be an \(\infty\)\nbd-bi\-cat\-e\-gory.
	
	Turning to inner thinness extensions in dimension \(\geq 3\), we only have to check it for \((\Delta^3,M_i') \rightarrow (\Delta^3,M_i'') \) since \(\widetilde{X}\) is \(2\)-trivial.
	In this case, the existence of a lift follows from Remark~\ref{rmk:j's are anod},
	which ensures it already exists at the level of the underlying scaled simplicial set, as above.
	
	Let us now show \(\widetilde{X}\) is saturated. 
	Since all the simplices of dimension \(\geq 3\) in \(\widetilde{X}\) are marked it will suffice to show that \(\widetilde{X}\) admits extensions against the maps \(\Del^3_{\eq} \to \thr(\Del^3)\) and \(\Del^3_{\eq} \ast \Del^0 \to \thr(\Del^3) \ast \Del^0\);
	indeed, for every \(m>0\), a non-degenerate marked \(2\)-simplex of $\Del^3_{\eq}\ast \Del^m$
	(resp.~\(\thr(\Del^3) \ast \Del^m\)) lives inside a stratified sub-simplicial set
	of the form \(\Del^3_{\eq}\ast \Del^{\{k\}}\) (resp.~\(\thr(\Del^3) \ast \Del^{\{k\}}\)), for some \(k \in \{0, 1, \dots, m\}\). 
	In the case of the inclusion \(\Del^3_{\eq} \to \thr(\Del^3)\), we observe that, since both  \(\Delta^3_{\eq}\) and \(\thr(\Delta^3)\) are 1-trivial by definition, any map \(\Delta^3_{\eq}\rightarrow \widetilde{X}\) factors through the core \(\infty\)-category of \(\widetilde{X}\). This is defined, in analogy with the case of scaled simplicial sets, to be the stratified sets whose underlying simplicial set is obtained by keeping only those \(n\)-simplices whose 2-dimensional faces are thin, and whose stratification is obtained by restriction. Therefore, a lift exists by the 2-out-of-6 property of equivalences in an \(\infty\)-category, which, in turn, follows from the analogous property for isomorphisms in the associated homotopy category.	
	For \(\Del^3_{\eq} \ast \Del^0 \to \thr(\Del^3) \ast \Del^0\), consider the following diagram of stratified sets (omitting the functor \(\iota\) for sake of simplicity):
	\[	\begin{tikzcd}
	(\Delta^4,T) \ar[r] \ar[d] & (\Delta^4,T') \ar[d]\ar[ddr, dotted, bend left=20]& \\
	\Delta^3_{\eq} \ast \Delta^0 \ar[drr,"f", bend right=15] \ar[r] & \thr(\Delta^3) \ast \Delta^0 & \\
	& & \widetilde{X}
	\end{tikzcd}\]
	where
	\[T := \{\Delta^{\{0,2,4\}}, \ \Delta^{\{1,2,3\}}, \ \Delta^{\{0,1,3\}}, \ \Delta^{\{1,3,4\}}, \ \Delta^{\{0,1,2\}}\},\]
	 
	\[T' := T\cup\{\Delta^{\{0,3,4\}}, \Delta^{\{0,1,4\}}\}\] and the top horizontal map belongs to the set \ref{item:scaled_anodyne_ii} of Definition \ref{anodyne defi}.
	The precomposition of \(f\) with \((\Delta^4,T)\rightarrow \Delta^3_{\eq}\ast \Delta^0\) admits the dotted extension to \((\Delta^4,T')\),  
	since \((X,T_X)\) is an \(\infty\)-bicategory and \((\Delta^4,T)\rightarrow (\Delta^4,T')\) is a generating anodyne map for the model structure for \(\infty\)-bicategories. We hence obtain an extension of \(f\) to a map \(g\colon W \to \widetilde{X}\), where \(W := (\Del^4,T')\plus{(\Del^4,T)}\Delta^3_{\eq} \ast \Delta^0\). Since \(\widetilde{X}\) admits extensions against \(\Del^3_{\eq} \subseteq \thr(\Del^3)\) as we saw above we may further extend \(g\) to a map \(g'\colon W' \to \widetilde{X}\), where \(W' := W \coprod_{\Del^3_{\eq}} \thr(\Del^3)\).
	
	We now claim that the map \(g'\) extends to all of \(\thr(\Del^3) \ast \Del^0\). To see this, note that \(W'\) and \(\thr(\Del^3) \ast \Del^0\) have the same underlying simplicial set \(\Del^4\) and the same marked edges, and that all the triangles contained in \(\Del^{\{0,1,2,3\}}\) are marked in \(W'\). Out of the six triangles in \(\Del^4\) which contain the vertex \(4\) we have that exactly four are marked in \(W'\): \(\Del^{\{0,2,4\}},\Del^{\{1,3,4\}},\Del^{\{0,1,4\}}\) and \(\Del^{\{0,3,4\}}\), whereas in \(\thr(\Del^3) \ast \Del^0\) all six are marked.
	Now since the marked edges in \(\widetilde{X}\) are exactly those which are equivalences in the \(\infty\)-bicategory \((X,T_X)\) we get that \(g'\) sends the edges \(\Del^{\{0,1\}}\) and \(\Del^{\{1,2\}}\) to equivalences. Applying Proposition~\ref{p:joyal-2} to the \(3\)-simplex \(g'|_{\Del^{\{0,1,2,4\}}}\) we may now conclude that \(g'(\Del^{\{1,2,4\}})\) is in \(T_X\), 
	and the same proposition applied to the \(3\)-simplex \(g'|_{\Del^{\{1,2,3,4\}}}\)
	shows that \(g'(\Del^{\{2,3,4\}})\) is in \(T_X\). Since the markings of \(W'\) and \(\thr(\Del^3) \ast \Del^0\) coincide in dimension greater than 2, we conclude that \(g'\) extends to \(\thr(\Del^3) \ast \Del^0\), as desired.
\end{proof}

\begin{prop}
	\label{Quillen adj}
	The adjunction
	\[ \xymatrixcolsep{1pc}
	\vcenter{\hbox{\xymatrix{
				**[l]\Ss \xtwocell[r]{}_{U}^{\iota
				}{'\perp}& **[r] \St_{2}}}}\] is a Quillen adjunction between the model structure for \(\infty\)-bicategories and that of 2-trivial saturated complicial sets.
\end{prop}

We start with the following simple fact.
\begin{lem}
\label{l:inner anod}
  The functor \(\iota\) sends maps in the set \ref{item:scaled_anodyne_i} to complicial trivial cofibrations.
\end{lem}
\begin{proof}
  Given \(0<i<n\), consider the following commutative square, where the horizontal maps are natural inclusions and the right-hand side vertical map is the complicial (inner) horn that features in Definition \ref{d:complicial-anodyne}:
  \[\begin{tikzcd}
   \iota(\Lam^n_i,\Delta^{\{i-1,i,i+1\}}) \ar[d] \ar[r] & (\Lam^n_i,M_i|_{\Lam^n_i}) \ar[d] \\
    \iota(\Del^n_i,\Delta^{\{i-1,i,i+1\}}) \ar[r] & (\Del^n,M_i)
  \end{tikzcd}\]
The horizontal maps are 2-trivializing morphisms in the sense of Definition \ref{d:complicial-anodyne}, so they are complicial trivial cofibrations. Hence, we conclude the proof thanks to the two-out-of-three property of weak equivalences.
\end{proof}

The next two lemmas are the main ingredients in the proof of Proposition \ref{Quillen adj}.

\begin{lem}\label{l:outer-anodyne}
For \(n \geq 3\) the map \[i_n\colon \bigl(\Lam^n_0,\{\Del^{\{0,1\}},\Del^{\{0,1,n\}}\}\bigr) \to \bigl(\Del^n, \{\Del^{\{0,1\}},\Del^{\{0,1,n\}}\}\bigr)\] is a trivial cofibration in \(\St_2\).
\end{lem}
\begin{proof}
It suffices to show that \(i_n\) has the left lifting property with respect to all fibrations between fibrant objects \(p\colon (X,M_X) \to (Y,M_Y)\). We therefore consider a lifting problem of the form
\begin{equation}\label{e:lift}
\begin{tikzcd}
(\Lam^n_0,\{\Del^{\{0,1\}},\Del^{\{0,1,n\}}\}) \ar[r, "f"] \ar[d] & (X,M_X) \ar[d, "p"] \\
(\Del^n, \{\Del^{\{0,1\}},\Del^{\{0,1,n\}}\}) \ar[r, "f"] \ar[ur, dotted] & (Y,M_Y)
\end{tikzcd}
\end{equation}
Since \(X\) is fibrant its underlying scaled simplicial set \((X,M_X \cap X_2)\) admits fillers for scaled inner horns by Lemma~\ref{l:inner anod}.
If we let \(\rho\colon\Del^n \to \Del^2\) be the surjective map defined by \(\rho^{-1}(0)=\{0\}, \rho^{-1}(1)=\{1,\dots,n-1\}\) and \(\rho^{-1}(2)=\{n\}\), and we define \(\sig\colon \Del^2 \to \Del^n\) to be the associated minimal section, we can apply Lemma~\ref{l:moving-lemma} to the inclusion \(A=\emptyset \hookrightarrow \Lam^n_0 = B\) (which consists of \((\sig \dashv \rho)\)-admissible simplicial sets, see Examples~\ref{ex:admissible}). Hence, we find a \((\sig \dashv \rho)\)-transformation \(h\colon \Del^1 \times \Lam^n_0 \to X\) with respect to the scaling \(M_X \cap X_2\), which restricts to \(f\) on \(\{1\}\times\Lam^n_0\). 

Let \(M\) be the set of simplices of \(\Del^1 \times \Lam^n_0\) containing the degenerate simplices and the simplices of the following form:
\begin{itemize}
\item
the simplices \(\Del^{\{1\}} \times \Del^{\{0,1\}}\) and \(\Del^{\{1\}} \times \Del^{\{0,1,n\}}\);
\item
the edges of the form \(\Del^1 \times \{i\}\) for \(i = 0,1,n\); 
\item
the triangles of the form \(\Del^{\{(0,i),(1,i),(1,j)\}}\) for every \(i,j \in [n]\);
\item
the triangles of the form \(\Del^{\{(0,i),(0,j),(1,j)\}}\) whenever \(j=0,1,n\) (\ie~\(j\) belongs to the image of \(\sigma\)) or \(j=i+1\) and \(i,j \in \{1,\dots,n-1\}\) (this second case is equivalent to having \(j = i+1\) and \(\rho(i) = \rho(j)\));
\item
the triangles \(\Del^{\{0\}} \times \Del^{\{0,1,i\}}\) for \(i \in \{2,\dots,n-1\}\);
\item
all the simplices of dimension \(\geq 3\);
\end{itemize}
By the definition of \((\sig \dashv \rho)\)-transformation, Remark \ref{rem:trans_deg_edges} and the comments in Notation \ref{n:L-A}, since \((X,M_X)\) is \(2\)-trivial we have that \(h\) extends to a map of stratified sets \((\Del^1 \times \Lam^n_0,M) \to (X,M_X)\). In particular, the second to last subset of \(M\) from the previous list is sent to thin triangles since, by point (2) in Definition~\ref{d:transformation} and by downward induction on \(i\) (see Remark \ref{r:degenerate-moving}), we have that the whole simplex \(h(\Delta^{\{0\}}\times \Delta^{\{0,1,\ldots n-1\}})\) degenerates to \(h(\Delta^{\{0\}}\times \Delta^{\{0,1\}})\). Let \(M'\) be the union of \(M\) with the simplices \(\Del^{\{0\}} \times \Del^{\{0,1\}}\) and \(\Del^{\{0\}} \times \Del^{\{0,1,n\}}\). By Remark~\ref{r:closure-nat} applied to the restrictions of \(h\) to \(\Del^1 \times\Del^{\{0,1\}}\) and \(\Del^1 \times\Del^{\{0,1,n\}}\), we see that \(h\) extends to a map of stratified sets \(h \colon (\Del^1 \times \Lam^n_0,M') \to (X,M_X)\).

Now consider the composed map \(ph\colon (\Del^1 \times \Lam^n_0,M') \to (Y,M_Y)\). Since the scaled simplicial set \((Y,M_Y \cap Y_2)\) also admits fillers for scaled inner horns by Lemma~\ref{l:inner anod} we may apply Lemma~\ref{l:moving-lemma} again, this time to the inclusion \(A=\Lam^n_0 \hookrightarrow \Del^n = B\). In this way, we can extend the \((\sig \dashv \rho)\)-transformation \(ph\) and the initial datum of \(g\) to a \((\sig \dashv \rho)\)-transformation \(H\colon (\Del^1 \times \Del^n,M') \to (Y,M_Y)\). 
Restricting \(h\) and \(H\) to \(\Del^{\{0\}} \subseteq \Del^1\), we now obtain a modified lifting problem of the form
\begin{equation}\label{e:modified-lift}
\xymatrix{
(\Lam^n_0,M'_0|_{\Lam^n_0}) \ar[r]^-{f'}\ar[d] & (X,M_X) \ar[d]^{p} \\
(\Del^n, M'_0) \ar[r]^-{g'}\ar@{.>}[ur]^{h'} & (Y,M_Y) \\
}
\end{equation}
where \(M'_0 := M'|_{\Del^{\{0\}} \times \Del^n}\) coincides with the set of all faces of \(\Del^n\) which contain the edge \(\Del^{\{0,1\}}\). Since the left vertical map in~\eqref{e:modified-lift} is an anodyne extension and the right vertical map is a fibration the dotted lift \(h'\) exists. The maps \(h',h\) and \(H\) now determine a commutative diagram of the form
\begin{equation}\label{e:modified-lift-2}
\xymatrix{
(\Del^{\{0\}} \times \Del^n,M'_0) \displaystyle\mathop{\coprod}_{\Del^{\{0\}} \times \Lam^n_0} (\Del^1 \times \Lam^n_0,M') \ar[r]\ar[d] & (X,M_X) \ar[d]^{p} \\
(\Del^1 \times \Del^n,M') \ar[r]^-{H}\ar@{.>}[ur] & (Y,M_Y) \\
}
\end{equation}
We now claim that the left vertical map in~\eqref{e:modified-lift-2} is an anodyne extension, and hence the lifting problem has a solution. To see this, first we note that \(\Del^n\) contains two non-degenerate faces that are not in \(\Lam^n_0\), the simplex \(\Del^{\{0,\dots,n\}}\) and the simplex \(\Del^{\{1,\dots,n\}}\). Furthermore, by the choice of \(\sig \dashv \rho\) the restriction of \(H\colon \Del^1 \times \Del^n \to Y\) to \(\Del^1 \times \Del^{\{1,\dots,n\}}\) is a \((\sig' \dashv \rho')\)-transformation, where \(\rho' \colon \Del^{\{1,\dots,n\}} \to \Del^{\{1,2\}}\) and \(\sig'\colon \Del^{\{1,2\}} \to \Del^{\{1,\dots,n\}}\) are the restrictions of \(\rho\) and \(\sig\), respectively. We can then factor the left vertical map as a composition of two maps, each of which is a pushout of a map of the form
\begin{equation}\label{e:step}
\big((\Del^{\{0\}} \times \Del^m\displaystyle\mathop{\coprod}_{\Del^{\{0\}} \times \partial \Del^m} \Del^1 \times \partial \Del^m,M''_0\big) \to (\Del^1 \times \Del^m,M'') 
\end{equation}
with \(m=n-1,n\), where the marking \(M''\) includes the edge \(\Del^1 \times \{0\}\), all the triangles of the form \(\Del^{\{(0,i),(1,i),(1,j)\}}\) and all the triangles of the form \(\Del^{\{(0,i),(0,i+1),(1,i+1)\}}\), and \(M''_0\) is the subset of those simplices which are contained in the domain of~\eqref{e:step}. This marking allows us to use the filtration of Remark~\ref{ob:filtration} (or, more precisely, its straightforward adaptation to stratified sets) and we can consequently factor~\eqref{e:step} into a sequence of (complicial) anodyne extensions. The dotted lift in~\eqref{e:modified-lift-2} then yields a lift in~\eqref{e:lift}, by restricting to \(\{1\}\times \Del^n\) as desired.
\end{proof}

\begin{lem}\label{l:saturation-anodyne}
Let \(j\colon(\Delta^4,T)\rightarrow (\Delta^4,T')\) be the generating scaled anodyne appearing second in the list in Definition~\ref{anodyne defi}. Then \(j\) is sent by \(\iota\colon \Ss \rightarrow \St_{2}\) to a trivial cofibration in \(\St_2\).
\end{lem}
\begin{proof}
It is enough to show that every fibration between fibrant objects satisfies the right lifting property against \(\iota(j)\). In fact, it is enough to check the lifting property against fibrant objects, since the underlying simplicial map of \(j\) is an isomorphism. Let us hence suppose that \(X\) is a fibrant stratified set with underlying scaled simplicial set \(X' := U(X)\) and let \(\sig\colon(\Del^4,T) \to X'\) be a map. Set \(x := \sig(0)\) and \(y := \sig(4)\). By Lemmas~\ref{l:inner anod} and~\ref{l:outer-anodyne}
we have that \(X'\) satisfies the extension property with respect to maps of type~\ref{item:scaled_anodyne_i} and~\ref{item:scaled_anodyne_iii} in Definition~\ref{anodyne defi}. It then follows from Corollary~\ref{c:slice} that the map of scaled simplicial sets \(p\colon \ovl{X'}_{/y} \to X'\) is an outer cartesian fibration and every edge which is marked in \(X'_{/y}\) is \(p\)-cartesian. Since \(\sig\) sends the triangles \(\Delta^{\{1,2,3\}},\Delta^{\{0,1,3\}}\) and \(\Delta^{\{0,1,2\}}\) to thin triangles and \(X\) admits fillers for thinness extensions it follows that \(\sig\) must send \(\Del^{\{0,2,3\}}\) to a thin triangle as well. We may therefore consider the \(4\)-simplex \(\sig\) as encoding a \(3\)-simplex \(\tau\colon\Del^3_{\sharp} \to \ovl{X'}_{/y}\) such that \(p\tau\colon\Del^3_{\sharp} \to X'\) coincides with \(\sig|_{\Del^{\{0,1,2,3\}}}\). Since \(\sig\) sends the triangles \(\Del^{\{0,2,4\}}\) and \(\Del^{\{1,3,4\}}\) to thin triangles in \(X'\) it follows that \(\tau\) sends the edges \(\Del^{\{0,2\}}\) and \(\Del^{\{1,3\}}\) to edges which are marked in \(X'_{/y}\). 

Let \(H= p\tau H'\colon \Del^1_{\flat} \times \Del^3_{\sharp} \to X'\) be the natural transformation from the constant map on \(x\) to \(p\tau\) induced by the unique natural transformation \(H'\colon\Del^1_{\flat} \times \Del^3_{\sharp} \to \Del^3_{\sharp}\) from the constant map on \(0\) to the identity. Now by Corollary~\ref{c:slice} the map $\ovl{X'}_{/y} \to X'$ is an outer fibration (and in particular thin triangle detecting) such that every edge in \(X'\) admits a marked \(p\)-cartesian lift in $\ovl{X'}_{/y}$. Applying Proposition~\ref{p:natural-lift} with respect to the inclusion \(A  = \emptyset \hookrightarrow \Del^3_{\sharp} = B\) we may consequently lift \(H\) to a natural transformation \(\wtl{H}\colon\Del^1_{\flat} \times \Del^3_{\sharp} \to \ovl{X'}_{/y}\) from some \(\tau'\) to \(\tau\), and such that $\wtl{H}(\Del^1_{\flat} \times \{i\})$ is a marked $p$-cartesian lift of $H(\Del^1_{\flat} \times \{i\})$ for every $i=0,...,3$. 
By Remark~\ref{r:closure} we have that the collection of marked edges in \(\ovl{X'}_{/y}\) satisfies a certain closure property: if \(\sig\colon\Del^2 \to \ovl{X'}_{/y}\) is a thin triangle such that \(\sig|_{\Del^{\{1,2\}}}\) is marked in \(X'_{/y}\) then \(\sig|_{\Del^{\{0,1\}}}\) is marked if and only if \(\sig|_{\Del^{\{0,2\}}}\) is marked. It then follows that \(\tau'\colon\Del^3_{\sharp} \to \ovl{X'}_{/y}\) also sends the edges \(\Del^{\{0,2\}}\) and \(\Del^{\{1,3\}}\) to marked edges in \(X'_{/y}\), since the homotopy \(\widetilde{H}\) provides invertible squares between \(\tau'_{\vert \Del^{\{0,2\}} }\) (respectively \( \tau'_{\vert \Del^{\{1,3\}} }\)) and \(\tau_{\vert \Del^{\{0,2\}} }\) (respectively \( \tau_{\vert \Del^{\{1,3\}} }\)), and the latter are marked by construction. In addition the image of \(\tau'\) lies by construction in the fiber \((\ovl{X}'_{/y})_x\) above \(x\). This means that \(\tau'\) corresponds to a \(4\)-simplex \(\sig'\colon (\Del^4,T) \to X'\) such that \(\sig'|_{\Del^{\{0,1,2,3\}}}\) degenerates to the point \(x\). 
It then follows that \(\sig'\) determines a map of \emph{stratified sets} of the form \(\sig''\colon \Del^3_{\eq} \ast \Del^0 \to X\). Since \(X\) is fibrant \(\sig''\) extends to a map \(\thr(\Del^3) \ast \Del^0 \to X\), which implies in particular that \(\sig'\) sends the triangles \(\Del^{\{0,1,4\}}\) and \(\Del^{\{0,3,4\}}\) to thin triangles in \(X'\), and so \(\tau'\colon\Del^3_{\sharp} \to \ovl{X'}_{/y}\) sends the edges \(\Del^{\{0,1\}}\) and \(\Del^{\{0,3\}}\) to marked edges in \(X'_{/y}\). By the closure property for marked edges invoked above we get that the same holds for \(\tau\). This, in turn, means that \(\sig\colon(\Del^4,T) \to X'\) also sends \(\Del^{\{0,1,4\}}\) and \(\Del^{\{0,3,4\}}\) to thin triangles and hence extends to a map \((\Del^4,T') \to X'\), yielding an extension \(\iota(\Del^3,T') \to X\), as desired.
\end{proof}

\begin{proof}[Proof of Proposition~\ref{Quillen adj}]
	We have already observed that \(\iota\) preserves cofibrations in Lemma \ref{basic obs}, so we are left with proving it preserves trivial cofibrations.
	By Proposition~\ref{prop:left-adj_is_quillen} 
	it is enough to check this for maps which are pushout-products of maps in \(\bS \cup \{i_k\colon \{k\}\hookrightarrow J_{\sharp}\}_{k=0,1}\) and cofibrations.
Better yet, by Proposition~\ref{basic obs} we have that \(\iota\) preserves pushout-products up to homotopy, 
and since \(\St_2\) is cartesian, it will just suffice to show that \(\iota\) sends \(\bS \cup \{i_k\colon \{k\}\hookrightarrow J_{\sharp}\}_{k=0,1}\) to trivial cofibrations. We now verify this claim case by case:
	\begin{itemize}
		\item for the map \((\Lambda^n_i,\{\Delta^{\{i-1,i,i+1\}}\}|_{\Lam^n_i})\rightarrow (\Delta^n,\{\Delta^{\{i-1,i,i+1\}}\})\) with \(0<i<n\), we note that this is precisely the content of Lemma \ref{l:inner anod};
		\item for the map \(j\colon(\Delta^4,T)\rightarrow (\Delta^4,T')\), this follows from Lemma~\ref{l:saturation-anodyne};
		\item for the map \((\Lambda^n_0\plus{\Delta^{\{0,1\}}}\Delta^0,\{\Delta^{\{0,1,n\}}\}|_{\Lam^n_0})\rightarrow (\Delta^n\plus{\Delta^{\{0,1\}}}\Delta^0,\{\Delta^{\{0,1,n\}}\})\) with \(n \geq 1\), when \(n=1,2\) this map becomes, after marking all simplices in dimension \(\geq 3\), a pushout of an outer complicial horn inclusion and is hence a trivial cofibration. When \(n \geq 3\) it is instead a pushout of the map appearing in Lemma~\ref{l:outer-anodyne}, and is hence again a trivial cofibration;
		\item		
		for the maps \(i_k\colon \{k\}\hookrightarrow J_{\sharp}\) with \(k=0,1\), Lemma \ref{fib repl} shows that a fibrant replacement of~\( \iota J_{\sharp}\) is given by the fully marked walking isomorphism \(E\). The desired result now follows from the fact that the inclusion \(\thr(\Del^1) \to E\) of the fully marked segment in \(E\) is complicial anodyne~\cite[Observation~43]{VerityWeakComplicialI}, and the inclusion \(\Del^{\{k\}} \hrar \thr(\Del^1)\) is complicial anodyne for each \(k=0,1\).
\qedhere
	\end{itemize}
\end{proof}

We now are in a position to prove the following main result.

\begin{thm}
	\label{Quillen eq}
	The adjunction 
	\[ \xymatrixcolsep{1pc}
	\vcenter{\hbox{\xymatrix{
				**[l]\Ss \xtwocell[r]{}_{U}^{\iota
				}{'\perp}& **[r] \St_{2}}}}\]
	is a Quillen equivalence.
\end{thm}

\begin{proof}
To show that the \emph{derived} unit is a weak equivalence it will suffice to check its components on a set of representatives of weak equivalence classes. By Lemma~\ref{fib repl}, these components can be represented by an isomorphism \((X,T_X) \to U\wtl{X}=(X,T_X)\) whenever \((X,T_X)\) is fibrant. We are hence left with showing that the derived counit is a weak equivalence.
	Given a fibrant stratified set \((Y,M)\), its counit \(\iota(Y,M \cap Y_2) \to (Y,M)\) must factor through the trivial cofibration \(f\colon\iota(Y,M \cap Y_2)	\to \wtl{Y}\) constructed in Lemma~\ref{fib repl}. It will hence suffice to show that the resulting map
	\[ \wtl{Y} \overset{g}{\longrightarrow} (Y,M) \] 
	is a weak equivalence in \(\St_2\). 
	We now observe that since the counit and \(f\) are both isomorphisms at the level of the underlying scaled simplicial sets, the same holds for \(g\). In addition, the stratified sets \(\wtl{Y}\) and \((Y,M)\) have the same marked simplices in dimension \(\geq 2\) by construction (and since \((Y,M)\) is fibrant and in particular \(2\)-trivial). We hence just need to verify that \(\tilde{Y}\) and \((Y,M)\) have the same marked \(1\)-simplices. In other words, we need to show that the marked edges in \((Y,M)\) are exactly the equivalences. Indeed, this holds for any fibrant stratified set: the marked edges are invertible since \((Y,M)\) has the right lifting property against the trivial cofibration \(\thr(\Del^1) \to \thr(J)\)~\cite[Observation~43]{VerityWeakComplicialI}, and the invertible edges are marked since \((Y,M)\) has the right lifting property against the saturation morphisms introduced in Definition \ref{d:complicial-anodyne}.
\end{proof}

The following corollary gives a pointwise characterization of natural equivalences, and we frame it here for future use:

\begin{cor}
	\label{lem: pointwise eq}
	Let \(X,Y\) be scaled simplicial sets and let \(h\colon \Del^1_{\flat} \times X\to Y\) be a pointwise invertible natural transformation. Then \(h_0\colon X \to Y\) is a bicategorical equivalence if and only if \(h_1\colon X \to Y\) is. Moreover, if \(X\) and \(Y\) are \(\infty\)-bicategories, then \(h\) extends to a map \(h\colon J_{\sharp}\times X \to Y\).
\end{cor}
\begin{proof}
	Let \(j\colon X \to \tilde{X}\) and \(\phi\colon Y \to \tilde{Y}\) be fibrant replacement maps consisting of anodyne morphisms. By solving the lifting problem depicted below, we get a map \(\tilde{h}\colon \Del^1_{\flat} \times \tilde{X} \to \tilde{Y}\) which is again pointwise invertible, since \(j\) and \(\phi\) do not alter the set of \(0\)-simplices being in the weakly saturated class generated by the morphisms \ref{item:scaled_anodyne_i}, \ref{item:scaled_anodyne_ii} and \ref{item:scaled_anodyne_iii}.
	\[\begin{tikzcd}
	\Del^1_{\flat} \times X \ar[r,"h"] \ar[d,"\Del^1\times j"{swap}] & Y\ar[r,"\phi"] & \tilde{Y}\\
	\Del^1_{\flat} \times \tilde{X} \ar[urr,dotted,"\tilde{h}"{swap}]
	\end{tikzcd}.\]
	Replacing \(X\) by \(\wtl{X}\) and \(Y\) by \(\wtl{Y}\) we may therefore assume without loss of generality that \(X\) and \(Y\) are \(\infty\)-bicategories. 
	
	Consider now the following diagram, where \(\iota\colon \Ss\to \St_2\) is the left Quillen equivalence of \S\ref{sec:equivalence} 
	and for an \(\infty\)-bicategory \(W\) we use \(\wtl{W}\) to denote the explicit fibrant replacement of \(\iota(W)\) of Lemma~\ref{fib repl}: 
	\[\begin{tikzcd}
	\iota(\Del^1_{\flat} \times X)\ar[d,"\simeq"] \ar[r,"\iota(h)"]& \iota(Y) \ar[r,"\simeq"]&\wtl{Y}\\
	\iota(\Del^1_{\flat}) \times \iota(X)  \ar[d,"\simeq"]\\
	\iota(\Del^1_{\flat}) \times \wtl{X} \ar[d]
	\ar[uurr, dotted, start anchor=north east, end anchor=south, "H"{description}]\\
	\thr(\Del^1) \times \wtl{X} \ar[d,"\simeq"]
	\ar[uuurr,dotted,start anchor=north east,  end anchor=south,"H'"{description, pos=.47}]\\
	\thr(J) \times \wtl{X}
	\ar[uuuurr,dotted,start anchor=north east,  end anchor=south, "H''"{description}]
	\end{tikzcd}\]
	A lift \(H\) exists since the composite of the top two left-hand side vertical maps is a trivial cofibration
by Lemma~\ref{fib repl}. The lift denoted by \(H'\) exists by our assumption on pointwise invertibility of \(h\). More precisely, note that the stratified inclusion \(\iota(\Del^1_{\flat}) \times \wtl{X} \hookrightarrow \thr(\Del^1) \times \wtl{X} \) is the identity on the underlying simplicial sets and on marked \(n\)-simplices for \(n>1\). Thus, the existence of \(H'\) follows by showing that given a map \(q\colon\iota(\Del^1_{\flat})\times \thr(\Del^1) \to W\) with \(W\) fibrant in \(\St_2\), if \(q|_{\Del^1 \times \{i\}}\) is marked in \(W\) for every \(i=0,1\), then \(q\) extends to a map from \(\thr(\Del^1)\times \thr(\Del^1)\), which is a straightforward verification. A lift \(H''\) exists since the inclusion \(\thr(\Del^1)\to \thr(J)\) is an anodyne map, as shown in \cite[Observation~43]{VerityWeakComplicialI}. Applying \(U\) we then get a map 
\[U(H'')\colon U(\thr(J) \times \wtl{X}) =  J_{\sharp}\times X \to Y = U(\wtl{Y}), \] 
whose restriction to \(\{0\} \times X\) and \(\{1\} \times X\) is given by \(h_0\) to \(h_1\) respectively. Since \(J_{\sharp}\) is a contractible \(\infty\)-bicategory it follows that \(h_0\) is a bicategorical equivalence if and only if \(h_1\) is.
\end{proof}

\section{The homotopy \pdftwo-category and the scaled \pdftwo-nerve}\label{sec:2-nerve}

In this section we construct a Quillen adjunction:
\[
\xymatrixcolsep{1pc}
\vcenter{\hbox{\xymatrix{
			**[l]\Ss \xtwocell[r]{}_{\mathcal{N}_2}^{\ho_2}{'\perp}& **[r] \nCat{2}}}}
\] and prove the right adjoint \(\mathcal{N}_2\) is homotopy fully faithful, \ie fully faithful in the \(\infty\)-categorical sense. We then prove that if \(\C\) is an \(\infty\)-bicategory then the unit map \(\C \to \mathcal{N}_2\ho_2(\C)\) is \(2\)-conservative in the sense that it detects thin triangles. 

Recall that in \cite{StreetOrientals}, Street defines the free \(\omega\)-category on the \(n\)-simplex, called the \(n\)-th \emph{oriental} and denoted by \(\mathcal{O}_{n}\), which extends to a cosimplicial \(\omega\)-category \(\mathcal{O}_{\bullet}\colon \Delta\to \nCat{\omega}\). We may then apply to it the ``intelligent'' truncation functor \(\tau_{\leq 2}\colon\nCat{\omega}\to \nCat{2}\) which sends an \(\omega\)-category \(X\) to the \(2\)-category having the same \(0\)-cells and \(1\)-cells, and whose \(2\)-cells are equivalence classes \([x]\) of \(2\)-cells in \(X\), where \([x]=[y]\) if there is a zig-zag of 3-cells connecting \(x\) and \(y\). The resulting \(2\)-categories then admit an explicit description (see, e.g., \cite[Corollary A.6]{AraMaltsiCondE}): the objects of \(\cO_n^{\leq 2} := \tau_{\leq 2}\cO_n\) are the elements of ordered set \([n]\), and given \(i,j \in [n]\), the category \(\Hom_{\tau_{\leq 2}\cO_n}(i,j)\) is the partially ordered set of subsets \(S \subseteq [n]\) such that \(\min(S)=i\) and \(\max(S)=j\), with composition given by union.

\begin{define}\label{d:scaled-2}
	We define the \emph{scaled \(2\)-nerve} \(\mathcal{N}_2(\D)\) of a \(2\)-category \(\D\) by the formula 
	\[\mathcal{N}_2(\D) := (\Hom_{\nCat{2}}(\cO_\bullet^{\leq 2},\D),T_{\D}) \in \Ss,\]
	where \(T_{\D}\) denotes the triangles corresponding to those maps \(\cO^{\leq 2}_2 \to \D\) which send the unique non-identity morphism in \(\Hom_{\cO^{\leq 2}_2}(0,2)\) to an isomorphism.
\end{define}

The cosimplicial object \(\cO^{\leq 2}_\bullet\) can be extended to the category \(\Del_{\sca}\) (see Remark~\ref{alternative def}) by sending \([2]_t\) to the \(2\)-category obtained from \(\cO^{\leq 2}_2\) by universally inverting its unique non-invertible \(2\)-cell. This results in a \(2\)-category with objects \(0,1,2\) and with the same mapping categories as \(\cO^{\leq 2}_2\) except that \(\Hom_{\cO^{\leq 2}_2}(0,2)\) is the ``walking isomorphism'', that is the trivial groupoid on two objects. By general considerations the functor \(\mathcal{N}_2\) then admits a left adjoint given by the restriction to scaled simplicial sets of the left Kan extension of \(\cO^{\leq 2}_\bullet\colon\Del_{\sca} \to \nCat{2}\) along the Yoneda embedding of \(\Del_{\sca}\).
We denote this left adjoint by \[\ho_2\colon\Ss \to \nCat{2}.\] 
We then have the following:
\begin{prop}
	\label{nerves coincide}
	The functor \(\ho_2\) is naturally isomorphic to the composite
	\[\Ss\overset{\fC^{\sca}}{\longrightarrow}\msCat\overset{\ho_*}{\longrightarrow}\nCat{2}.\] 
	In particular, \(\ho_2\) is a weak equivalences preserving left Quillen functor, being a composite of such (see Proposition~\ref{p:quillen-2}).\end{prop}

   By uniqueness of adjoints, it follows that the scaled \(2\)-nerve \(\mathcal{N}_2\) identifies with the composition \(\rN^{\sca} \circ \rN_*\).

\begin{proof}
	Both \(\ho_2\) and the composite \(\ho_*\fC^{\sca}\) are left Kan extensions of their restriction to \(\Del_{\sca}\), and so it will suffice to construct a natural isomorphism on their restriction to \(\Del_{\sca}\).
	Let us first consider the further restriction to the subcategory \(\Del \subseteq \Del_{\sca}\), so that we are dealing with the simplicial objects \(\cO^{\leq 2}_{\bullet}\) and \(\ho_*\Del^{\bullet}\). The isomorphism in this case follows from ~\cite[Corollary A.6]{AraMaltsiCondE} and~\cite[Remark 3.7.5]{LurieGoodwillie}.
	
	To extend this natural isomorphism to \(\Del_{\sca}\) we observe that \(\Del_{\sca}\) is obtained from \(\Del\) by freely adding the object \([2]_t\) and factorizing the all degeneracy maps from \([2]\) to \([0],[1]\) through \([2]_t\) in a compatible manner. 
	The desired extension of the natural isomorphism now follows from the fact that in both cases the arrow \([2] \to [2]_t\) is sent to the universal inversion of the unique non-invertible \(2\)-cell (this can be deduced, for instance, from the fact that \(L\ovl{\ho}(\Delta^1_{\sharp})\) is the walking isomorphism), while the \(2\)-categories associated to \([0]\) and \([1]\) have all their \(2\)-cells invertible.
\end{proof}

We now show the scaled \(2\)-nerve is homotopy fully faithful.

\begin{prop}
	\label{nerve is fully faithful}
	The counit \(\epsilon_{\mathcal{C}}\colon \ho_2\mathcal{N}_2 \mathcal{C} \rightarrow\mathcal{C}\) is an equivalence of \(2\)\nbd-cat\-e\-gories.
	More precisely, it is bijective on objects and an equivalence on hom-categories.
\end{prop}

\begin{proof}
	By Proposition~\ref{nerves coincide} the adjunction \(\ho_2 \dashv \N_2\) can be identified with the composition of the adjunctions \(\fC^{\sca} \dashv \rN^{\sca}\) and \(\ho_* \dashv \rN\). It follows that the counit map \(\epsilon_{\C}\) factors as a composition
	\begin{equation}\label{e:composition} 
	\ho_*\fC^{\sca}\rN^{\sca}\rN_*(\C) \to \ho_*\rN_*(\C) \to \C 
	\end{equation}
	where the first map is induced by the counit of the adjunction \(\fC^{\sca} \dashv \rN^{\sca}\) and the second is the counit of \(\ho_* \dashv \rN_*\). Since \(\fC^{\sca} \dashv \rN^{\sca}\) is a Quillen equivalence and \(\rN_*(\C)\) is fibrant (since all objects in \(\nCat{2}\) are fibrant) the first map in~\eqref{e:composition} is the image under \(\ho_*\) of a weak equivalence, and is hence a \(2\)-categorical equivalence by Proposition~\ref{p:quillen-2}. 
The second map in~\eqref{e:composition} is also a 2-categorical equivalence: Indeed, it is bijective on objects by definition, and at the level of hom-categories it is given by the counit of \(L\ovl{\ho} \dashv \rN^+\iota\), which is an isomorphism since both \(\iota\colon\Cat \to \RelCat\) and \(\rN^+\colon\RelCat \to \s^+\) are fully-faithful right adjoints.	
\end{proof}

We now address the unit map of \(\ho_2 \dashv \N_2\).
\begin{thm}\label{conservativity thm}
	Let \(\C\) be an \(\infty\)-bicategory. Then the unit map
	\[ \C \to \N_2\ho_2(\C) \]
	is \ndef{\(2\)-conservative} in the sense that a triangle in \(\C\) is thin if and only if its image in \(\N_2\ho_2(\C)\) is thin. 
\end{thm}

\begin{rem}
	Triangles in \(\N_2\ho_2(\C)\) correspond to \(2\)-functors \(\cO^{\leq 2}_2 \to \ho_2(\C)\), and hence to lax commutative triangles 
	\[ \xymatrix{
		& y \ar^{g}[dr] & \\
		x\ar[ur]^{f}\ar[rr]_{h} & \twocell{u}{\alp}{0.4}{0}{0.12} & z \\
	}\]
	in \(\ho_2\C\). By definition, such a triangle is thin in \(\N_2\ho_2(\C)\) if and only if the \(2\)-cell \(\alp\) is invertible. We then interpret Theorem~\ref{conservativity thm} as saying that the homotopy \(2\)-category \emph{detects invertibility of \(2\)-cells}. 
\end{rem}

Before we come to the proof of Theorem~\ref{conservativity thm} let us recall its \((\infty,1)\)-categorical analogue:

\begin{prop}\label{pre-conservativity thm}
	Let \(\C^{\natural} = (\C,\Eq(\C))\) be a fibrant marked simplicial set, that is an \(\infty\)-category \(\C\) marked by its equivalences. Then the unit map \(\C^{\natural} \to \rN^+\iota L\ovl{\ho}(\C^{\natural})\) detects marked edges. 
\end{prop}
\begin{proof}
	Since all marked edges are equivalences we have that 
	\(\him(\C^{\natural}) \subseteq \ho(\C)\) consists of isomorphisms and hence
	\(L\ovl{\ho}(\C^{\natural}) = L(\ho(\C),\him(\C^{\natural})) \cong \ho(\C)\). The desired claim is then equivalent to saying that an arrow in \(\C\) is an equivalence if and only if its corresponding arrow in \(\ho(\C)\) is an isomorphism. Indeed, this is simply the definition of equivalences. 
\end{proof}

\begin{proof}[Proof of Theorem~\ref{conservativity thm}]
	We first argue that in order to prove the claim for \(\C\) we may replace \(\C\) by any equivalent model. To see this, suppose that \(f\colon\C \to \D\) is a bicategorical equivalence of \(\infty\)-bicategories and consider the commutative square
	\[ \xymatrix{
		\C \ar[r]^-{\simeq}\ar[d] & \D \ar[d] \\
		\N_2\ho_2\C \ar[r]^-{\simeq} & \N_2\ho_2\D \ .
	}
	\]
	Since the left Quillen functor \(\ho_2\) preserves weak equivalences by Proposition~\ref{nerves coincide} and the right Quillen functor \(\N_2\) preserves weak equivalences since every \(2\)-category is fibrant we have that the bottom horizontal map is a weak equivalence. By Proposition~\ref{p:detects} it now follows that both horizontal maps detect thin triangles.
	The desired claim for \(\D\) hence implies the same for \(\C\). We may consequently assume without loss of generality that \(\C\) is of the form \(\rN^{\sca}(\E)\) for some fibrant and cofibrant marked-simplicial category \(\E\). 
	Now in this case the counit map
\[\epsilon_{\E} \colon \fC^{\sca}\rN^{\sca}(\E) \to \E \]
	is a weak equivalence of marked-simplicial categories. Consider the commutative diagram
	\[ \xymatrix{
		\rN^{\sca}(\E) \ar[d]\ar@{=}[dr] &  \\
		\rN^{\sca}\fC^{\sca}\rN^{\sca}(\E) \ar[r]\ar[d] & \rN^{\sca}(\E) \ar[d] \\
		\rN^{\sca}\rN_*\ho_*\fC^{\sca}\rN^{\sca}(\E) \ar[r] & \rN^{\sca}\rN_*\ho_*(\E) 
	}\]
	in which the top left vertical map is induced by the unit of \(\fC^{\sca} \dashv \rN^{\sca}\), the two horizontal maps by the counit of \(\fC^{\sca} \dashv \rN^{\sca}\), and the two bottom vertical maps by the unit of \(\ho_* \dashv \rN_*\). The diagonal equality is then a consequence of the triangle identities, and the composed vertical map is the one we wish to show detects thin triangles (see Proposition~\ref{nerves coincide}). Note that the bottom horizontal map is a weak equivalence, since \(\ho_*\) sends the counit equivalence between cofibrant objects to a weak equivalence of (fibrant) 2-categories, which is then preserved by the composite right Quillen functor \(\rN^{\sca}\rN_*\). We may hence show instead that the right vertical map detects thin triangles. But this map is obtained by applying \(\rN^{\sca}\) to the unit map \(\E \to \rN_*\ho_*\E\). By the definition of thin triangles in scaled nerves of marked simplicial categories, it will now suffice to show that for every \(x,y \in \E\) the map \(\Map_{\E}(x,y) \to \Map_{\rN_*\ho_*\E}(x,y) = \rN^+\iota L\ovl{\ho}\Map_{\E}(x,y)\) detects marked edges. Indeed, since \(\E\) is fibrant each \(\Map_{\E}(x,y)\) is fibrant as a marked simplicial set and so the desired result follows from Proposition~\ref{pre-conservativity thm}. 
\end{proof}

\appendix

\section{Recollections on Cisinski's and Olschok's theory}

Let \(\Kk\) be a locally presentable category. Recall that an arrow \(f\colon X \to Y\) in \(\Kk\) is called a \ndef{monomorphism} if \(\Hom_{\Kk}(Z,X) \to \Hom_{\Kk}(Z,Y)\) is injective for every \(Z \in \Kk\). We will fix the following standing assumption:

\begin{assume}\label{a:standing}
	The class of monomorphisms in \(\Kk\) is weakly saturated, and is generated as a weakly saturated class by a \emph{set} \(\mathfrak{M}\) of monomorphisms. In addition, the map \(\emptyset \to X\) from the initial object to any other object is a monomorphism.
\end{assume}

\begin{rem}\label{r:reflective}
	Assumption~\ref{a:standing} holds for every presheaf category and more generally any topos. Furthermore, if Assumption~\ref{a:standing} holds for \(\Kk\) and \(\Aa \subseteq \Kk\) is a reflective subcategory such that the reflector \(r\colon \Kk \to \Aa\) preserves monomorphisms then Assumption~\ref{a:standing} also holds for \(\Aa\). 
\end{rem}

\begin{notate}[Grothendieck]
  We call \emph{trivial fibration} any arrow of~\(\Kk\)
which has the right lifting property with respect to
all monomorphisms of \(\Kk\). A class of arrows \(\W\) is a \emph{localiser} if it
satisfies the following conditions:
\begin{enumerate}
  \item the class \(\W\) has the \(2\)-out-of-\(3\) property;
  \item the class \(\W\) contains all trivial fibrations;
  \item the class of monomorphisms which are also elements of \(\W\)
  is closed under pushout and transfinite compositions.
\end{enumerate}
We shall always assume that a localizer \(\W\) is closed under retracts
(cf.~\cite[\S4]{OlschokLeft}).
\end{notate}

In this appendix, we shall review some basic elements
of the theory developed by Cisinski~\cite{CisinskiPrefaisceaux} (see also~\cite[\S2.4]{CisinskiHigherCats})
and Olschok~\cite{OlschokLeft}
which studies the conditions under which, given a small set \(S\) of monomorphisms
of \(\Kk\), there is a model category structure on~\(\Kk\) where the cofibrations
are the monomorphisms and the class \(\W(S)\) of weak equivalences
is the smallest localiser containing the set \(S\). 
We note that the theory is developed more generally in~\cite{OlschokLeft}, starting from a suitable factorization systems on \(\Kk\). The case we consider here then corresponds to taking the factorization system of monomorphisms and trivial fibrations. 
In general, we tend here to follow more closely the notation and terminology of~\cite{CisinskiPrefaisceaux}.
Another short summary of these results, but closer in form to~\cite{OlschokLeft},
is provided in~\cite[\S2]{GaucherLeftDetermined}.

\begin{notate}
  Let \(\Kk\) be a locally presentable category. For \(f \colon x \to y\) and \(g \colon z \to w\)
  two arrows of \(\Kk\), we shall denote by \(f \hat{\times} g \colon (x \times w) \coprod_{x \times z} (y \times z) \to y \times w\)
  the pushout-product of \(f\) and \(g\).
\end{notate}

\begin{define}\label{d:cylinder}
	Let \(\Kk\) be a locally presentable category with terminal object~\(e\).
	A \emph{cylinder} of \(\Kk\) is a pair \(\mathcal{I} =(I, \partial)\)
	where \(I\) is an object of \(\Kk\) and \(\partial \colon e \coprod e \to I\)
	is a monomorphism in \(\Kk\); for \(\eps=0, 1\) we denote by
	\(\partial^{\eps}\colon e \to I\) the composition of \(\partial\) with each of the two canonical maps \(e \to e \coprod e\) (which are always monos by our assumption~\ref{a:standing}).
	We say that a cylinder \(\I\) is \ndef{cartesian} if the following two conditions hold:
	\begin{enumerate}
		\item the functor \(I \times - \colon \Kk \to \Kk\) admits a
		right adjoint \(\Km(I, -) \colon \Kk \to \Kk\) (which is then called the \emph{path functor});
		\item for every monomorphism \(i\) of \(\Kk\),
		the arrows \(\partial^\eps \hat{\times} i\), \(\eps=0, 1\),
		and \(\partial \hat{\times} i\)
		are monomorphisms.
	\end{enumerate}
\end{define}

\begin{rem}\label{r:easy-cylinder}
If \(\Kk\) is cartesian closed and the collection of monomorphisms is closed under pushout-products then any cylinder in \(\Kk\) is cartesian.
\end{rem}

\begin{rem}
The notion of cylinder in Definition~\ref{d:cylinder} is a special case of the more general notion of functorial cylinder, see \cite[Definition~2.9]{OlschokLeft}, where the functor corresponding to a cylinder \(\I = (I,\partial)\) is given by \(CX := I \times X\) (cf.\ \cite[Exemple~1.3.8]{CisinskiPrefaisceaux}).  
\end{rem}

\begin{notate}\label{nota:anodyne}
	Let \(\Kk\) be a locally presentable category, \(S\) a
	set of monomorphisms of \(\Kk\) and
	\(\mathcal I=(I, \partial)\) a cylinder.
	For any set \(T\) of arrows of \(\Kk\) we denote by \(\mathbf{\Lambda}(T)\)
	the set of arrows of the form \(\partial \hat{\times} f\), for \(f\) in \(T\). We then define  
	\begin{description}
		\item[$(0)$] \(\Lambda_{\mathcal{I}}^0(S) = S \cup 	\{\partial^\eps \hat{\times} i : \eps=0, 1\text{ and }i \in \mathfrak{M}\}
\),
		\item[$(i)$] \(\Lambda_{\mathcal{I}}^{i+1}(S) = \mathbf{\Lambda}(\Lambda_{\mathcal I}^i)\), \(i >0\),
		\item[$(\infty)$] \(\Lambda_{\mathcal{I}}(S) = \bigcup_{i \geq 0} \Lambda_{\mathcal I}^i\)
	\end{description}
	We shall say that \(\Lambda_{\mathcal{I}}(S)\) is a set of
	\emph{generating \((\mathcal{I}, S)\)-anodyne maps},
	or simply \emph{generating anodyne maps}, and the smallest
	weakly saturated class of \(\Kk\) containing \(\Lambda_{\mathcal{I}}(S)\)
	will be denoted by \(\An_{\mathcal{I}}(S)\), or simply \(\An\),
	and called the class of \emph{\((\mathcal{I}, S)\)-anodyne maps},
	or simply \emph{anodyne maps}.
\end{notate}

\begin{thm}[\cite{CisinskiPrefaisceaux}, \cite{OlschokLeft}]\label{thm:cisinski_model}
	Let \(\Kk\) be a locally presentable category satisfying Assumption~\ref{a:standing},
	\(\mathcal I=(I, \partial)\) a cartesian cylinder
	and \(S\) a subset of monomorphisms of \(\Kk\). 
	Then there exists a model category structure on \(\Kk\)
	having the monomorphisms as cofibrations and such that
	the fibrant objects are precisely the objects of \(\Kk\)
	having the right lifting property with respect
	to the set \(\Lambda_{\mathcal{I}}(S)\).
	Moreover, the class of weak equivalences \(\W(S)\) is the smallest
	localiser of \(\Kk\) containing \(\Lambda_{\mathcal{I}}(S)\).
\end{thm}
\begin{proof}
When \(\Kk\) is a presheaf category this is Cisinski's original result proven in~\cite[Theorem~1.3.22]{CisinskiPrefaisceaux}. In general
the statement is a particular case of~\cite[Theorem~3.16]{OlschokLeft};
see also~\cite[Theorem~2.5 and Proposition~2.6]{GaucherLeftDetermined}.
\end{proof}

\begin{rem}\label{rmk:fibrations-between-fibrants}
	It follows from~\cite[Corollary~3.29]{OlschokLeft} (see also Proposition~1.3.36 of~\cite{CisinskiPrefaisceaux}) that the fibrations between fibrant objects of the model category of Theorem~\ref{thm:cisinski_model} are precisely the maps having the right lifting property with respect to the set~ \(\Lambda_{\mathcal{I}}(S)\). 
\end{rem}

\begin{rem}\label{rmk:weak_equivalences}
	The weak equivalences of the model category structure
	of Theorem~\ref{thm:cisinski_model} are the morphisms \(p \colon X \to Y\) of \(\Kk\)
	such that for every fibrant object \(Z\) we have that the
	induced function \(p^* \colon \Kk(Y, Z) \to \Kk(X, Z)\)
	becomes a bijection when mod out by the usual relation
	of \(\mathcal{I}\)\nbd-ho\-mot\-o\-py given by the interval \(I\times -\);
  see~\cite[Définition~1.3.21]{CisinskiPrefaisceaux} and \cite[Definition~3.12]{OlschokLeft}.
\end{rem}

\begin{prop}\label{prop:left-adj_is_quillen}
	Let \(\Kk\) be a locally presentable category,
	\(\mathcal I\) a cartesian cylinder, \(S\) a subset of monomorphisms of \(\Kk\).
	Consider a model category \(\mathcal{M}\)
	and a functor \(F \colon \Kk \to \mathcal{M}\) preserving small
	colimits and mapping monomorphisms of \(\Kk\) to cofibrations of \(\mathcal{M}\).
	Then \(F\) is a left Quillen functor if and only if
	it maps the maps in \(\Lam_{\I}(S)\) to trivial cofibrations in \(\M\).
\end{prop}

\begin{proof}
	In the case where \(\Kk\) is a presheaf category this is Proposition~2.4.40 of~\cite{CisinskiHigherCats}, but the argument works in general: the existence of a right adjoint \(G\colon \M \to \Kk\) follows from the adjoint functor theorem (this does not require \(\M\) to be locally presentable, see~\cite[Remark 5.5.2.10]{HTT}). By virtue of~\cite[Proposition E.2.14]{JoyalQCatsApplications} it now suffices to show that \(G\) preserves fibrations between fibrant objects. By Remark~\ref{rmk:fibrations-between-fibrants} this is equivalent to \(F\) sending \(\Lam_{\I}(S)\) to trivial cofibrations. 
\end{proof}

\bibliographystyle{amsplain}
\bibliography{biblio}

\providecommand{\bysame}{\leavevmode\hbox to3em{\hrulefill}\thinspace}
\providecommand{\MR}{\relax\ifhmode\unskip\space\fi MR }
\providecommand{\MRhref}[2]{%
  \href{http://www.ams.org/mathscinet-getitem?mr=#1}{#2}
}
\providecommand{\href}[2]{#2}
\begin{thebibliography}{10}

\bibitem{AraQCatvsRezk}
Dimitri Ara, \emph{Higher quasi-categories vs higher {R}ezk spaces}, J.
  K-Theory \textbf{14} (2014), no.~3, 701--749.

\bibitem{AraMaltsiCondE}
Dimitri Ara and Georges Maltsiniotis, \emph{The homotopy type of the
  $\infty$-category associated to a simplicial complex},
  \href{http://arxiv.org/abs/1503.02720}{arxiv:1503.02720}, 2015, preprint.

\bibitem{BarwickSchommerPriesUnicity}
Clark Barwick and Christopher Schommer-Pries, \emph{On the unicity of the
  theory of higher categories}, J. Amer. Math. Soc. \textbf{34} (2021), no.~4,
  1011--1058.

\bibitem{BergerMoerdijkHomotopyEnriched}
Clemens {Berger} and Ieke {Moerdijk}, \emph{{On the homotopy theory of enriched
  categories}}, {Q. J. Math.} \textbf{64} (2013), no.~3, 805--846.

\bibitem{BergnerRezkInftynI}
Julia~E. Bergner and Charles Rezk, \emph{Comparison of models for
  {$(\infty,n)$}-categories, {I}}, Geom. Topol. \textbf{17} (2013), no.~4,
  2163--2202.

\bibitem{BergnerRezkInftynII}
\bysame, \emph{Comparison of models for {$(\infty, n)$}-categories, {II}},
  J.~Topol. \textbf{13} (2020), no.~4, 1554--1581.

\bibitem{BoardmanVogt}
John~M. Boardman and Rainer~M. Vogt, \emph{Homotopy invariant algebraic
  structures on topological spaces}, Lecture Notes in Mathematic, vol. 347,
  Springer-Verlag, Berlin-New York, 1973.

\bibitem{CampbellCellular}
Alexander {Campbell}, \emph{{A homotopy coherent cellular nerve for
  bicategories}}, {Adv. Math.} \textbf{368} (2020), 66.

\bibitem{cubical}
Tim Campion, Chris Kapulkin, and Yuki Maehara, \emph{Comical sets: a cubical
  model for $(\infty, n)$-categories}, arXiv preprint arXiv:2005.07603 (2020).

\bibitem{CisinskiPrefaisceaux}
Denis-Charles Cisinski, \emph{Les préfaisceaux comme modèles des types
  d'homotopie}, Astérisque, Soc.~Math.~Français, 2006.

\bibitem{CisinskiHigherCats}
\bysame, \emph{Higher categories and homotopical algebra}, Cambridge Studies in
  Advanced Mathematics, vol. 180, Cambridge University Press, Cambridge, 2019.

\bibitem{DKM-equivalence}
Brandon Doherty, Chris Kapulkin, and Yuki Maehara, \emph{Equivalence of cubical
  and simplicial approaches to $(\infty, n)$-categories}, arXiv preprint
  arXiv:2106.09428 (2021).

\bibitem{GagnaHarpazLanariLaxLimits}
Andrea Gagna, Yonatan Harpaz, and Edoardo Lanari, \emph{{F}ibrations and lax
  limits of {$(\infty,2)$}-categories},
  \href{http://arxiv.org/abs/2012.04537}{2012.04537}, 2020, preprint.

\bibitem{GagnaHarpazLanari2Fib}
\bysame, \emph{Cartesian fibrations of $(\infty,2)$-categories},
  \href{http://arxiv.org/abs/2107.12356}{arxiv:2107.12356}, 2021, preprint.

\bibitem{GagnaHarpazLanariGrayLaxFunctors}
\bysame, \emph{Gray tensor products and {L}ax functors of
  {$(\infty,2)$}-categories}, Adv. Math. \textbf{391} (2021), Paper No. 107986,
  32.

\bibitem{GaucherLeftDetermined}
Philippe Gaucher, \emph{Left determined model categories}, New York J. Math.
  \textbf{21} (2015), 1093--1115.

\bibitem{GindiRigidification}
Harry Gindi, \emph{Coherent nerve for higher quasicategories}, Theory Appl.
  Categ. \textbf{37} (2021), no.~23, 709--817.

\bibitem{JoyalQCatsAndKan}
Andr{\'{e}} Joyal, \emph{Quasi-categories and {K}an complexes}, J. Pure Appl.
  Algebra \textbf{175} (2002), no.~1-3, 207--222, Special volume celebrating
  the 70th birthday of Professor Max Kelly.

\bibitem{JoyalQCatsApplications}
\bysame, \emph{The theory of quasi-categories and its applications}, CRM
  Quaderns \textbf{45 (II)} (2008), 147--496.

\bibitem{JoyalTierneyQcatVSSegal}
Andr\'{e} Joyal and Myles Tierney, \emph{Quasi-categories vs {S}egal spaces},
  Categories in algebra, geometry and mathematical physics, Contemp. Math.,
  vol. 431, Amer. Math. Soc., Providence, RI, 2007, pp.~277--326.

\bibitem{LackModel2Cat}
Stephen Lack, \emph{A {Q}uillen model structure for {$2$}-categories},
  {$K$}-{T}heory \textbf{26} (2002), no.~2, 171--205.

\bibitem{LackModelBicat}
\bysame, \emph{A {Q}uillen model structure for bicategories}, $K$-Theory
  \textbf{33} (2004), no.~3, 185--197.

\bibitem{LurieGoodwillie}
Jacob Lurie, \emph{{\((\infty, 2)\)}-categories and the {G}oodwillie {C}alculus
  {I}}, \href{http://arxiv.org/abs/0905.0462}{arxiv:0905.0462}, 2009, preprint.

\bibitem{HTT}
\bysame, \emph{Higher {T}opos {T}heory}, Annals of Mathematics Studies, vol.
  170, Princeton University Press, Princeton, NJ, 2009.

\bibitem{LurieKerodon}
\bysame, \emph{Kerodon}, \url{https://kerodon.net}, 2018.

\bibitem{OlschokLeft}
Marc Olschok, \emph{Left determined model structures for locally presentable
  categories}, Appl. Categ. Structures \textbf{19} (2011), no.~6, 901--938.

\bibitem{OzornovaRovelliNComplicial}
Victoriya Ozornova and Martina Rovelli, \emph{{Model structures for \((
  \infty,n)\)-categories on (pre)stratified simplicial sets and prestratified
  simplicial spaces}}, {Algebr.~Geom.~Topol.} \textbf{20} (2020), no.~3,
  1543--1600.

\bibitem{OzornovaRovelli2Complicial}
Viktoriya Ozornova and Martina Rovelli, \emph{Nerves of 2-categories and
  2-categorification of {$(\infty,2)$}-categories}, Adv. Math. \textbf{391}
  (2021), Paper No. 107948, 39.

\bibitem{OzornovaRovelliVerityGray}
Viktoriya Ozornova, Martina Rovelli, and Dominic Verity, \emph{{G}ray tensor
  product and saturated $n$-complicial sets},
  \href{http://arxiv.org/abs/2007.01235}{arxiv:2007.01235}, 2020, preprint.

\bibitem{RiehlOuverture}
Emily Riehl, \emph{Complicial sets, an overture}, 2016 {MATRIX} annals, MATRIX
  Book Ser., vol.~1, Springer, Cham, 2018, pp.~49--76. \MR{3792516}

\bibitem{StreetOrientals}
Ross Street, \emph{The algebra of oriented simplexes}, Journal of Pure and
  Applied Algebra \textbf{49} (1987), no.~3, 283--335.

\bibitem{VerityComplicial}
Dominic Verity, \emph{Complicial sets characterising the simplicial nerves of
  strict {$\omega$}-categories}, Mem.~Amer.~Math.~Soc. \textbf{193} (2008),
  no.~905, xvi+184.

\bibitem{VerityWeakComplicialI}
\bysame, \emph{Weak complicial sets. {I}. {B}asic homotopy theory}, Adv.~Math.
  \textbf{219} (2008), no.~4, 1081--1149.

\end{thebibliography}
\end{document}